\documentclass[12pt]{amsart} 
\usepackage[utf8]{inputenc} 
\usepackage[T1]{fontenc}    
\usepackage[colorlinks=true]{hyperref}
 \hypersetup{urlcolor=blue, citecolor=blue} 
\usepackage{url}            
\usepackage{booktabs}       
\usepackage{amsfonts}       
\usepackage{nicefrac}       
\usepackage{graphicx}
\usepackage{subcaption}
\usepackage{fullpage}
\usepackage{dsfont}
\usepackage{soul}
\usepackage[normalem]{ulem}
\usepackage{cancel}
\usepackage{mathrsfs}
\usepackage{footmisc}

\usepackage{mathtools,amsmath,amsfonts,amssymb,amsthm,subcaption}
\usepackage{centernot}
\usepackage{bbm}
\usepackage{xcolor}
\usepackage{enumitem}
\usepackage{algorithm}
\usepackage{algpseudocode}

\theoremstyle{plain}\newtheorem{proposition}{Proposition}[section]
\theoremstyle{plain}\newtheorem{theorem}{Theorem}[section]
\theoremstyle{plain}\newtheorem{lemma}{Lemma}[section]
\theoremstyle{plain}

\theoremstyle{definition}
\theoremstyle{definition}\newtheorem{remark}{Remark}[section]
\theoremstyle{definition}\newtheorem{assumption}{Assumption}[section]

\floatname{algorithm}{Algorithm}
\algblock{Input}{EndInput}
\algnotext{EndInput}
\algblock{Output}{EndOutput}
\algnotext{EndOutput}



\renewcommand{\geq}{\geqslant}

\renewcommand{\leq}{\leqslant}

\newcommand{\ds}{\displaystyle} 
\newcommand{\be}{\begin{equation}}
\newcommand{\ee}{\end{equation}}

\newcommand{\f}{F}
\newcommand{\fbar}{\overline{\f}}
\newcommand{\fbarN}{\fbar^N}
\newcommand{\fbarNt}{\fbarN_t}
\newcommand{\tildefbarNt}{\widetilde{\f}^N_t}
\newcommand{\fbarmu}{\overline{\mathscr{\f}}^\mu}      
\newcommand{\fbarnu}{\overline{\mathscr{\f}}^\nu} 
\newcommand{\fbarnut}{\fbarnu(t)} 
\newcommand{\fbarmut}{\fbarmu_t}  
\newcommand{\muN}{\mu^N_t}
\newcommand{\tildemuN}{\widetilde{\mu}^N}
\newcommand{\bx}{\mathbf x}
\newcommand{\nd}{\textnormal{d}}
\newcommand{\dt}{\nd t}
\newcommand{\bxi}{\bx^j}
\newcommand{\bxit}{\bxi_t}
\newcommand{\bxiit}{\bx^i_t}
\newcommand{\tildebxit}{\widetilde{\bx}^j_t}
\newcommand{\tildebxiit}{\widetilde{\bx}^i_t}
\newcommand{\mi}{m^j}
\newcommand{\mitt}{\mi_t}
\newcommand{\miit}{m^i_t}
\newcommand{\tildemit}{\widetilde{m}^j_t}
\newcommand{\tildemiit}{\widetilde{m}^i_t}
\newcommand{\dmit}{\nd\mitt}
\newcommand{\dx}{\nd \bx}
\newcommand{\dxit}{\dx^j_t}
\newcommand{\diff}{\sigma}
\newcommand{\diffit}{\diff^j_t}
\newcommand{\dW}{\nd W}
\newcommand{\Wit}{W^j_t}
\newcommand{\dWit}{\dW^j_t}
\newcommand{\ddt}{\frac{\nd}{\dt}}
\newcommand{\Nhood}{{\mathcal N}}

\DeclareMathOperator*{\argmin}{\textnormal{argmin}}
\newcommand{\vel}{{\sf M}}

\newcommand{\salpha}{\lambda} 
\newcommand{\SGD}{optimization with adjusted annealing rate} 
\newcommand{\SGDc}{optimization with adjusted annealing rate, }
\newcommand{\gammab}{\sigma} 
\newcommand{\betab}{\alpha} 
\newcommand{\alphalam}{\lambda}  
\newcommand{\betamu}{\beta}  

\newcommand{\nut}{\nu(t)}   
\newcommand{\nutx}{\nu(t,\bx)}   
\newcommand{\nuNtx}{\nu^N(t,\bx)}   

\newcommand{\mut}{\mu_t}   
\newcommand{\muat}[1]{\mu_{#1}}   
\newcommand{\muxmat}[1]{\mu_{#1}(\bx,m)}
\newcommand{\mutxm}{\mu_t(\bx,m)} 
\newcommand{\muNtxm}{\mu^N_t(\bx,m)} 
\newcommand{\dmutxm}{\hspace*{0.02cm}\nd\hspace*{-0.03cm}\mutxm} 
\newcommand{\dmuNtxm}{\hspace*{0.02cm}\nd\hspace*{-0.03cm}\muNtxm} 
\newcommand{\dmutm}{\hspace*{0.02cm}\nd\hspace*{-0.03cm}\mu_t(\cdot,m)} 
\newcommand{\dmuxmat}[1]{\hspace*{0.02cm}\nd\hspace*{-0.03cm}\mu_{#1}(\bx,m)}

\newcommand{\fbarmuat}[1]{\fbarmu_{#1}}  
\newcommand{\tildemuNt}{\tildemuN_t} 

\usepackage{tikz-cd}

\numberwithin{equation}{section}

\title[Swarm-based optimization\\meets simulated annealing]{Swarm-based gradient descent\\meets simulated annealing}

\author{Zhiyan Ding}
\address{Department of Mathematics, University of California, Berkeley}
\email{zding.m@berkeley.edu}

\author{Martin Guerra}
\address{Department of Mathematics, University of Wisconsin-Madison}
\email{mguerra4@wisc.edu}

\author{Qin Li}
\address{Department of Mathematics, University of Wisconsin-Madison}
\email{qinli@math.wisc.edu}

\author{Eitan Tadmor}
\address{Department of Mathematics and Institute for Physical Science \& Technology\newline \hspace*{0.3cm}   University of Maryland, College Park}
\email{tadmor@umd.edu}

\date{\today}



\thanks{\textbf{Acknowledgments.} Research  was supported by NSF grant DMS-2012292 (MG), by NSF grant DMS-2308440 and ONR grant N00014-21-1-2140 (QL) and by ONR grants N00014-2112773, N00014-2412659 (ET).  ET  thanks the Fondations Sciences Math\'{e}matiques des Paris  (FSMP) and LJLL at the Sorbonne University for the support and hospitality.}

\begin{document}

\newpage
\maketitle
\begin{abstract}
We introduce a novel method, called Swarm-based Simulated Annealing (SSA), for non-convex optimization which is at the interface between the swarm-based gradient-descent (SBGD) \cite{eitan_2022,eitan_2023}, and Simulated Annealing (SA) \cite{vcerny1985thermodynamical,Kirkpatrick_1983,geman1986diffusions}. Similar to SBGD, we introduce a swarm of agents, each identified with a position, $\bx$ \emph{and} mass $m$, to explore the ambient space. Similar to SA, the agents proceed in the gradient descent direction, and are subject to Brownian motion. The annealing rate, however, is dictated by a decreasing function of their mass. As a consequence, instead of the SA protocol for time-decreasing temperature, here the swarm decides how to `cool down' agents, depending on their own accumulated mass. The dynamics of masses is coupled with the dynamics of positions: agents at higher ground transfer (part of) their mass to those at lower ground. Consequently, the resulting SSA optimizer is dynamically divided between heavier, cooler agents viewed as `leaders' and lighter, warmer agents viewed as `explorers'. Mean-field convergence analysis and benchmark optimizations demonstrate the effectiveness of the SSA method as a multi-dimensional global optimizer.
\end{abstract}
\setcounter{tocdepth}{1}
\tableofcontents

\section{Introduction}\label{sec:intro}
We introduce a new swarm-based  optimization method to compute the global minimum of  non-convex objective functions, $\ds \bx_*:=\argmin_{\Omega \subset \mathbb{R}^d} \f(\bx)$. The swarm consists of $N$ agents --- enumerated $j=1,2,\ldots, N$, each of which is identified with a time-dependent position, $\bxit \in {\mathbb R}^d$, and time-dependent mass, $\mitt \in {\mathbb R}_+$. The positions are governed by overdamped Langevin process\footnote{\label{foot:notation}
A word about notation. Throughout the paper the time-dependent quantities  in deterministic models are denoted $\square(t,\cdots)$ while time-dependent quantities which arise from stochastic  descriptions are denoted $\square_t(\cdots)$.},
\[
\dxit= -\nabla\f\left(\bxit\right)\dt +\sqrt{2\diffit}\dWit, \qquad j=1,2,\ldots, N,
\]
where $\left\{W^j_t\right\}^N_{j=1}$ are independent Brownian motions with    amplitude $\sqrt{2\diffit}$, where $\diffit$ is the scaled ``temperature'' (dating back to Langevin \cite{lemons1997paul}), or  is referred to as  the \emph{annealing-rate}, \cite{Kirkpatrick_1983,vcerny1985thermodynamical,gidas1985global,geman1986diffusions,CHS1987}. 
 The intricate aspect of the dynamics is a proper tuning of these annealing-rates  $\diffit$. This is where the masses, $\{\mitt\}$, come into play: they are driven by
\[ 
\dmit=-\mitt\left(\f\left(\bxit\right)-\fbarNt\right)\dt, \qquad j=1,2,\ldots,N,
\]
where $\fbarNt$ is the mass-weighted average of the $N$ agents, $\ds \fbarNt:=\frac{\sum_{j=1}^N\mitt\f(\bxit)}{\sum_{j=1}^N \mitt}$. 
We tune the annealing rates  as a decreasing function of the mass, $\diffit= \sigma(\mitt)$.
Thus, our protocol for annealing-rate prevents lighter agents from being trapped in basins of \emph{local} minima while `cools down' heavier agents at the basin of attraction of \emph{global} minimum of $\f$. This is where the Swarm-Based Gradient Descent (SBGD) \cite{eitan_2022} meets Simulated Annealing (SA), \cite{Kirkpatrick_1983,geman1984stochastic,vcerny1985thermodynamical}, except that the protocol for ``cooling process'' is different: instead of an explicit recipe for decreasing the temperature as is done in SA, our swarm-based approach lets different agents at a lower ground cool down, depending on their increasing mass $\diffit= \sigma(\mitt)$. 

\noindent
{\bf Why provisional minimum?} We refer to the mass-weighted average
\begin{equation}\label{eq:provisional_min}
\fbarNt:=\frac{\sum_{j=1}^N\mitt\f(\bxit)}{\sum_{j=1}^N \mitt}
\end{equation}
as the \emph{provisional minimum}. To clarify why provisional minimum is needed, we note  that according to the mass equation above, agents that are above the provisional minimum will shed a fraction of their mass --- masses that are transferred to agents that are below the provisional minimum. Since the total mass remains constant, $\sum_j \dmit/\dt=0$, masses tend to concentrate with agents near or below the provisional minimum, and it therefore makes sense to adjust the annealing rate, $\diffit$, as a \emph{decreasing} function of the mass,
\[
\diffit=\gammab(\mitt): (0,\infty) \mapsto  {\mathbb R}_+, \qquad j=1,2,\ldots, N.
\]
Given this adjustment of mass-dependent annealing, \emph{it is expected} that  for a large crowd of agents, $N\gg1$, \eqref{eq:provisional_min} will approach the global minimum,
\[
\lim_{N\rightarrow \infty}\fbarNt \stackrel{t\rightarrow \infty}{\longrightarrow} \f_*:=\f(\bx_*).
\]
Indeed, our main result states that, under the appropriate assumptions (in particular, the assumption of uniqueness of global minimizer made in Assumption~\ref{assumption: assmptn1} below), this is true.

The method described above lies at the interface of the deterministic SBGD approach and the stochastic-based SA. 
We recall that the SBGD is governed by a deterministic system which governs the swarm of $N$-agents, and for $j=1,2,\ldots, N$
\be\tag{SBGD}\label{eq:SBGD}
\ddt\bx^j(t)=-\betab(m^j(t))\nabla \f(\bx^j(t)),\quad
\ddt m^j(t) =-\big(\f(\bx^j(t))-\fbar^N(t)\big)m^j(t)\,.
\ee
The key feature  advocated in \eqref{eq:SBGD} is swarm dynamics of both positions and masses embedded in ${\mathbb R}^d\times {\mathbb R}_+$. The additional dimension of mass serves as a platform for \emph{communication among agents}, encoded in the provisional minimum, which in \cite{eitan_2022} was taken as the minimum of the crowd at the given time, $\fbar^N(t)=\min_j \f(\bx^j(t))$. 
The same swarm-based methodology is used in this work, except for two distinct features. First, the provisional minimum is given here by~\eqref{eq:provisional_min} which eventually is expected to approach the global minimum, $\min_j \f(\bx^j(t))$, used in \cite{eitan_2022}. The second and more essential distinction, is how  masses are being used to adjust the dynamics of positions along the gradient descent: in SBGD, it is the \emph{time step} which is adjusted as a decreasing function of the mass, $\betab=\betab(m^j(t))$. In our new proposed algorithm, masses are used to adjust the annealing rate, $\gammab=\gammab(\mitt)$, as a decreasing function of mass. This new algorithm is named Swarm-based Simulated Annealing (SSA) optimizer:
\be\tag{SSA}\label{eq:SSA}
\left\{\begin{split}
\dxit&=-\nabla \f(\bxit)\dt+\sqrt{2\gammab(\mitt)}\dWit,\\
\dmit& =-\big(\f(\bxit)-\fbarNt\big)\mitt\dt,
\end{split}\right.\qquad j=1,2,\ldots, N.
\ee
\begin{remark}[{\bf On the choice of provisional minimum}]  The SBGD method \eqref{eq:SBGD} employs the actual minimum, $\fbar(t)=\min_j \f(\bx^j(t))$ as its provisional minimum \cite{eitan_2022, eitan_2023}. This, however, is not amenable to the mean-field limit analysis that we will pursue in Section \ref{sec:mean_field_analysis}. The mean-field limit requires us to tour the whole landscape, ending with
\[
\lim_{N\rightarrow \infty} \fbar^N(t)=\lim_{N\rightarrow \infty} \min_{j=1,\ldots, N} \f(\bx^j(t)),
\]
which is already the global minimum we are looking for. Instead, the use of mass-weighted average  $\fbarNt$ in \eqref{eq:SSA} as a provisional minimum, admits a mean-field interpretation which is shown to be driven towards the same desired global minimum.   
 \end{remark}
 
\subsection{In-swarm communication combined with stochastic search}
Application of deterministic swarm-based methods reveals that such methods \emph{succeed in the accuracy sense but fail in the probability sense}, --- namely when these methods succeed, then they find the global optimizer, yet there is non-zero probability for the methods to fail to do so.
In contrast, stochastic methods \emph{ succeed in the probability sense but fail in the accuracy sense}, --- namely such methods can always find the global basin, but they do not necessarily return a true global optimizer.
In this context, our swarm-based \SGD\,  takes advantage of randomness in \eqref{eq:SSA}${}_1$ combined with communication-based swarming in \eqref{eq:SSA}${}_2$,  with the aim of succeeding in both the probability and the accuracy sense. This is achieved by a dynamic process that combines swarm communication and randomness, as described below.

The essential role of communication in SBGD was already emphasized in \cite{eitan_2022,eitan_2023}. This is amplified in the present context, upon setting $\gammab(\cdot)\equiv 0$ in \eqref{eq:SSA}${}_1$. One ends up with a crowd of $N$ independent agents driven by gradient descent, 
\[
\dot{\bx}^j(t)=-\nabla \f(\bx^j(t)), \qquad j=1,2,\ldots, N,.
\]
The success of such a non-communicating crowd in exploring the ambient landscape is \emph{significantly worse} than that of a communication-based crowd. In Section~\ref{sec:other_methods}, we show that the method fails in the probability sense.
 
Communication is encoded in the provisional minimum, where agents `communicate' their height relative to $\fbarNt$. In our swarm-based \SGD\, \eqref{eq:SSA}, $\fbarNt$ is taken as the weighted average height, weighted by the different masses. Consequently, there is a dynamic distinction between agents above the provisional minimum and agents below or at the level of the provisional minimum. Agents above the provisional minimum carry  lighter mass and therefore explore the ambient landscape with a relatively large Brownian motion; these are viewed as the \emph{explorers} of the crowd. The heavier agents --- those that are below or at the level of the provisional minimum, evolve in the gradient direction with a smaller Brownian motion and are viewed as \emph{leaders} of the crowd. Of course, once an exploring light agent `hits' a new lower ground at, say, $\bxit$ where $\f(\bxit) \ll -\fbarNt$, it is expected to accumulate  more mass from the higher agents, to cool down, and eventually to become a leader, driving the crowd to a new lower point.

We note that by fixing a constant annealing rate, \eqref{eq:SSA}${}_1$ with $ \gammab(\mitt)\equiv \sigma>0$, we also encounter a problem. In this case, we recover Langevin Monte Carlo  
\[
\dxit= -\nabla\f(\bxit)\dt +\sqrt{2\sigma}\dWit, \qquad j=1,2,\ldots, N,
\]
whose long time invariant measure suggests $\bx \propto e^{-\f(\bx)/\sigma}$. It admits the proper convergence in the zero-variance limit, $\bx \stackrel{\sigma\rightarrow 0}{\longrightarrow} \argmin_\bx \f(\bx)$, but, as we show in Section~\ref{sec:other_methods} below, the resulting method fails in the accuracy sense for any finite $\sigma$.
Thus, the dynamic process of relabeling explorers and leaders, coupled through a mass-dependent annealing rate, $\diffit=\gammab(\mitt)$, is at the heart of the matter.

In Section \ref{sec:proofs_main_results} we prove global convergence of~\eqref{eq:SSA} in both, the sense of probability and the sense of accuracy. This is achieved in a two-step argument: first, we conduct a large-crowd, mean-field analysis summarized in Theorem \ref{thm:mean_field_thm} below, which  allows us to pass from  finitely many Stochastic-ODEs governing the particle system \eqref{eq:SSA}, to the limiting dynamics for the probability measure, $\mutxm$, governed by the limiting PDE \eqref{eqs:mean_field}; and second, the large-time behavior of the latter, summarized in Theorem \ref{thm:long_time_thm}, implies that in this mean-field limit, the PDE in long time returns the the provisional minimum as the global minimum. Combining these results we conclude in Theorem \ref{thm:rigorous} the claimed convergence $\displaystyle \mathop{\liminf}_{t\rightarrow \infty}\lim_{N\rightarrow \infty}\fbarNt= \f_*$ with a precise convergence rate of polynomial type. Further, in Section \ref{sec:macro} we derive the corresponding macroscopic description of the system.

\subsection{The swarm-based \SGD} The exchange of mass encoded in  \eqref{eq:SSA} assigns smaller masses to agents with higher values, while the decreasing property of $\gammab(m)$ ensures that these agents experience larger random perturbations. As a result, lighter agents serve as explorers, while  heavier agents serve as leaders. As noted in \cite{eitan_2022}, upon normalization of the total mass, $\sum_j \mitt\equiv \sum_j m^j_0=1$, masses can be interpreted as `probabilities of finding a global minimum'; heavier agents have larger probability in doing so. The computation proceeds by discretization of \eqref{eq:SSA}  using Euler-Maruyama's formula with a proper time step $h>0$,
\begin{equation}
    \bx_{n+1}^{j} = \bx_{n}^{j} - h\nabla \f(\bx_{n}^{j}) + \sqrt{2h\gammab\left(m_{n}^{j}\right)}\boldsymbol{\xi}^{j}
\end{equation}
for $j=1,2,...,N$ and $\boldsymbol{\xi}^{j} \sim \mathcal{N}(0,\mathbf{I}_{d})$ i.i.d for all $j$. The detailed algorithm is summarized in Algorithm \ref{algo:SSA}. The mass is conserved throughout the algorithm. According to Line 6, the update gives: $
\sum_{j=1}^N\mi_{n+1}=\sum_{j=1}^N\mi_{n}-h \sum_{j=1}^N\mi_{n}\left(\f(\bxi_n)-\fbar_n\right)=\sum_{j=1}^N\mi_{n}$, where the second summation term becomes zero owing to the definition of $\fbar_n$ shown in Line 3 in the algorithm.

The performance of the algorithm is presented in Section~\ref{sec:numerical_exp}. This follows the analytical study in Sections \ref{sec:main} and \ref{sec:proofs_main_results} in which we state, and respectively, prove the main analytical results of the paper, regarding the large-crowd, large-time dynamics associated with \eqref{eq:SSA}. 

\begin{algorithm}[H]
    \caption{Swarm-based Simulated Annealing (SSA)}\label{algo:SSA}
    \begin{algorithmic}[1]
    \Input: Initial distribution $\mu_0(x,m)$, $h$ as step size, and $n_T$ as $\#$ of iterations
    \EndInput
    \State $\{(\bx_{0}^{j},m_0^j)\}_{j=1}^{N}$ are i.i.d. drawn from $\mu_{0}$  for $j=1,2,...,N$
    \hspace{0.in} 
        \State $\ds \fbar_{0} \gets \frac{\sum_{j=1}^{N}m_{0}^{j}\f(\bx_{0}^{j})}{\sum^N_{j=1}m^j_0}$
    \hspace{0.2in} \% $\ds \fbar_n\equiv \fbar_n^N=\frac{\sum^N_{j=1} m^j_n\f(\bx^j_n)}{\sum^N_{j=1}m^j_n}$  the provisional minimum at $t^n$
    \For {$n = 0,1,2,...,n_T-1$}
    \For {$j=1,2,...,N$}
     \State $m_{n+1}^{j} \gets m_{n}^{j} - m_{n}^{j}h\big(\f(\bx_{n}^{j}) - \fbar_{n}\big)$ for $j=1,2,...,N$
       \EndFor
       \State Generate $N$ samples i.i.d such that $\boldsymbol{\xi}^{j}\sim \mathcal{N}(0,\mathbf{I}_{d})$
    \State $\bx_{n+1}^{j} \gets  \bx_{n}^{j} - h\nabla \f(\bx_{n}^{j}) + \sqrt{2h\gammab(m_{n}^{j})}\boldsymbol{\xi}^{j}$  for $j=1,2,...,N$
    \State $\ds \fbar_{n+1} \gets \sum_{j=1}^{N}m_{n+1}^{j}\f(\bx_{n+1}^{j})/\sum_{j=1}^{N}m_{n+1}^{j}$
    \EndFor
    \State $j_{\text{opt}} \gets \argmin_j \f(\bx^j_{n+1})$
    \Output : $\bx_{n_T}^{j_{\text{opt}}}$ and $\fbar_{n_T}$
    \EndOutput
    \end{algorithmic}
\end{algorithm}

\subsection{Related work}\label{sec:rw}
There are several well-known non-convex optimization algorithms which combine swarm-based and/or stochastic effects. We mention biologically inspired methods of ant colony optimization \cite{mohan2012survey},
artificial bee colony optimization \cite{karaboga2014comprehensive},  and firefly
optimization \cite{yang2009firefly}, and physically inspired methods of wind-driven optimization, \cite{bayraktar2013wind}.
We also mention particle swarm optimization methods which explore the state space with randomized drifts toward the best global position, \cite{kennedy1995particle,poli2007particle} and its stochastic version in \cite{Grassi_2023}.

 The stochastic part of our method is motivated by simulated annealing driven by stochastic noise that is `cooled down' as time evolves, \cite{Henderson2003handbook,monmarche2018hypocoercivity}.
The key feature in one-particle SA dynamics is the protocol for `cooling down', i.e., setting a properly tuned decreasing-in-time annealing rate $ \gammab_t \propto c/\sqrt{\log{t}}$, \cite{gidas1985nonstationary,gidas1985global,geman1986diffusions,Laarhoven_1987_1,CHS1987}. This should be compared with our swarm-based approach, in which the annealing rate of each agent is dictated as a decreasing function of its mass\footnote{in fact -- we allow a larger class of mass scaling, covered  in Assumptions \ref{assumption: assmptn2} and  \ref{assumption: assmptn1}.}, $\gammab_t \mapsto \gammab(\mitt)$; essentially, we let the swarm decide how to `cool down' the agents as they accumulate more mass.

The idea of gaining from the stochastic interaction of more than just one agent can be found in the more recent work \cite{chen2019accelerating}. Here Langevin diffusion swaps between two agents  --- a ``global explorer” and a ``local explorer”, with two annealing rates which correspond to high and low temperatures. The dynamics takes place in ${\mathbb R}^d\times {\mathbb R}^d$. In our method, there is a crowd of $N$ agents which  evolve their mass-dependent annealing rates in $\big(\mathbb{R}^d\times \mathbb{R}_+\big)^{\otimes N}$.

 In the present context, it is instructive to calibrate our method with the Consensus Based Optimization (CBO), \cite{pinnau2017consensus,carrillo2018analytical, totzeck2021trends, Carrillo_2021,Carrillo_2023, Riedl_2022}. The CBO is governed by a swarm-based dynamics of SDEs
\be\label{eq:CBO}
\nd \bxit = -\lambda\left(\bxit - \overline{\f}^\salpha_t\right)\dt + \sigma\left|\bxit - \overline{\f}^\salpha_t\right|\dWit , \qquad
\overline{\f}^\salpha_t :=\sum_{j=1}^N \bxit \left(\frac{\exp(-\salpha \f(\bxit))}{\sum_{i=1}^N \exp(-\salpha \f(\bx^i_t))}\right),
\ee
which are steering towards the weighted minimum $\overline{\f}^\salpha_t$. Indeed,  the weighted min $\overline{\f}^\salpha_t$ can be viewed as the provisional minimum driving the CBO \eqref{eq:CBO}, in view of the Laplace principle, 
\[
\lim_{\salpha\rightarrow \infty}\left(-\frac{1}{\salpha}\log\left(\int \omega^\salpha_\f(\bx)\nd \rho\right)\right) = \min_{\textnormal{supp}(\rho)}\f, \qquad \omega^\salpha_\f(\bx):= \exp(-\salpha \f(\bx)).
\]
It is here that we observe the main difference between our swarm-based \SGD\, \eqref{eq:SSA} and the CBO \eqref{eq:CBO}: while the latter requires a user tuning of parameter $\salpha$  in $\overline{\f}^\salpha$,  the swarm-based framework of \cite{eitan_2022} provides us with an \emph{adaptive} mechanism which dynamically adjusts the provisional minimum based on the mass distribution in  different parts of the landscape, $\fbarNt$, and the corresponding mean-field $\fbarmut$, outlined in \eqref{eqs:mean_field} below.

\section{Statement of main results}\label{sec:main}
\subsection{From empirical distribution to mean-field}\label{sec:mean_field_analysis}
We define the \emph{empirical distribution} $\muN$, which records the \emph{ensemble} distribution of $\{\bxit,\mitt\}_{j=1}^N$ that obey \eqref{eq:SSA}, as a measure on the feature space $(\bx,m)$,
\begin{equation}\label{eqn:dist_stochastic}
    \muN(\bx,m) = \frac{1}{N}\sum_{j=1}^{N}\delta_{\bxit}(\bx)\otimes \delta_{\mitt}(m)\,.
\end{equation}
 Recall that the provisional minimum is given by the weighted average in \eqref{eq:provisional_min}, which is expressed as 
\begin{equation}\label{eqn:ave_stochastic}
    \fbarNt = \frac{\ds \sum_{j=1}^N \mitt \f(\bxit)}{\ds \sum_{j=1}^N \mitt} = \frac{\ds \iint m\f(\bx)\dmuNtxm}{\ds \iint m\dmuNtxm}.
\end{equation}
To formally derive the mean-field limit, we use a smooth test function $\phi(x,m)$ that is compactly supported on $\mathbb{R}^d\times \mathbb{R}$. By definition,
\[
 \iint\phi(\bx,m)\muN(\bx,m)\nd x\nd m = \frac{1}{N}\sum^N_{j=1}\phi(\bxit,\mitt)\,.
\]
Differentiating, recalling It\^o's formula, and using \eqref{eq:SSA} yield 
\begin{align*}
\frac{1}{N}\sum_{j=1}^{N}& \mathbb{E}\left[\nd\phi(\bxit,\mitt)\right] \\
 & = \frac{1}{N}\sum_{j=1}^N\mathbb{E}\left[\nabla_\bx\phi(\bxit,\mitt)\cdot\nd \bxit +\frac{1}{2}\langle D^2_\bx\phi , 2\gammab(\mitt)\rangle_{\text{Tr}}\dt + \partial_m\phi(\bxit,\mitt)\nd \mitt\right]\\
&  = \frac{1}{N}\sum_{j=1}^N \mathbb{E}\underbrace{\left[\nabla_\bx \phi \cdot \big(-\nabla \f(\bxit)\dt + \sqrt{2\gammab(\mitt)}\dWit\big)\right]}_{\text{Term I}}   + \frac{1}{N}\sum_{j=1}^N\mathbb{E}\underbrace{\left[\langle D^2_\bx\phi , \gammab(\mitt)\rangle_{\text{Tr}}\dt\right]}_{\text{Term II}}
\\
  & \quad +  \frac{1}{N}\sum_{j=1}^N \mathbb{E}\underbrace{\left[-\partial_{m}\phi(\bxit,\mitt)(\f(\bxit)-\fbarNt)\mitt\dt\right]}_{\text{Term III}},
\end{align*}
where $D^2_\bx\phi$ is the Hessian. 
The first term, $\mathrm{I}=\mathrm{I}_1+\mathrm{I}_2$, where 
\begin{equation*}
\begin{split}
\mathrm{I}_1 &:= \frac{1}{N}\sum_{j=1}^{N}\mathbb{E}\left[\nabla_\bx \phi(\bxit,\mitt) \cdot -\nabla \f(\bxit)\right] = -\iint \nabla_\bx\phi(\bx,m) \cdot \nabla \f(\bx)\muN(\bx,m)\dx \nd m\\
& \ = \iint \phi(\bx,m) \nabla_\bx\cdot (\muN\nabla \f(\bx))\dx \nd m,
\end{split}
\end{equation*}
and $\mathrm{I}_2:= \mathbb{E}[\dWit] = 0$. For the second term, we have
\begin{equation*}
\begin{split}
\mathrm{II} &:=\frac{1}{N}\sum_{j=1}^{N}\mathbb{E}[\langle D^2_\bx\phi(\bxit,\mitt) , \gammab(\mitt)\rangle_{\text{Tr}}\dt] = \iint  \Delta_\bx\phi(\bx,m)\gammab(m) \muN(\bx,m)\dx\nd m\\
 & \ = \iint \phi(\bx,m)\gammab(m)\Delta_\bx \muN(\bx,m)\dx\nd m,
\end{split}
\end{equation*}
and finally, the third term
\begin{align*}
\mathrm{III} &:=
    \frac{1}{N}\sum_{j=1}^{N}\mathbb{E}\left[-\partial_{m}\phi(\bxit\mitt)\mitt(\f(\bxit)-\fbarNt)\right]\\ 
    & = -\iint\partial_{m}\phi(\bx,m)(m(\f(\bx)-\fbarNt))\muN(\bx,m)\dx\nd m \\
    & = \iint \phi(\bx,m)\partial_{m}(m\muN(\bx,m))(\f(\bx)-\fbarNt)\dx\nd m\,.
\end{align*}
Adding all these terms together, we conclude that the dynamics of the empirical distribution, $\muN$, is governed by the Vlasov equation
\begin{equation}\label{eqn:mean_field_N}
    \partial_{t}\muN = \nabla_\bx\cdot (\muN\nabla\f) + \big(\f(\bx)-\fbarNt\big)\partial_{m}(m\muN) + \gammab(m)\Delta_\bx\muN\,.
\end{equation}

This formally provides the mean-field equation. In the mean-field limit, $\mu=\mu_t(\bx,m)$ is governed by:
\begin{subequations}\label{eqs:mean_field}
\begin{equation}\label{eqn:mean_field_pde}
    \partial_{t}\mu = \nabla_\bx\cdot (\mu\nabla\f) + \big(\f(\bx)-\fbarmut)\big)\partial_{m}(m\mu) + \gammab(m)\Delta_\bx\mu\,,
\end{equation}
driven by the corresponding  provisional minimum
\begin{equation}\label{eqn:bar_f_t_mu}
\fbarmut:=\frac{\mathbb{E}_{\mu}[m\f(\bx)]}{\mathbb{E}_{\mu}[m]}=\frac{\ds \iint m\f(\bx) \dmutxm}{\ds \iint m\dmutxm}\,.
\end{equation}
\end{subequations}
Observe that $\fbarmut$ denotes  the provisional minimum associated with mean-field $\mu$, which is to be distinguished from  the discrete provisional minimum in \eqref{eqn:ave_stochastic}, denoted $\fbarNt$.
 We note that, similar to the discrete mass conservation in \eqref{eq:SSA}, $\frac{1}{N}\sum_j \mitt\equiv 1$, the total mass in \eqref{eqn:mean_field_pde} is conserved and equals one 
 \[
\ddt\int m\dmutxm = 0,\quad \int m\dmutxm=1\,.
\]

Our first result quantifies the convergence of the empirical distribution, $\muN$, to its limiting mean-field $\mut$. To this end, we make the following assumption. \begin{assumption}\label{assumption: assmptn2}\mbox{}
       \begin{itemize}
        \item{\emph{Lipschitz bound}} --- \ \  $\f$, $\nabla \f$  and $\sqrt{\gammab(\cdot)}$ are all Lipschitz continuous functions with the Lipschitz constant denoted by $L$;
        \item{\emph{Boundedness}} --- \ \ There exists $D>0$ such that $|\f|+\sqrt{\gammab}\leq D$. Without loss of generality, we set $D\geq 1$ and can further assume $\f\geq 0$;
        \item{\emph{Bounded support uniformly in $m$}} --- there exists $M_{\max} <\infty$ such that\ \ 
\[
       \text{Supp}_\bx(\mu_{0}(\bx,m)) \subset [0, M_{\max}], \qquad \forall m\in (0,\infty).
 \]
    \end{itemize}
\end{assumption}

\noindent
The next theorem shows that the empirical distribution $\muN $ and the mean-field $\mu$ stay close. The proof is postponed to Section \ref{sec:mean_field_proof}.
\begin{theorem}[{\bf Mean-field limit}]\label{thm:mean_field_thm}
Assume that Assumption \ref{assumption: assmptn2} holds. Let $\mut=\mutxm$ be the  mean-field solution of  \eqref{eqn:mean_field_pde} and let $\ds \muN = \frac{1}{N}\sum_{j=1}^{N}\delta_{\bxit}(\bx)\otimes\delta_{\mitt}(m)$ be the empirical distribution associated with the ensemble of swarm-based solutions ~\eqref{eq:SSA} subject to compatible initial data so that $\{\bx^j_0\,,m^j_0\}$ are i.i.d. samples of $\muxmat{t=0}$. Then $\mut$ and $\muN$ are close in the Wasserstein sense:
\begin{equation}\label{eqn:close_in_mu}W_2\left(\mut,\muN\right)\rightarrow 0,\quad \text{in probability as $N\rightarrow\infty$}\,,
\end{equation}
and the corresponding provisional minimum \eqref{eqn:ave_stochastic}, $\fbarNt$, converges to $\fbarmut$ in \eqref{eqn:bar_f_t_mu}, with the law of large numbers rate: there exists a constant $C_t=C_t(\muat{0}, \f, d)>0$ independent of $N$ such that
        \begin{equation}\label{eqn:close_in_f}
            \mathbb{E}\left[\left|\fbarNt-\fbarmut\right|\right]<\frac{C_t}{\sqrt{N}}\,.
        \end{equation}
\end{theorem}
The bound \eqref{eqn:close_in_f} provides  the crucial convergence bound, which translates the provisional minimum of the many-agent system  to that of the limiting mean-field equation.

\subsection{Large-time convergence--- from mean-field to global minimum}\label{sec:thm_converge_mean_field}
Upon translating the convergence of $\fbarNt$ to that of $\fbarmut$, we now switch gear to study the convergence of the mean-field limiting PDE. To this end, we need another set of assumptions for $\f,\gammab$, and $\mu_0$:
\begin{assumption}\label{assumption: assmptn1}
We assume that the following conditions hold true:
\begin{itemize}
    \item    $\gammab(m)$ has a bounded support: there exists an $m_c>0$ such that $\gammab(m) = 0$ when $m > m_c$.
     \item
      $(m_c, \infty)\cap \mathrm{supp}_{m}(\muat{0})\neq\emptyset$
       and $\ds M(0)=\iint m\dmuxmat{0}=1$.
   \item {$\f$ admits a unique global minimum $\bx_*:=\argmin_{\bx\in \Omega} \f(\bx)$, and thus $\nabla\f(\bx_*)=0$}.
\end{itemize}
\end{assumption}
The next theorem shows that the provisional minimum of the mean-field, $\fbarmut$, is converging towards the global minimum of $\f$. The proof is postponed to Section \ref{sec:long_time_proof} below.
\begin{theorem}[Large time behavior]\label{thm:long_time_thm}
    Assume that Assumptions \ref{assumption: assmptn2}, \ref{assumption: assmptn1} hold. Let $\mu$ be the mean-field with provisional minimum $\fbarmut$. We have
    \begin{equation}\label{eqn:f_liminf}
    \liminf_{t\rightarrow\infty}\fbarmut=\min_\bx \f(\bx)=\f_*\,.
    \end{equation}
    Specifically, --- for any $\epsilon>0$, we let $\Nhood_\epsilon$ denote the following $\epsilon$ neighborhood:
     \begin{equation}\label{eqn:Nhood_eps}
    \Nhood_\epsilon:=\left\{\bx\,|\,\f(\bx)<\f_* +\epsilon\right\}\,.
      \end{equation}
     Then, there exists a time  $t_{\epsilon}(m_c)< \infty$ defined by\footnote{Note that $\int_{\Nhood_\epsilon}\int^{\infty}_{m_c}m\dmutxm\leq \iint m\dmutxm =1$ for mass conservation. The term in the logarithm is always bigger than $1$, so that the expression on the right hand side of~\eqref{eqn:t_ep_delta} quantifies positive time.}:
     \begin{equation}\label{eqn:t_ep_delta}
    t_{\epsilon}(m_c)=\frac{2}{\epsilon}\log\left(\left(\sup_{t\in[0,1]}\int_{\Nhood_\epsilon}\int^{\infty}_{m_c}m\dmutxm\right)^{-1}\right)+1,
    \end{equation}
    and $t\leq t_\epsilon(m_c)$ such that
    \begin{equation}\label{eqn:f_exist}
    \fbarmut<\f_*+\epsilon.
    \end{equation}
\end{theorem}
Observe that even if $\Nhood_\epsilon\cap \mathrm{supp}_\bx\muat{0}(\cdot,m)=\emptyset$, then the presence of the Brownian motion implies that there exists a time, $t\in(0,1]$, such that $\Nhood_\epsilon\cap \mathrm{supp}_\bx\mut(\cdot,m)\neq\emptyset$, and consequently, \eqref{eqn:t_ep_delta} always yields a finite number threshold, $t_\epsilon$. The combination of~\eqref{eqn:t_ep_delta}-\eqref{eqn:f_exist} together claims that with roughly $t_\epsilon\approx\frac{1}{\epsilon}$ time, the global optimal value can be found within $\epsilon$ accuracy, implying an $\mathcal{O}(1/\epsilon)$ convergence. However, the constant dependence heavily rely on the initial data, and we cannot claim the optimality of the `convergence rate': our numerical findings reported in Section \ref{sec:numerical_exp} indicate the possibility of a much better convergence rate in the sense that $t_\epsilon=\log(1/\epsilon)$ is already sufficient to ensure \eqref{eqn:f_exist}. See Figure~\ref{fig:Ackley_1D_results_2}.

Combing Theorems \ref{thm:mean_field_thm} and \ref{thm:long_time_thm} we conclude the following main result of this paper, namely ---  the convergence of the empirical distribution to the mean-field.
\begin{theorem}\label{thm:rigorous}
    Assume that Assumptions \ref{assumption: assmptn2}, \ref{assumption: assmptn1} hold\footnote{Assumption~\ref{assumption: assmptn1} requires the uniqueness of the global minimum. This assumption can be replaced by finite many global minima $\{\bx_k\}_{k=1}^K$ with $\f(\bx_k)=\f_*$ for all $k$. The details of the analysis need to be revised accordingly. In the proof (Section~\ref{sec:long_time_proof}), we point out the specific location to revise to accommodate the situation where multiple global minima co-exist.}. Let $\muN(\bx,m)$ be the empirical distribution associated with an ensemble of solutions of the swarm-based \SGD\,  \eqref{eq:SSA}, $(\bxit,\mitt)$, subject to the initial data drawn i.i.d. from $\muat{t=0}$. Fix an arbitrary $\epsilon >0$. Then, for a large enough crowd spelled out in~\eqref{eqn:close_in_f}, depending on time, $t_{\epsilon/2}$ defined in~\eqref{eqn:t_ep_delta},
\[
N > \frac{4C_{t_{\epsilon/2}}^{2}}{\epsilon^{2}} \ \ \textnormal{where} \ \  t_{\epsilon/2}<\frac{4}{\epsilon}\log\left(\left(\sup_{t\in[0,1]}\int_{\Nhood_{\epsilon/2}}\int^{\infty}_{m_c}m\dmutxm\right)^{-1}\right)+1,
\]
so that there is a $t<t_{\epsilon/2}$ and the provisional minimum $\fbar^N_t$ is  within $\epsilon$ of $\min_\bx \f(\bx)$, 
\[
\mathbb{E}\left[\left|\fbar^N_t- \min_{\bx}\f(\bx)\right|\right] < \epsilon.
\]
\end{theorem}
\noindent
\emph{Proof of Theorem \ref{thm:rigorous}.}
By Theorem~\ref{thm:long_time_thm}, there is a $t<t_{\epsilon/2}$ so that
\[
\left|\fbarmuat{t} - \min_{\bx}\f(\bx)\right| < \frac{\epsilon}{2},
\]
and according to Theorem~\ref{thm:mean_field_thm}, there exists a constant, $C_t$, such that 
\[
 \mathbb{E}\left[\left|\fbar^N_{t} - \fbarmuat{t}\right|\right]<\frac{C_t}{\sqrt{N}}.
 \]
Hence, for large enough $N$ such that $\displaystyle \frac{C_{t}}{\sqrt{N}}<\frac{\epsilon}{2}$, or equivalently $N > \frac{4C^2_{t_{\epsilon/2}}}{\epsilon^2}$ (note $C_{t_{\epsilon/2}}>C_t$), we conclude
\begin{equation*}
\begin{split}
\mathbb{E}\left[\left|\fbar^N_{t} - \min_{\bx}\f(\bx)\right|\right]  \leq\mathbb{E}\left[\left|\fbar^N_{t} - \fbarmuat{t}\right|\right] + \left|\fbarmuat{t} - \min_{\bx}\f(\bx)\right|\leq \epsilon\,. \quad  \square
 \end{split}
\end{equation*}

\section{From mean-field to macroscopic description}\label{sec:macro}
We shall briefly comment on the macroscopic description of the swarm-based \SGDc furnished in terms of the first two moments ---
the density, $\rho: {\mathbb R}_+\times {\mathbb R}^d \mapsto {\mathbb R}_+$, and momentum, $\rho \vel: {\mathbb R}_+\times {\mathbb R}^d \mapsto {\mathbb R}$, (observe that time-dependence  is denoted as usual for the  density $\rho=\rho(t,\bx)$,  macroscopic mass, $\vel=\vel(t,\bx)$ etc.)
\[
\rho(t,\bx):=\int \mutxm\nd m, \qquad \rho \vel(t,\bx):=\int m\mutxm\nd m.
\]
Taking the first two moments of the mean-field \eqref{eqs:mean_field}, we find
\be\label{eq:moments}
\left\{
\begin{split}
\rho_t - \nabla_\bx\cdot(\rho\nabla \f)&= \Delta_\bx\int \gammab(m)\mutxm\nd m, \\
(\rho \vel)_t - \nabla_\bx\cdot(\rho \vel \nabla \f)& = \big(\fbar(t)-\f(\bx)\big)\rho \vel + \Delta_\bx\int m\gammab(m)\mutxm\nd m.
\end{split}
\right.
\ee
We normalize the total mass and momentum, $\ds \int \rho(t,\bx)\dx\equiv 1$ and $\ds \int \rho \vel(t,\bx)\dx\equiv 1$, which leaves us with a provisional minimum, $\fbar(t)=\fbarmut$, given by 
\[
\fbar(t)=\int \f(\bx)\rho \vel(t,\bx) \dx. 
\]
To unlock a closed form of the diffusion on the right, we assume the simplest closure based on a mono-kinetic  pseudo-Maxwellian , $\mutxm=\rho(t,\bx)\delta(m-\vel(t,\bx))$. This yields 
\begin{subequations}\label{eqs:moments-closed}
\begin{align}\label{eq:moments-closeda}
\rho_t - \nabla_\bx\cdot(\rho\nabla \f)&= \Delta_\bx(\gammab(t,\bx)\rho(t,\bx)), \qquad \gammab(t,\bx):=\gammab(\vel(t,\bx))\\
(\rho \vel)_t - \nabla_\bx\cdot(\rho \vel \nabla \f)& = \big(\fbar(t)-\f(\bx)\big)\rho \vel + \Delta_\bx\big( \gammab(t,\bx)\rho \vel(t,\bx)\big). \label{eq:moments-closedb}
\end{align}
\end{subequations}

The momentum \eqref{eq:moments-closedb} can be converted into a  drift-diffusion
equation for the velocity $\vel$,  
\begin{subequations}\label{eqs:moments-final}
\begin{align}\label{eq:moments-finala}
\rho_t - \nabla_\bx\cdot(\rho\nabla \f)&= \Delta_\bx(\gammab(t,\bx)\rho(t,\bx)), \\
 \vel_t - \nabla \f\cdot\nabla_\bx \vel& = \big(\fbar(t)-\f(\bx)\big) \vel + \frac{1}{\rho}\Delta_\bx\big( \gammab(t,\bx)\rho \vel\big) - \frac{1}{\rho}\Delta_\bx\big( \gammab(t,\bx)\rho \big) \vel, \quad \bx\in \textnormal{supp}\{\rho(t,\cdot)\}.
\end{align}
\end{subequations}
We shall not dwell on a detailed study of \eqref{eqs:moments-final}, but note  the decent estimate
\[
\ddt \int (\f(\bx)-\f_*)\rho(t,\bx)\dx= -\int |\nabla \f(\bx)|^2\rho(t,\bx)\dx
+\int \Delta_\bx\f(\bx)\gammab(t,\bx)\rho(t,\bx)\dx,
\]
with \emph{expected}  long-time behavior $\rho(t,\bx) \stackrel{t\rightarrow \infty}{\longrightarrow}\delta(\bx-\bx_*)$ and $\vel(t,\bx) \stackrel{t\rightarrow \infty}{\longrightarrow}{\mathds 1}(\bx_*)$. Observe that $\gammab(t,\bx)$ should vanish as $\bx\rightarrow \bx_*$.

\section{Pitfalls}\label{sec:other_methods}
We examine the swarm-based \SGD\, in two extreme cases --- when $\gammab(m)\equiv 0$ and $\gammab(m)\equiv 1$. The former case turns off stochasticity and amounts to a deterministic system of $N$ independent gradient 
descending agents; we show that their convergence fails in probability sense. The latter turns on stochasticity uniformly for all agents, which amounts to a stochastic system of $N$ independent gradient descending agents; we show that their convergence fails in accuracy  sense.

We note that in both cases, the resulting systems are passive systems in the sense that  the weights of agents do \emph{not} affect their trajectories. As a consequence, agents roam  the landscape with equal randomness with no distinction between leaders and explorers. The convergence of the provisional minimum $\fbarmut$ is naturally worse than in the communicating swarm-based dynamics.

\subsection{Swarming with no communication}
\label{sec:deterministic}
What is  the effect of Brownian motion in our swarm-based \SGD? If we turn off the amplitude, $\gammab(m)\equiv 0$, then \eqref{eq:SSA} is reduced to a deterministic system of \emph{non-interacting} agents with no Brownian motion (and again we switch notations of time-dependence in deterministic quantities, $\bxi(t), \mi(t), \fbarN(t)$ etc.)
\be\label{eqn:deterministic_system}
    \left\{\begin{split}
        \ddt \bxi(t) & =  -\nabla \f(\bxi(t)) \\
        \ddt \mi(t) & =  -\mi(t)\big(\f(\bxi(t)) - \fbarN(t)\big), \quad 
    \end{split}\right. j=1,...,N.
\ee
We still assume random initial configurations: the initial drawing is independent, with $\{\bxi(0)\}^N_{j=1}$ drawn from an initial distribution $\rho_0$, and $\mi(0)=\frac{1}{N}$. Since the total mass is conserved, we can be $M(t) = 1$, so that   provisional minimum is reduced to the usual average $\fbarN(t) = \sum_{j=1}^{N}\mi(t)\f(\bxi(t))$. The main feature \emph{missing}  in \eqref{eqn:deterministic_system} is communication. Different agents at different positions proceed along gradient decent, independently of each other: since communication through provisional minimum is missing, the dynamics lacks the collective engagement of the  agents as a self-organized swarm.

The essential role of communication was already highlighted in  the context of SBGD method, \cite{eitan_2022}: once we decouple the step-size and mass, setting a \emph{uniform} step-size in \eqref{eq:SBGD}, that is, $\betab(m) \equiv 1$,
then the success rate of the resulting non-communicating swarm-based method decreases  dramatically.  
The system \eqref{eqn:deterministic_system} provides yet another realization of the essential role of communication. We note that since all particles are ``descending,'' it is also straightforward to show that the collective mean objective $\fbarN(t)$ is always decreasing: $\ddt\fbarN(t)\leq 0$. See Section \ref{sec:pf_of_thhm_1} for the proof.
However, even though $\fbarN(t)$ continues to decay in time, it does not guarantee the desired convergence to the global minimum,  $\ds\liminf_{t\rightarrow\infty}\fbarN(t)=\min_\bx \f(\bx)$. This is due to the passive nature of the gradient descent of the agents.  In particular, it is likely that all agents will become trapped at local minima or saddle points, especially when the number of particles is not sufficiently large. Turning on the communication term that adjusts strength of Brownian motion, as done in~\eqref{eq:SSA}, allows particles situated at local minima to experience increased randomness. This increased randomness facilitates their escape from local basins. With the communication and the randomness turned off, we have the following:
\begin{theorem}[{\bf Lack of communication and failure in probability}]\label{thm:small_probability} Assume $\f$ has a unique global 
minimum $\bx_*$ inside an open set $\Omega$ so that $\f_*=\f(\bx_*)$. Moreover, 
 $\nabla \f(\bx)=0$ on $\partial \Omega$. 
 If we let $\{\bxi(0)\}$ drawn from $\rho_0$, $\mi(0)=\frac{1}{N}$ and evolve according to~\eqref{eqn:deterministic_system}, then the success rate is bounded:
\begin{equation}\label{eqn:stop_decay_property} \mathbb{P}\left(\lim_{t\rightarrow\infty}\fbarN(t)=\f_*\right)\leq 1-\left(1-\int_\Omega \nd\rho_0\right)^N\,.
\end{equation}
In particular, if $\mathrm{supp}_\bx(\rho_0)\cap\Omega=\emptyset$, the success rate is $0$, meaning $\mathbb{P}\left(\lim_{t\rightarrow\infty}\fbarN(t)=\f_*\right)=0$ for all $N$.
\end{theorem}
\noindent
\emph{Proof}. To prove \eqref{eqn:stop_decay_property}, we first notice that without the Brownian motion, the deterministic system drives the samples to the local minima of the local basin in which the samples are located at $t=0$, so the event of $\left\{\lim_{t\rightarrow\infty}\fbarN(t)=\min \f\right\}\subset\left\{\exists \bx^i(0)\in \Omega\right\}$. This implies
\[
\mathbb{P}\left(\lim_{t\rightarrow\infty}\fbarN(t)=\min \f\right)\leq \mathbb{P}\left(\left\{\exists \bx^j(0)\in \Omega\right\}\right)=1-\left(1-\int_\Omega \nd\rho_0\right)^N\,.
\]
When $\mathrm{supp}_\bx(\rho_0)\cap\Omega=\emptyset$, we have $\int_\Omega\nd\rho_0=0$, so using this formula, we have for all $N$,
\[
\mathbb{P}\left(\lim_{t\rightarrow\infty}\fbarN(t)=\min \f\right)=0\,,
\]
Since $\fbarN(t)$ is a decreasing  in $t$, the above equality leads to
$\mathbb{P}\left(\lim_{t\rightarrow\infty}\fbarN(t)>\min \f\right)=1.\hfill \square$

This theorem explicitly spells out the ``probability failure'', so there is a nonzero probability that the mean objectives obtained by the particles cannot ever achieve the global optimizer.

\subsection{Stochastic system}
The limitation of \eqref{eqn:deterministic_system} clearly originates from the fact that particle evolution closely resembles classical gradient descent, lacking a mechanism for domain exploration and avoiding entrapment in local minima. The other end of the spectrum is to introduce ``too much'' exploration. Indeed,  setting $\sigma = 1$, we have:

\begin{equation}\label{eqn:langevin_system}
    \left\{\begin{split}
        \dxit & =  -\nabla \f(\bxit)\dt + \sqrt{2}\dWit \\
        \dmit & =  -\mitt\big(\f(\bxit) - \fbarNt\big)\dt
    \end{split}\right.,\quad j=1,...,N
\end{equation}
where $\Wit$ is the Brownian motion at time $t$ associated with particle $j$. The only difference compared to \eqref{eqn:deterministic_system} is the addition of a constant Brownian motion term. This Brownian motion term allows more effective exploration of the domain by preventing agents from becoming trapped at any local minima.

The trajectory of $\{\bxit\}$ resembles that of Langevin Monte Carlo, and under appropriate assumptions, it is well known that the distribution of $\bxit$ converges to $\rho_\infty\propto \exp(-\f(\bx))$ as $t\rightarrow\infty$. Incorporating mass exchange in \eqref{eqn:langevin_system} does not eliminate this issue. Carrying out the same mean-field analysis as was done in Section~\ref{sec:mean_field_analysis}, we find that the limiting PDE is:
\begin{equation}\label{eqn:F-P_ex_no_gamma}
\begin{split}
    \partial_{t}\mu = \nabla_\bx\cdot (\mu\nabla\f) + \big(\f(\bx)-\fbarmut\big)\partial_{m}(m\mu) &+\Delta_\bx\mu, \\
      \displaystyle \fbarmut = \frac{\ds \iint m\f(\bx)\mutxm\nd \bx\nd m}{\ds \iint m\mutxm\nd \bx\nd m}.
     \end{split}
\end{equation}
Analyzing this mean-field PDE, we could show that there is always a positive gap between $\fbarmut$ and $\f_{*}$.
\begin{theorem}\label{thm:no_gamma}
Consider a function $\f$ satisfying  the conditions outlined in Assumptions \ref{assumption: assmptn2} and \ref{assumption: assmptn1}, and further assume it is uniformly strongly convex in a ball around the global minimum $\bx_*$, meaning $\Delta_\bx \f\geq \xi d$ for $\xi>0$ and $\bx\in B_R(\bx_*)$. Then there exists $\epsilon>0$ (depending on $L,D,\xi, d$ and $R$) such that the provisional minimum associated with the mean field $\mu$ in \eqref{eqn:F-P_ex_no_gamma} satisfies
\begin{equation}\label{eqn:lower_bound_f_bar}
\fbarmut\geq \min\{\f_*+\epsilon,\fbarmuat{0}\}, \qquad \forall t>0.
\end{equation}
\end{theorem}
For brevity, the proof of this theorem is in Appendix \ref{sec:langevin}. This theorem explicitly spells out the ``accuracy failure.'' The global optimizer can not be achieved by the mean objective.

\section{Proofs of the main results}\label{sec:proofs_main_results}
In this section we prove Theorem \ref{thm:mean_field_thm} on the convergence to the mean-field limit, and Theorem \ref{thm:long_time_thm} on the large-time behavior of the limiting equation. 
\subsection{Convergence to the mean-field limit}\label{sec:mean_field_proof}
 Theorem~\ref{thm:mean_field_thm} states that the empirical distribution $\muN$ that assembles all samples that follow the coupled SDE system~\eqref{eq:SSA}, is close to the limiting distribution $\mu$ that solves the PDE~\eqref{eqn:mean_field_pde}.

There are a few different but related approaches for justifying mean-field limits, including the coupling method, the energy/entropy estimates, the utilization of BBGKY (Bogoliubov–Born–Green–Kirkwood–Yvon hierarchy) hierarchies and running the tightness argument to confine a sequence of measures for the convergence; see references~\cite{Meleard1996, Sznitman1991}. Here  we employ the coupling method, introduced in~\cite{Sznitman1991}. The  coupling method places a stronger requirement on the field, but the machinery is rather simple to use, and it has the advantage of providing precise convergence rate in terms of  the number of particles, $N$.

To apply the coupling method, one usually designs an auxiliary system that is pushed forward by the underlying field. This auxiliary system serves as a bridge  linking  the discrete SDE system and the limiting mean-field PDE. On the one hand, it is driven by the underlying field and thus its ensemble closely resembles $\rho$, and on the other hand, this system adopts the SDE structure and can be easily compared with the original self-consistent SDE system. In our context, we define such auxiliary system as $\left\{\tildebxit\,,\tildemit\right\}_{j=1}^N$, that are governed by:
\begin{equation}\label{eqn:swarm_continuous}
    \left\{\begin{array}{rcl}
    \nd\tildebxit & = & -\nabla \f\left(\tildebxit\right)\dt + \sqrt{2\gammab\left(\tildemit\right)}\dWit \\
    \nd\tildemit & = &  -\tildemit(\f\left(\tildebxit\right) - \fbarmut)\dt,
    \end{array}\right.
\end{equation}
with the initial data set to be $\{\widetilde{\bx}^j_0 = \bxi_0\,,\widetilde{m}^j_0=\mi_0\}$, and $\fbarmu = \mathbb{E}_\mu[mF(\bx)]/\mathbb{E}_\mu[m]$ was defined in~\eqref{eqn:bar_f_t_mu}. For this system, we define the corresponding empirical distribution and the provisional minimum:
\begin{equation}\label{eqn:aux_continuous}
    \tildemuNt=\frac{1}{N}\sum_i\delta_{\tildebxit}(\bx)\otimes \delta_{\tildemit}(m)\,,\quad\text{and}\quad
    \tildefbarNt  = \frac{\frac{1}{N}\sum_{j}\tildemit \f\left(\tildebxit\right)}{\mathbb{E}_\mu[m]}\,.
\end{equation}

We note that this system has a clear similarity with the original SDE~\eqref{eq:SSA} but is passive, in the sense that the dynamics of $\tildemit$ is driven by $\fbarmu$ instead of $\tildefbarNt$. While the dynamics of $\widetilde{\bx}_t$ agrees with that of $\bx_t$, the mass changes call for different rates. In particular, in~\eqref{eq:SSA}, we compare $\f(\bx_t)$ with $\fbarNt$ that is self-generated in~\eqref{eqn:ave_stochastic}, and here, $\f(\widetilde{\bx}_t)$ is compared with $\fbarmu$ defined from the underlying PDE system~\eqref{eqn:mean_field_pde}. As a consequence, this new $\{\widetilde{\bx}^j_0 = \bxi_0\,,\widetilde{m}^j_0=\mi_0\}$ system gets passively pushed forward by the underlying $\mu$.

This design of the new auxiliary system prompts two lemmas. 
\begin{lemma}\label{lem:aux_mean}
Suppose Assumption~\ref{assumption: assmptn2} holds. Let $\{\tildebxit\,,\tildemit\}_{j=1}^N$ be the system defined in~\eqref{eqn:swarm_continuous}
 and let $\mut$ solve~\eqref{eqn:mean_field_pde}. If the initial condition $\{\widetilde{\bx}^j_0 = \bxi_0\,,\widetilde{m}^j_0=\mi_0\}$ is drawn from $\muat{t=0}$, then
\begin{itemize}
    \item $\mut$ and $\tildemuNt$ are close in Wasserstein sense:
     \[
    W_{2}(\mut,\tildemuNt) \to 0\,, \quad \text{in probability as }N\to \infty\,.
    \]
    \item There exists a constant $C= C(t, d, \muat{0}, \f)$ independent of $N$ such that $\tildefbarNt$ and $\fbarmut$ are close
    \[
    \mathbb{E}(|\tildefbarNt - \fbarmut|) < \frac{C}{\sqrt{N}}\,,
    \]
    where $\tildefbarNt$ and $\fbarmut$ are defined in~\eqref{eqn:aux_continuous} and~\eqref{eqn:bar_f_t_mu} respectively.
    \end{itemize}
\end{lemma}
and
\begin{lemma}\label{lem:aux_sde}
Suppose Assumption~\ref{assumption: assmptn2} holds. Let $\{\tildebxit\,,\tildemit\}_{j=1}^N$ be the system defined in \eqref{eqn:swarm_continuous}, and let $\{{\bx}^{j}_{t}\,,{m}^j_t\}_{j=1}^N$ solve~\eqref{eq:SSA}. Suppose the two systems have the same initial conditions, then
\begin{itemize}
    \item $\muN$ and $\tildemuNt$ are close in Wasserstein sense:
     \begin{equation}\label{eqn:mean_field_w2}
    W_{2}(\muN,\tildemuNt) \to 0\,, \quad \text{in probability as }N\to \infty\,.
    \end{equation}
    \item There exists a constant $C= C(t, d, \muat{0}, \f)$ independent of $N$ such that $\tildefbarNt$ converges to $\fbarNt$ with the rate of
      \begin{equation}\label{eqn:meanfield_objective}
    \mathbb{E}(|\tildefbarNt - \fbarNt|) < \frac{C}{\sqrt{N}}\,.
    \end{equation}
    Here $\tildefbarNt$ and $\fbarNt$ are defined in~\eqref{eqn:aux_continuous} and~\eqref{eqn:ave_stochastic}.
\end{itemize}
\end{lemma}

With these two lemmas, Theorem~\ref{thm:mean_field_thm} is a straightforward corollary, calling triangle inequalities:
\[
W_2(\mut,\muN) \leq W_2(\mut,\tildemuNt)+W_2(\tildemuNt,\muN)\,,
\] 
and
\[
\mathbb{E}(|\fbarmut - \fbarNt|)\leq \mathbb{E}(|\fbarmut-\tildefbarNt|)+\mathbb{E}(|\tildefbarNt - \fbarNt|)\,.
\]
Here we omit the details and proceed to provide the proof of the two lemmas.

\begin{proof}[Proof of Lemma~\ref{lem:aux_mean}]
Calling Theorem 1 in~\cite{Fournier2015},
\[
\mathbb{E}\left(W_2(\mut,\tildemuNt)\right)=\mathcal{O}(N^{-\theta}) \to 0\,,\quad\text{as}\quad N\to\infty\,,
\]
with $\theta=\max\{1/2,1/d\}$. This implies
    \[
    W_{2}(\mut,\tildemuNt) \to 0\,, \quad \text{in probability as }N\to \infty\,.
    \]
For weak convergence, we note that $(\widetilde{\bx},\widetilde{m})$ are i.i.d. samples from $\mu$ and thus:
\[
\mathbb{E}\left(\left|\frac{1}{N}\sum_{j=1}^N\tildemit \f\left(\tildebxit\right)-\int m\f(\bx)\dmutxm\right|\right)<\frac{\mathrm{Var}^{1/2}_{\mu}(m\f)}{\sqrt{N}}\leq \frac{C(t)M}{\sqrt{N}}\,.
\]
\end{proof}

The proof for Lemma~\ref{lem:aux_sde} is the standard coupling method. This amounts to comparing the original SDE~\eqref{eq:SSA} with the auxiliary SDE~\eqref{eqn:swarm_continuous}. In particular, we note that the $\bx$ component of the two systems are almost identical:
\[
\bxit = -\nabla\f(\bxit)\dt+\sqrt{2\sigma(\mitt)}\dWit\,,\quad\text{vs}\quad\tildebxit = -\nabla\f(\tildebxit)\dt +\sqrt{2\sigma(\tildemit)}\dWit\,.
\]
Then we can subtract one from the other and trace the growth of the error. The same strategy is applied to tracking $\mitt-\tildemit$. The combination of the two pieces of error gives the control of the Wasserstein distance of $W_2(\muN,\tildemuNt)$.

\begin{proof}[Proof of Lemma~\ref{lem:aux_sde}]
Noting that $\muN = \frac{1}{N}\sum_j\delta_{\bxit}(\bx)\otimes \delta_{\mitt}(m)$ and $\tildemuNt = \frac{1}{N}\sum_j\delta_{\tildebxit}(\bx)\otimes \delta_{\tildemit}(m)$, then recall the definition of the Wasserstein metric, we have:
\begin{equation}\label{eqn:bound_w_ensemble}
W_2(\tildemuNt,\muN)^2\leq \left(\frac{1}{N}\sum^N_{j=1}\left[|\bxit-\tildebxit|^2+\left|\mitt-\tildemit\right|^2\right]\right)\,.
\end{equation}
Define:
\[
\Omega=\left\{\left|\frac{1}{N}\sum_{j=1}^N m^j_0-\mathbb{E}_{\mu_0}(m)\right|\leq \frac{1}{2}\right\}\,,
\]
With Chebyshev's inequality and boundedness of $m$ from Assumption~\ref{assumption: assmptn2}, we obtain
\[
\mathbb{P}\left(\Omega^c\right)\leq \frac{4\mathbb{E}_{\mu_0}(m^2)}{N}\leq \frac{C}{N}\,,\quad\text{and thus accordingly}\quad \mathbb{P}\left(\Omega\right)\geq 1-\frac{C}{N}\geq \frac{1}{2}\,.
\]
Define
\begin{equation}\label{eqn:def_G}
G(t) = \mathbb{E}\left(\frac{1}{N}\sum^N_{j=1}\left[|\bxit-\tildebxit|^2+\left|\mitt-\tildemit\right|^2\right]\middle| \Omega\right)\,.
\end{equation}
If we can show $G(t)\sim 1/N$, then conditioned on $\Omega$, a set whose probability goes to $1$, $W_2(\tildemuNt,\muN)^2\to0$ as $N\to\infty$, proving~\eqref{eqn:mean_field_w2}.

Similarly, to estimate~\eqref{eqn:meanfield_objective}, we note:
\begin{equation}\label{eqn:bound:mean_obj_ensemble}
\begin{aligned}
&|\tildefbarNt - \fbarNt| \\
\leq &\left|\frac{\frac{1}{N}\sum^N_{j=1}\left( \tildemit \f\left(\tildebxit\right)-\mitt \f(\bxit)\right)}{\mathbb{E}_{\mu}[m]}\right|+\left|\left(\frac{1}{N}\sum^N_{j=1}\mitt \f(\bxit) \right)\left(\frac{1}{\mathbb{E}_{\mu}[m]}-\frac{1}{\frac{1}{N}\sum m^j_t}\right)\right| \\
\underset{\rm (i)}{\leq} & C\left[\left|\frac{1}{N}\sum^N_{j=1}\left(\mitt-\tildemit\right)\f(\tildebxit)  + \frac{1}{N}\sum^N_{j=1}\mitt\left(\f(\bxit)-\f\left(\tildebxit\right)\right)\right|+\left|\frac{\frac{1}{N}\sum m^j_t-\mathbb{E}_{\mu}[m]}{\mathbb{E}_{\mu}[m]\left(\frac{1}{N}\sum m^j_t\right)}\right|\right]\\
\underset{\rm (ii)}{=} &C\left[\left|\frac{1}{N}\sum^N_{j=1}\left(\mitt-\tildemit\right)\f(\tildebxit)  + \frac{1}{N}\sum^N_{j=1}\mitt\left(\f(\bxit)-\f\left(\tildebxit\right)\right)\right|+\left|\frac{\frac{1}{N}\sum m^j_0-\mathbb{E}_{\mu_0}[m]}{\mathbb{E}_{\mu_0}[m]\left(\frac{1}{N}\sum m^j_0\right)}\right|\right]\\
\underset{\rm (iii)}{\leq} & C\left[\sqrt{\frac{1}{N}\sum^N_{j=1}\left[|\bxit-\tildebxit|^2+\left|\mitt-\tildemit\right|^2\right]}+\left|{\frac{1}{N}\sum m^j_0-\mathbb{E}_{\mu_0}[m]}\right|\right]\,,
\end{aligned}
\end{equation}
Here in (i) we used the upper boundedness of both $m$\footnote{Firstly we notice that, according to the Boundedness in $m$ condition in Assumption~\ref{assumption: assmptn2}, the initial data for $\mitt$ is bounded by $M_{\max}$, meaning $|m^j_0|<M_{\max}$ for all $j$ at the initial time. Given that $|\f| \leq D$, then according to the $\mitt$ equation in~\eqref{eq:SSA}, we call Gr\"onwall inequality to obtain that
\begin{equation}\label{eqn:bound_masses}
|\mitt|\leq M_{\max}e^{2Dt}
\end{equation}
for all time. The same argument applies to bound $\tildemit$. We denote the upper bound $c$ that depends on $M_{\max}$, $D$ and $t$.
\label{foot:bound_mass_foot}} and $\f$. In (ii), we apply the conservation of mass: $\sum m^j_t = \sum m^j_0$ and $\mathbb{E}_{\mu}[m] = \mathbb{E}_{\mu_0}[m]$. In (iii), we used the triangle inequality, Cauchy-Schwarz inequality, and Lipschitz continuity of $\f$. Consequently, the constant $C$ depends on $D$, $L$, and $M_{\max}$ (from Assumption~\ref{assumption: assmptn2}).

To insert it to~\eqref{eqn:meanfield_objective}, 
then:
\[
\mathbb{E}(|\tildefbarNt - \fbarNt|) \leq \mathbb{E}\left(|\tildefbarNt - \fbarNt| \big| \Omega\right)\mathbb{P}(\Omega) + \mathbb{E}\left(|\tildefbarNt - \fbarNt| \big| \Omega^c\right)\mathbb{P}(\Omega^c)\,.
\]

So combining with~\eqref{eqn:bound:mean_obj_ensemble}, we rewrite~\eqref{eqn:meanfield_objective}:
\begin{equation}\label{eqn:mean_obj_ensemble}
\begin{aligned}
    &\mathbb{E}(|\tildefbarNt - \fbarNt|) \\
    \leq &\mathbb{E}\left(|\tildefbarNt - \fbarNt| \big| \Omega\right) + \frac{C}{N} \\
   \leq &C\underbrace{\mathbb{E}\left(\sqrt{\frac{1}{N}\sum^N_{j=1}\left[|\bxit-\tildebxit|^2+\left|\mitt-\tildemit\right|^2\right]}\middle| \Omega\right)}_{\leq \sqrt{G(t)}}+C\underbrace{\mathbb{E}\left(\left|\frac{1}{N}\sum m^j_0-\mathbb{E}_{\mu_0}[m]\right|\middle| \Omega\right)}_{\leq \frac{1}{\mathbb{P}(\Omega)}\mathbb{E}\left(\left|\frac{1}{N}\sum m^j_0-\mathbb{E}_{\mu_0}[m]\right|\right)\leq \frac{C}{\sqrt{N}}}+\frac{C}{N}\,.
    \end{aligned}
\end{equation}
Once again,
if we can show $G(t)\sim 1/N$,~\eqref{eqn:meanfield_objective} is complete.

The rest of the proof is dedicated to showing the bound for $G(t)$. To do so, we will trace the evolution of $|\bxit-\tildebxit|^2$ and $|\mitt-\tildemit|^2$ respectively. 
\begin{itemize}
    \item To trace $|\bxit-\tildebxit|^2$, we compare the $\bxit$-equation in~\eqref{eq:SSA} and the $\tildebxit$-equation in~\eqref{eqn:swarm_continuous} to have:
\[
g(\bx) = |\bx|^{2}, \quad \nabla_\bx g(\bx) = 2\bx, \quad D^2_\bx g = 2{\mathbb I}\,.
\]
Then $g(\bxit - \tildebxit) = |\bxit - \tildebxit|^{2} = \langle \bxit - \tildebxit, \bxit - \tildebxit\rangle$, and with It\^o's formula we get
\begin{align*}
\nd g(\bxit - \tildebxit) & = \left(\langle \nabla g,-(\nabla \f(\bxit)-\nabla \f\left(\tildebxit\right))\rangle + \text{Tr}\left(\left(\sqrt{\gammab(\mitt)} - \sqrt{\gammab\left(\tildemit\right)}\right)^{2}D^2_{\bx}g\right)\right)\dt \\
& \quad + \left\langle \nabla g, \sqrt{2\gammab(\mitt)}-\sqrt{2\gammab\left(\tildemit\right)}\right\rangle\dWit \\
& = \left(-2\langle \bxit - \tildebxit, \nabla \f(\bxit) - \nabla \f\left(\tildebxit\right)\rangle + 2d\left|\sqrt{\gammab(\mitt)} - \sqrt{\gammab\left(\tildemit\right)}\right|^{2}\right)\dt \\
& \quad + 2\left\langle \bxit - \tildebxit, \sqrt{2\gammab(\mitt)} - \sqrt{2\gammab\left(\tildemit\right)} \right\rangle\dWit\,.
\end{align*}
Taking the expectation of the whole formula, the mean of the Brownian motion gets dropped out and we have:
\begin{equation}\label{eqn:x-y_evolve}
\begin{aligned}
    \ddt\mathbb{E}[g(\bxit - \tildebxit)|\Omega]&=\ddt\mathbb{E}[|\bxit - \tildebxit|^{2}|\Omega] \leq 2(L+L^2)\left(\mathbb{E}[|\bxit - \tildebxit|^{2}|\Omega] + d\mathbb{E}[|\mitt-\tildemit|^{2}|\Omega]\right)\,,
    \end{aligned}
\end{equation}
where we used Assumption~\ref{assumption: assmptn2} that both $\nabla\f$ and $\sqrt{\gammab}$ are $L$-Lipschitz.
\item To trace $|\mitt-\tildemit|^2$, we compare the $m$-equation in~\eqref{eq:SSA} and the $m$-equation in~\eqref{eqn:swarm_continuous} to have:
\begin{equation}\label{eqn:m-n_evolve_prepare1}
\begin{aligned}
    \ddt |\mitt - \tildemit|^{2} & = -2(\mitt - \tildemit)(\mitt \f(\bxit) - \tildemit \f(\tildebxit)) \\
    & \quad + 2(\mitt-\tildemit)\left(\frac{\frac{1}{N}\sum_{i=1}^{N}\miit \f(\bxiit)}{\frac{1}{N}\sum^N_{i=1}m^i_t} - \frac{\mathbb{E}_{\mu}[mF(\bx)]}{\mathbb{E}_{\mu}(m)}\right)\\
    & = -2(\mitt - \tildemit)^{2}\f(\bxit)- 2(\mitt - \tildemit)\tildemit\left(\f(\bxit)-\f(\tildebxit)\right) \\
    & \quad +2(\mitt-\tildemit)\left(\frac{1}{N}\sum_{i=1}^{N} \miit \f(\bxiit) - \frac{1}{ N}\sum_{i=1}^{N}\tildemiit \f(\tildebxiit)\right)/\left(\frac{1}{N}\sum^N_{i=1}m^i_t\right) \\
    & \quad + 2(\mitt-\tildemit)\left(\frac{\frac{1}{N}\sum_{i=1}^{N}\tildemiit \f(\tildebxiit)}{\frac{1}{N}\sum^N_{i=1}m^i_t} - \frac{\mathbb{E}_{\mu}[m\f(\bx)]}{\mathbb{E}_{\mu}(m)}\right)\,.
    \end{aligned}
    \end{equation}
Re-order all terms, we have
\begin{equation}\label{eqn:m-n_evolve_prepare}
\begin{aligned}
    \ddt |\mitt - \tildemit|^{2} & = \underbrace{-2(\mitt - \tildemit)^{2}\f(\bxit)}_{\text{(i)}} \underbrace{- 2(\mitt - \tildemit)\tildemit(\f(\bxit)-\f\left(\tildebxit\right))}_{\text{(ii)}} \\
    & \quad \underbrace{+ 2(\mitt - \tildemit)\left(\frac{1}{N}\sum_{i=1}^{N}(\miit - \tildemiit)\f(\bxiit)\right)/\left(\frac{1}{N}\sum^N_{i=1}m^i_t\right)}_{\text{(iii)}} \\
    & \quad + \underbrace{2(\mitt - \tildemit)\left(\frac{1}{N}\sum_{i=1}^{N}\tildemiit(\f(\bxiit) - \f(\tildebxiit))\right)/\left(\frac{1}{N}\sum^N_{i=1}m^i_t\right)}_{\text{(iv)}} \\
    & \quad \underbrace{+ 2(\mitt - \tildemit)\left(\frac{\frac{1}{N}\sum_{i=1}^{N}\tildemiit \f(\tildebxiit)}{\frac{1}{N}\sum^N_{i=1}m^i_t} - \frac{\mathbb{E}_{\mu}[m\f(\bx)]}{\mathbb{E}_{\mu}(m)}\right)}_{\text{(v)}}\,.
\end{aligned}
\end{equation}
We analyze each term separately,
\begin{enumerate}[label=(\roman*)]
    \item According to Assumption~\ref{assumption: assmptn2}, $\f(\bxit) \geq 0$ so,
     $\displaystyle   -2(\mitt-\tildemit)^{2}\f(\bxit) \leq 0$.
    \item Using Young's inequality,
    \begin{equation*}
    \begin{split}
    -2(\mitt-\tildemit)\tildemit(\f(\bxit)-\f\left(\tildebxit\right)) 
    & \leq  \tildemit(|\mitt-\tildemit|^{2} + |\f(\bxit) - \f\left(\tildebxit\right)|^{2}) \\
    & \leq  c[|\mitt-\tildemit|^{2} + L^{2}|\bxit - \tildebxit|^{2}]\,,
    \end{split}
    \end{equation*}
    where $c$ comes from the boundedness of $m$, seen in Footnote~\ref{foot:bound_mass_foot} and $L$ is the Lipschitz constant from Assumption~\ref{assumption: assmptn2}.
    \item Because of the total mass preserving and the definition of $\Omega$, we have
    \[
    \frac{1}{N}\sum^N_{i=1}m^i_t=\frac{1}{N}\sum^N_{i=1}m^i_0\geq \frac{1}{2}\,.
    \]
    Thus,
    \begin{equation*}
    \begin{split}
     &2(\mitt - \tildemit)\left(\frac{1}{N}\sum_{i=1}^{N}(\miit - \tildemiit)\f(\bxiit)\right)/\left(\frac{1}{N}\sum^N_{i=1}m^i_t\right)\\
      \leq & 2|\mitt - \tildemit|^{2} + 2\left(\frac{1}{N}\sum_{i=1}^{N}(\miit - \tildemiit)\f(\bxiit)\right)^{2} \\
      \leq  & D^2\left[|\mitt - \tildemit|^{2} + \left(\frac{1}{N}\sum_{i=1}^{N}(\miit - \tildemiit)\right)^{2}\right] \\
     \leq  & D^2\left[|\mitt - \tildemit|^{2} + \frac{1}{N}\sum_{i=1}^{N}|\miit - \tildemiit|^{2}\right]\,,
    \end{split}
    \end{equation*}
    where we used Young's in the first inequality, boundedness of $\f$ in the second (following Assumption~\ref{assumption: assmptn2}), and the H\"older's inequality in the last.
    \item Using the same argument as in (iii) we get
    \begin{equation*}
    \begin{split}
    & 2(\mitt - \tildemit)\left(\frac{1}{N}\sum_{i=1}^{N}\tildemiit(\f(\bxiit) - \f(\tildebxiit))\right)/\left(\frac{1}{N}\sum^N_{i=1}m^i_t\right) \\
    \leq & 2|\mitt - \tildemit|^{2} + 2\left(\frac{1}{ N}\sum_{i=1}^{N}\tildemiit(\f(\bxiit) - \f(\tildebxiit))\right)^{2}\\
    \leq & c\left[|\mitt - \tildemit|^{2} + \left(\frac{1}{ N}\sum_{i=1}^{N}(\f(\bxiit)-\f(\tildebxiit))\right)^{2}\right] \\
    \leq & c\left[|\mitt - \tildemit|^{2} + \frac{1}{ N}\sum_{i=1}^{N}|\f(\bxiit) - \f(\tildebxiit)|^{2}\right] \\
    \leq & c\left[|\mitt - \tildemit|^{2} + \frac{L^{2}}{N}\sum_{i=1}^{N}|\bxiit - \tildebxiit|^{2}\right]
    \end{split}
    \end{equation*}
    where we used Young's inequality in the first inequality, boundedness of $\tildemit$ in the second (see Footnote~\ref{foot:bound_mass_foot}), H\"older's in the third, and finally the Lipschitz continuity of $\f$ in the last inequality.
    \item Using the same argument,
    \begin{align*}
    &2(\mitt - \tildemit)\left(\frac{\frac{1}{N}\sum_{i=1}^{N}\tildemiit \f(\tildebxiit)}{\frac{1}{N}\sum^N_{i=1}m^i_t} - \frac{\mathbb{E}_{\mu}[m\f(\bx)]}{\mathbb{E}_{\mu}(m)}\right) \\ 
    = &2(\mitt - \tildemit)\left(\frac{1}{N}\sum_{i=1}^{N}\tildemiit \f(\tildebxiit) - \mathbb{E}_{\mu}[m\f(\bx)]\right)/\left(\frac{1}{N}\sum^N_{i=1}m^i_t\right)\\
    &+2(m^j_t-\widetilde{m}^i_t)\mathbb{E}_{\mu}[m\f(\bx)]\left(\frac{1}{\frac{1}{N}\sum^N_{i=1}m^i_t}-\frac{1}{\mathbb{E}_{\mu}(m)}\right)\\
    \leq & 4|\mitt - \tildemit|^{2}  + 2\left|\frac{1}{N}\sum_{i=1}^{N}\tildemiit\f(\tildebxiit) - \mathbb{E}_{\mu}[m\f(\bx)]\right|^{2}+2D^2\left|\frac{1}{N}\sum^N_{i=1}m^i_0 - \mathbb{E}_{\mu_0}(m)\right|^2\,,
    \end{align*}
where we use $\frac{1}{N}\sum^N_{i=1}\miit=\frac{1}{N}\sum^N_{i=1}m^i_0\geq\frac{1}{2}$, $\mathbb{E}_{\mu}(m)=\mathbb{E}_{\mu_0}(m)$, and boundedness of $\f$ in the last inequality.
\end{enumerate}
Collecting these terms together to feed in~\eqref{eqn:m-n_evolve_prepare}, and sum up against $j$, we have:
\begin{equation}\label{eqn:m-n_evolve}
\begin{aligned}
    &\ddt\left(\mathbb{E}\left[\frac{1}{N}\sum_{j=1}^{N}|\mitt - \tildemit|^{2}\middle|\Omega\right]\right)\\
    &\leq \widetilde{C}\mathbb{E}\left[\frac{1}{N}\sum_{j=1}^{N}|\bxit - \tildebxit|^{2} + \frac{1}{N}\sum_{j=1}^{N}|\mitt - \tildemit|^{2}\middle|\Omega\right] \\
    & \quad + \widetilde{C}\mathbb{E}\left[\left|\frac{1}{N}\sum_{i=1}^{N}\tildemiit\f(\tildebxiit) - \mathbb{E}_{\mu}[m\f(\bx)]\right|^{2}+\left|\frac{1}{N}\sum^N_{i=1}m^i_0 - \mathbb{E}_{\mu_0}(m)\right|^2\middle|\Omega\right]\,,
     \end{aligned}
\end{equation}
where $\widetilde{C}$ depends on $D$, $M_{\max}$, $t$ and $L$.
\end{itemize}
Recall the definition of $G$ in~\eqref{eqn:def_G}. We collect~\eqref{eqn:x-y_evolve} and~\eqref{eqn:m-n_evolve} to have:
\begin{equation}\label{eqn:expectancy_bound}
\begin{aligned}
\ddt G(t)
& \leq  C(t)\mathbb{E} \left[\frac{1}{N}\sum_{j=1}^{N}|\bxit - \tildebxit|^{2}+ \frac{1}{N}\sum_{j=1}^{N}|\mitt - \tildemit|^{2}\middle|\Omega\right] \\
 & \qquad + C(t)\underbrace{\mathbb{E}\left[\left|\frac{1}{N}\sum_{i=1}^{N}\tildemiit\f(\tildebxiit) - \mathbb{E}_{\mu}[m\f(\bx)]\right|^{2}+\left|\frac{1}{N}\sum^N_{i=1}m^i_0 - \mathbb{E}_{\mu_0}(m)\right|^2\middle|\Omega\right]}_{\text{(vi)}}\\
& =  {C}(t)\left(G(t) + \text{\text{(vi)}}\right),
\end{aligned}
\end{equation}
where ${C}(t)$ depends on $D$, $M_{\max}$, $t$, $L$ and $d$. Noting that $\text{term (vi)}$ can be further controlled using the law of large numbers: for the first term, let $X=m\f(\bx)$ be a random variable. Then, $\tildemit \f\left(\tildebxit\right)$ is a realization of $X$ and we have
\[
\overline{X}_N = \frac{1}{N}\sum_{j=1}^{N} \tildemit \f(\tildebxit)\,.
\]
Since $\{(\tildebxit,\tildemit)\}_{j=1}^{N}$ are i.i.d. samples of $\mu$, we have
\[
\mathbb{E}[|\overline{X}_N - \mathbb{E}[X]|^{2}|\Omega] \leq \mathbb{E}[|\overline{X}_N - \mathbb{E}[X]|^{2}]/\mathbb{P}(\Omega)\leq \text{Var}[\overline{X}_N]/\mathbb{P}(\Omega) \leq \frac{1}{N}\text{Var}[X]/\mathbb{P}(\Omega) = \mathcal{O}\left(\frac{1}{N}\right)\,,
\]
where we use $\mathbb{P}(\Omega)\geq 1 - \frac{C}{N}$. Similar bound also holds for the second term in (vi). Combining with~\eqref{eqn:expectancy_bound}, we have:
\begin{equation*}   \ddt G(t)  \leq  C(t)G(t) + \mathcal{O}\left(\frac{1}{N}\right) \quad\text{with}\quad G(0) =  0\,.
\end{equation*}
Finally, calling Gr\"onwall's inequality, we have
$G(t) \leq \mathcal{O}\left(\frac{1}{N}\right),$
which completes the proof.
\end{proof}

\subsection{Large time behavior}\label{sec:long_time_proof}
This section is dedicated to the proof of Theorem~\ref{thm:long_time_thm}. This theorem quantifies the large time behavior of the PDE~\eqref{eqn:mean_field_pde}. We note that~\eqref{eqn:f_liminf} and~\eqref{eqn:f_exist} are asymptotic and rate-specific respectively. Both proofs are by contradiction. To start off, we first present two lemmas.
\begin{lemma}\label{lem:long_time_lemma_1}
Under the condition of Theorem~\ref{thm:long_time_thm}, assume $\liminf_{t\rightarrow\infty}\fbarmut>\f_*$, then: there exists $\epsilon_0>0$ such that for any $t\geq0$, we have $\fbarmut>\f_*+\epsilon_0$.
\end{lemma}
At the same time,
\begin{lemma}\label{lem:long_time_lemma_2}
    Under the condition of Theorem~\ref{thm:long_time_thm}, if there is an $\epsilon>0$ so that $\fbarmut>\f_*+\epsilon$ for all $t\geq 0$, then the total mass $\int\int m\dmutxm\to\infty$ as $t\to\infty$ and is not conserved.
\end{lemma}
The proof of these two lemmas are found in Appendix~\ref{sec:appendix_b}\footnote{The two lemmas require the assumption of a unique global minimum as stated in Assumption~\ref{assumption: assmptn1}. However, in the proof in Appendix~\ref{sec:appendix_b}, we will discuss that this assumption can be relaxed, but analysis needs to be altered.}. Putting these two lemmas together, we can prove the theorem.

\begin{proof}[Proof of Theorem~\ref{thm:long_time_thm}]
We prove~\eqref{eqn:f_liminf} using proof by contradiction. This is to assume~\eqref{eqn:f_liminf} is not true, meaning $\liminf_{t\rightarrow\infty}\fbarmut>\f_*$. Then according to Lemma~\ref{lem:long_time_lemma_1}, there exists $\epsilon_0>0$ such that $\fbarmut>\f_*+\epsilon_0$ for all $t\geq 0$. Setting $\epsilon = \epsilon_0$ in Lemma~\ref{lem:long_time_lemma_2}, mass is not conserved, contradicting the property of the equation. Hence~\eqref{eqn:f_liminf} has to be true, namely:
\[
\liminf_{t\to\infty}\fbarmut = \f_*\,.
\]

To show~\eqref{eqn:f_exist}, we again deploy a proof by contradiction and the argument is rather similar to that of~\eqref{eqn:f_liminf}. Assuming that it is not true, then for all $t\leq t_{\epsilon}(m_{c})$, 
\[
\fbarmut\geq \f_* + \epsilon\,.
\]
We will show this leads to a contradiction. Define $\eta_{\epsilon}$ to be the time so that

\[
\eta_\epsilon = \arg\sup_{t\in[0,1]}\int_{\Nhood_\epsilon}\int^\infty_{m_c}m\dmutxm\,\quad\Rightarrow \quad \int_{\Nhood_\epsilon}\int^\infty_{m_c}m\dmuxmat{\eta_\epsilon}>0\,.
\]
Since the probability is non-trivial, one can sample a point in the set of $\Nhood_\epsilon\times(m_c,\infty)$. Starting with $\bx(\eta_\epsilon)\in \Nhood_\epsilon$ and $m(\eta_\epsilon)>m_c$, with the same calculation as in~\eqref{eqn:total_mass_lower}, we get:
\[
\bx(t)\in \Nhood_\epsilon,\quad m(t)>e^{\epsilon (t-\eta_\epsilon)/2}m(\eta_\epsilon) \geq e^{\epsilon (t-1)/2}m(\eta_\epsilon)
\]
for any $t>1\geq \eta_\epsilon$, which makes:
\[
\int_{\Nhood_\epsilon}\int^\infty_{m_c}m\dmutxm>e^{\epsilon (t-1)/2}\int_{\Nhood_\epsilon}\int^\infty_{m_c}m\dmuxmat{\eta_\epsilon}\,.
\]
In particular, we choose $t=t_{\epsilon}(m_{c})$, then according to the definition~\eqref{eqn:t_ep_delta}:
\[
\int_{\Nhood_\epsilon}\int^\infty_{m_c}m\dmuxmat{t_{\epsilon}(m_{c})}>e^{\epsilon (t_{\epsilon}(m_{c})-1)/2}\int_{\Nhood_\epsilon}\int^\infty_{m_c}m\dmuxmat{\eta_\epsilon}\geq 1\,.
\]
This contradicts the mass conservation:
\[
\int_{\Nhood_\epsilon}\int^\infty_{m_c}m\dmutxm\leq 1\,,\quad\forall t\geq 1\,.
\]
Hence there is a $t\leq t_{\epsilon}(m_c)$ so that $\fbarmut< \f_* + \epsilon$.
\end{proof}

\section{Numerical experiments}\label{sec:numerical_exp}
This section is dedicated to numerical results. We will showcase numerical evidence that validate theoretical statements on deriving the mean-field limit and the long time behavior of $\fbar^N$. Furthermore, we will also present some numerical findings that are yet to be understood theoretically.

All numerical experiments are ran on two classical non-convex functions: the Ackley type functions (in 1D, 2D and 10D) and the Rastrigin type functions in 1D, defined as follows
\begin{align*}
\f_{A}(\bx) & = -A\exp\left(-C\sqrt{\frac{1}{d}\sum_{i=1}^{d} x_{i}^{2}}\right) - \exp\left(\frac{1}{d}\sum_{i=1}^{d}\cos(2\pi x_{i})\right) + A + e + D,\\
\f_{R}(\bx) & = 10d + \sum_{i=1}^{d}[x_{i}^{2} - 10\cos(2\pi x_{i})] + D\,.
\end{align*}
Throughout the numerical section, we chose the following parameters \(A=20\), \(C=0.2\) and \(D=0\). The profiles of the two functions in 1D are presented in Figure \ref{fig:1D_exp}. They are both highly oscillatory, presenting abundant local minima. Moreover, the local minima for the Rastrigin function give very close function values to the global minimum, and can easily confuse samples on convergence.

\begin{figure}[ht]
\centering
\begin{subfigure}{.5\textwidth}
  \centering
  \includegraphics[width=1\linewidth]{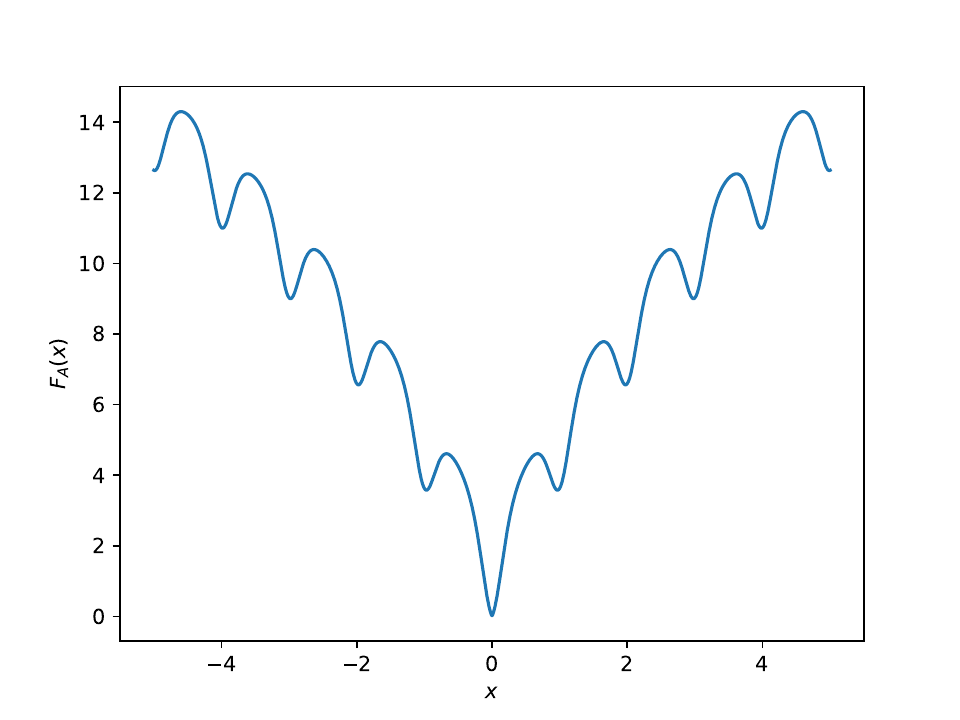}
  \caption{Ackley function.}
  \label{fig:Ackley_1D}
\end{subfigure}%
\begin{subfigure}{.5\textwidth}
  \centering
  \includegraphics[width=1\linewidth]{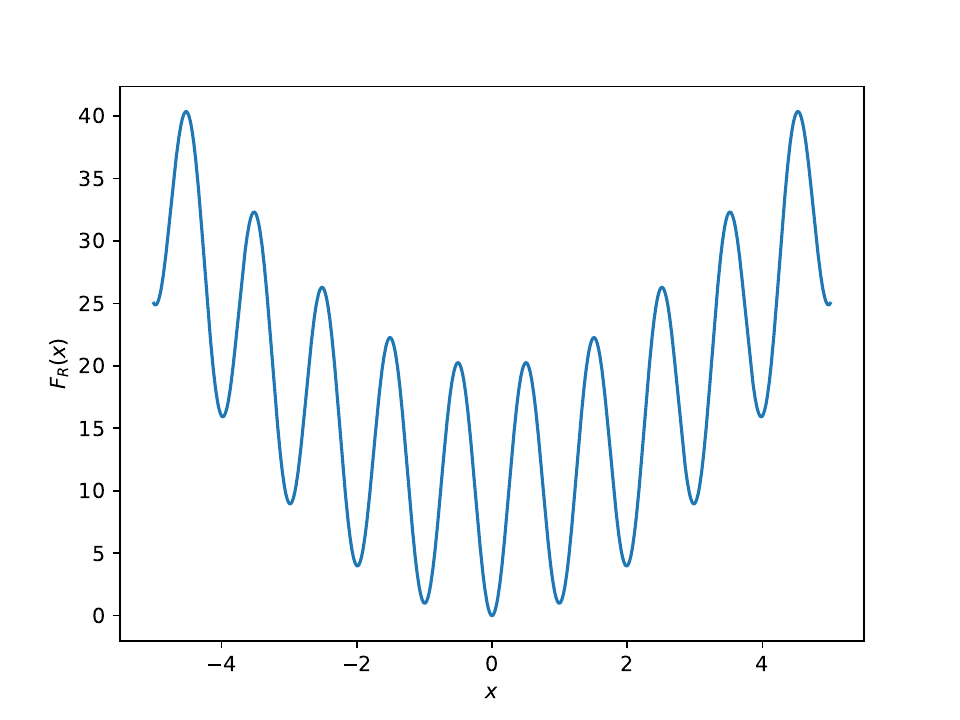}
  \caption{Rastrigin function.}
  \label{fig:Rastrigin_1D}
\end{subfigure}
\caption{Non-convex functions that will be used for the numerical examples when \(d=1\).}
\label{fig:1D_exp}
\end{figure}

For the numerical evidence in Section~\ref{sec:numerical_evidence} we will use the following $\gammab$ function in the application of Algorithm \ref{algo:SSA}:
\begin{equation}\label{eqn:gamma_smooth}
    \gammab_{\alphalam,\betamu}(m) = \begin{cases}
    \alphalam\exp\left(\frac{m}{m-\betamu}\right) & m < \betamu \\
    0 & m \geq \betamu
    \end{cases}\,.
\end{equation}
The smooth decay of the function (see Figure \ref{fig:gamma_smooth}) gives different magnitudes of noise to each particle depending on their mass, but eventually for sufficiently high mass the noise will disappear. We only present three examples in this section. Extensive numerical experiments are presented in Appendix~\ref{sec:more_num_exp}.

For the numerical study in Section \ref{sec:numerical_study} we will use the following $\gammab$ function so we have a more strict change in the application of the noise after certain amount of mass:
\begin{equation}\label{eqn:gamma_continuous}
    \gammab_{\alphalam,\betamu}(m) = \alphalam\left(-\frac{1}{2}\tanh(1000(m - \betamu)) + \frac{1}{2}\right)\,.
\end{equation}
The hyperbolic tangent form of the function is chosen to closely resemble a discontinuous cut-off function for samples whose weight is below $\betamu$ (see Figure \ref{fig:gamma_step}): this means particles with smaller weights are encouraged to explore the landscape of the objective with the same rate $\alphalam$ while all other heavier particles are given zero stochasticity. Experimentally we observe a very neat relation between $\betamu$ and $N$ for a consistent convergence rate, and we leave that to Section~\ref{sec:numerical_study}.

\begin{figure}[ht]
    \centering
    \begin{subfigure}{.49\textwidth}
    \includegraphics[width=1\linewidth]{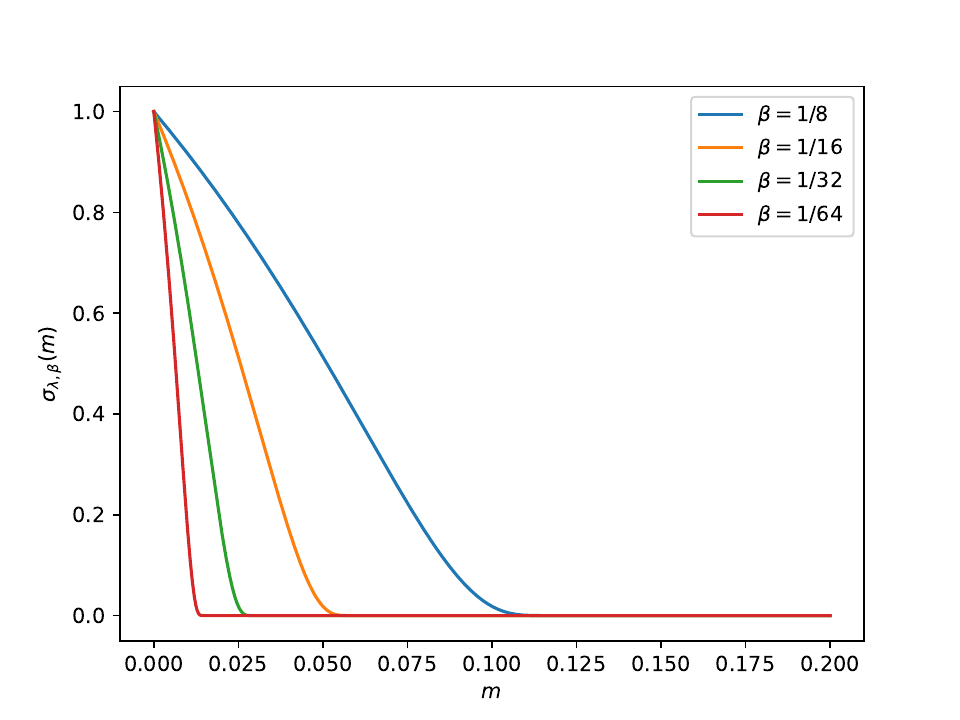}
    \caption{Smooth $\gammab$.}
    \label{fig:gamma_smooth}
    \end{subfigure}
    \begin{subfigure}{.49\textwidth}
    \includegraphics[width=1\linewidth]{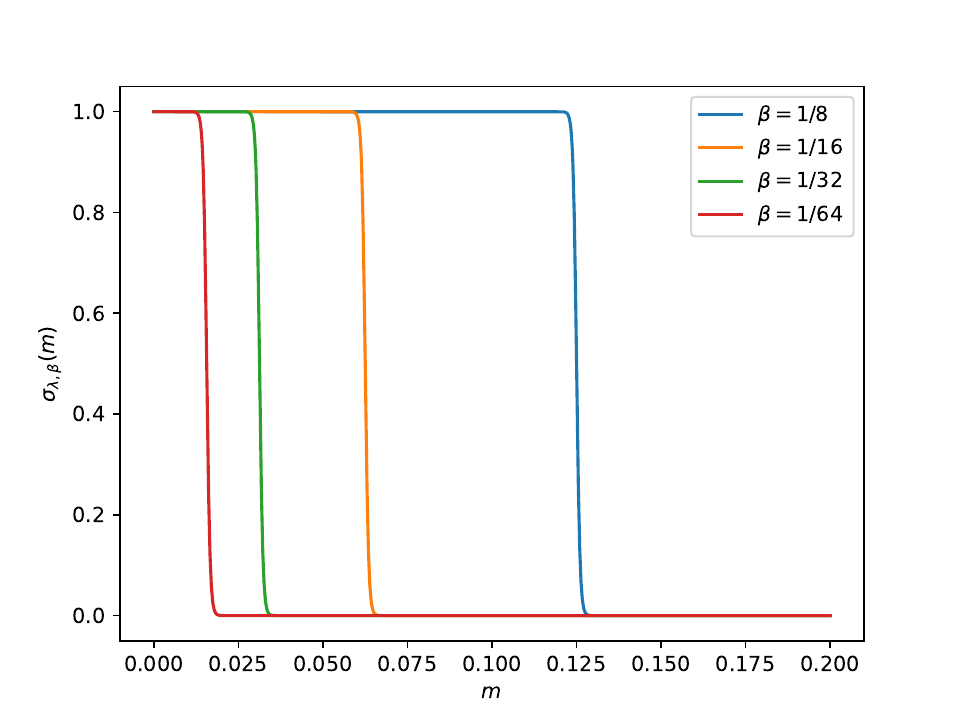}
    \caption{Step-like $\gammab$.}
    \label{fig:gamma_step}
    \end{subfigure}
    \caption{$\gammab$ function for different values of the parameter $\betamu$ and $\alphalam=1$.}
    \label{fig:gamma_func}
\end{figure}

To demonstrate the performance of the algorithm, we choose the experiments where the initial swarm does not cover the global minimum basin, to rule out the easy-to-converge situations. The convergence of the algorithm purely depends on the balancing between Brownian motion and local gradient descent through mass-transfer.

\subsection{Validation of theoretical results}\label{sec:numerical_evidence}

The first part of our simulation is to showcase the theoretical results we had for Algorithm~\ref{algo:SSA}. For this, we will apply the algorithm on the Ackley function for $d=1,\, 2$, then Rastrigin for $d=1$ and, finally to Ackley function with $d=10$. In all cases we will show the performance of the algorithm and check its dependence on $T$ and $N$.

\subsubsection{Ackley function in lower dimensions}
We experiment with the Ackley function for $d=1$ and $d=2$. First, for $d=1$, we pick $N=8$ particles, $\betamu=1/8$, $\alphalam=1$ and run the algorithm for step size $h=10^{-4}$ and $20000$ iterations for a total time of $T=2$. Figure~\ref{fig:Ackley_1D_results_1} demonstrates a few time instances of the particles: It is clear that they gradually move towards the global minimum. We then tested the convergence with $N=8,\, 16,\, 32,\, 64$. Each run of the algorithm provides one trajectory of $\fbarNt$, and in simulation we ran the algorithm $20$ times and compute its expectation, first and third quartiles. The solid lines in Figure~\ref{fig:Ackley_1D_results_2} demonstrate the convergence of $\mathbb{E}[\fbarNt]$ in time, and the shaded areas shows the difference between the first and third quartiles.

\begin{figure}[ht]
    \centering
    \includegraphics[scale = 0.33]{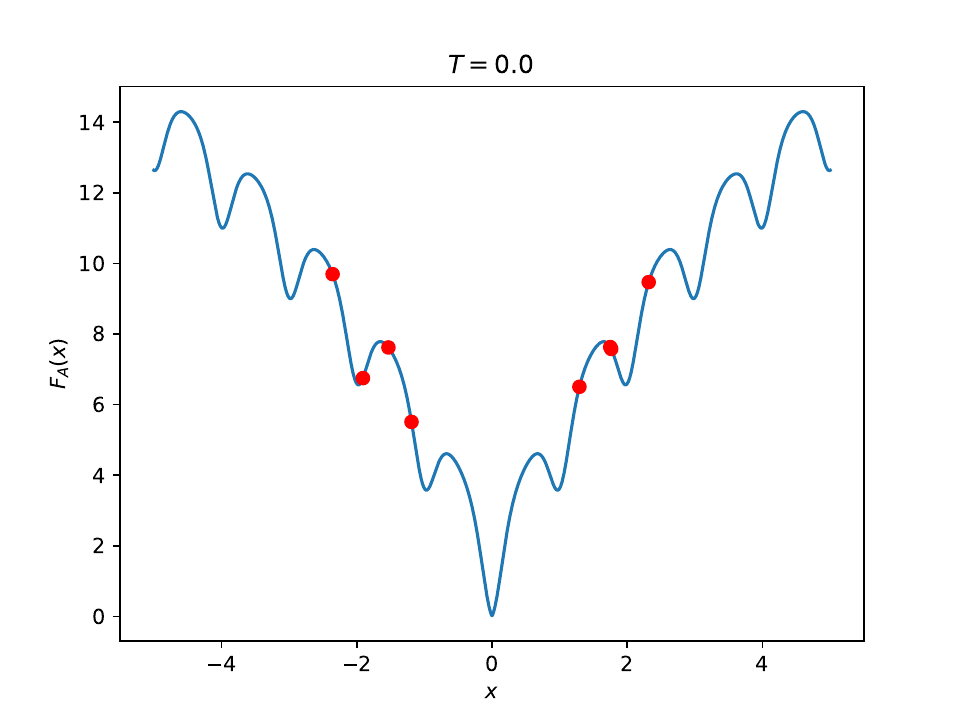}
    \includegraphics[scale = 0.33]{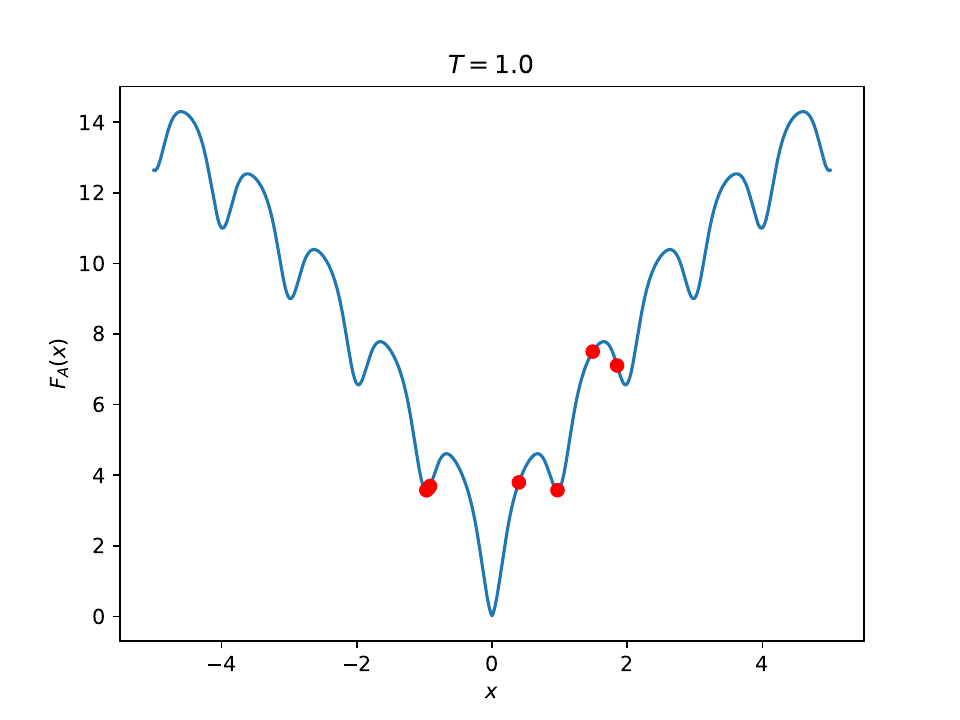}
    \includegraphics[scale = 0.33]{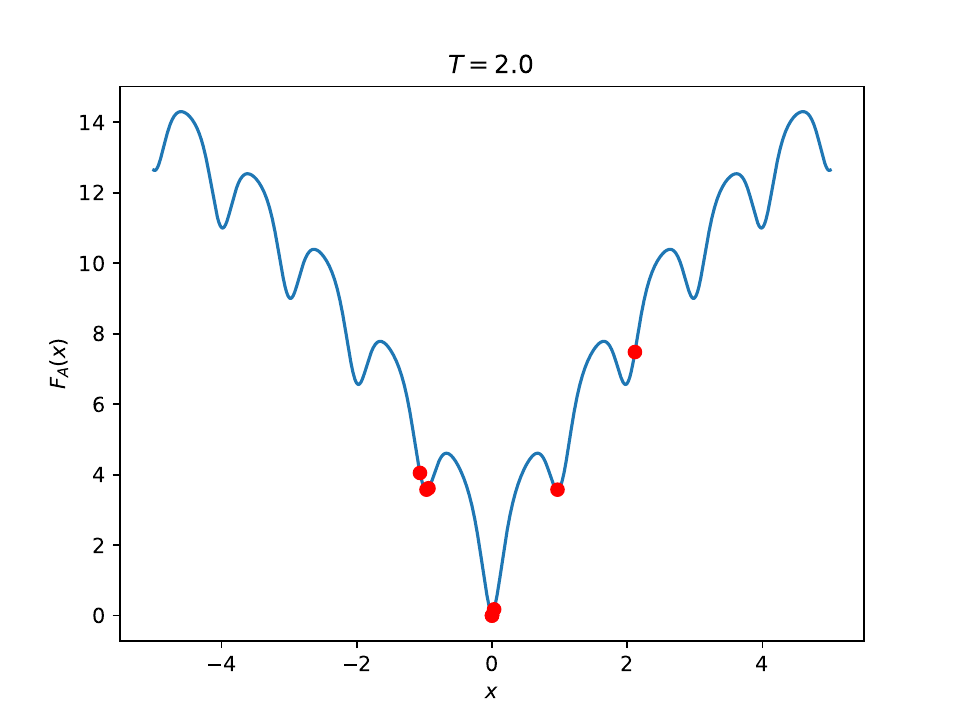}
    \caption{Swarm movement (in red) for Ackley function in $d=1$.}
    \label{fig:Ackley_1D_results_1}
\end{figure}

\begin{figure}[ht]
    \centering
    \includegraphics[scale = 0.75]{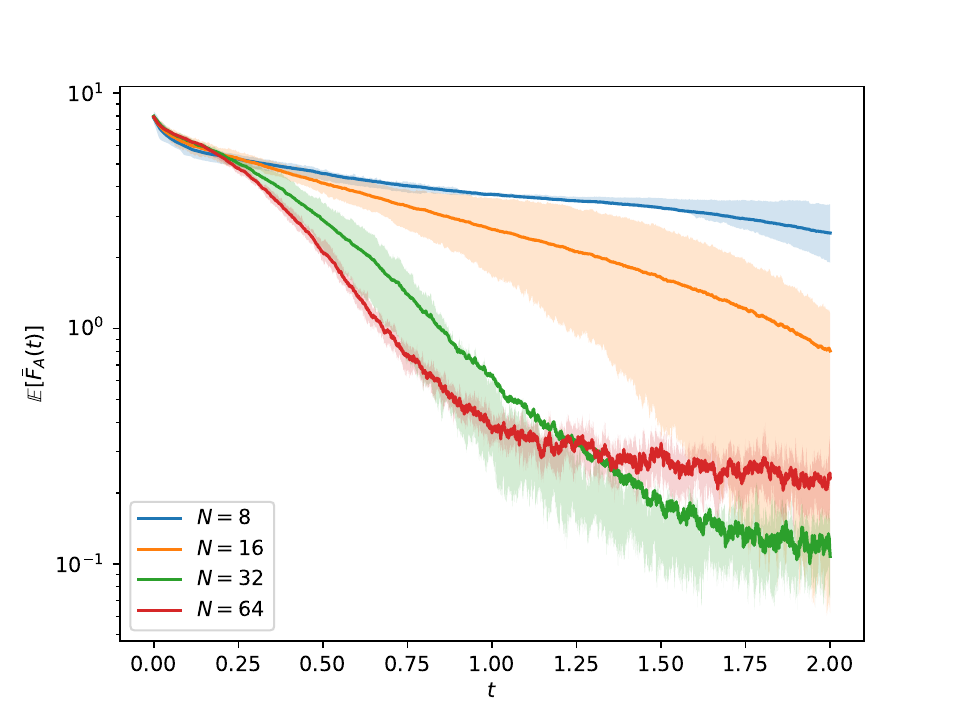}
    \caption{Expectation of $\fbarNt$ for different values of $N$ for Ackley function in $d=1$ in log-scale. The shaded area shows the difference between the lower quartile and the higher quartile for $\fbarNt$.}
    \label{fig:Ackley_1D_results_2}
\end{figure}

For $d=2$, we pick $N=16$ particles, $\betamu=1/16$, $\alphalam=1$ and we run the algorithm for step size $h=10^{-4}$ and $20000$ iterations for a total time of $T=2$. In Figure~\ref{fig:Ackley_2D_results_1} we visualize the convergence of particles to the global minimum in a few time instances. We repeat the experiment for $d=1$ case and compute the mean, first and third quartiles of $\fbar^N$ for different choices of $N$, and they are plotted in Figure~\ref{fig:Ackley_2D_results_2}.

\begin{figure}[ht]
        \centering
        \includegraphics[scale = 0.33]{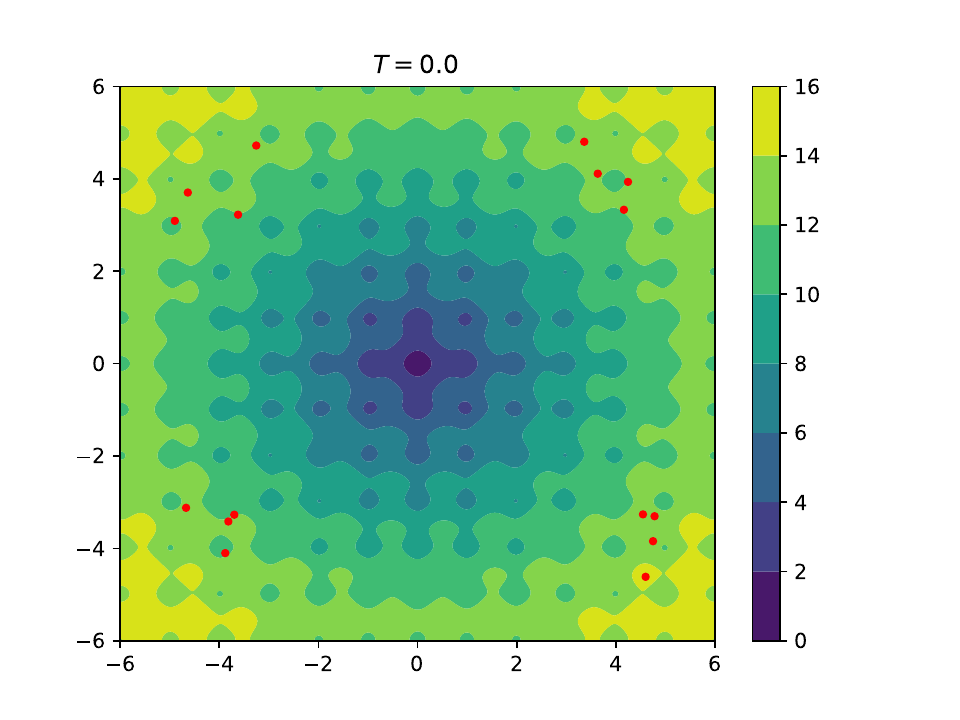}
        \includegraphics[scale = 0.33]{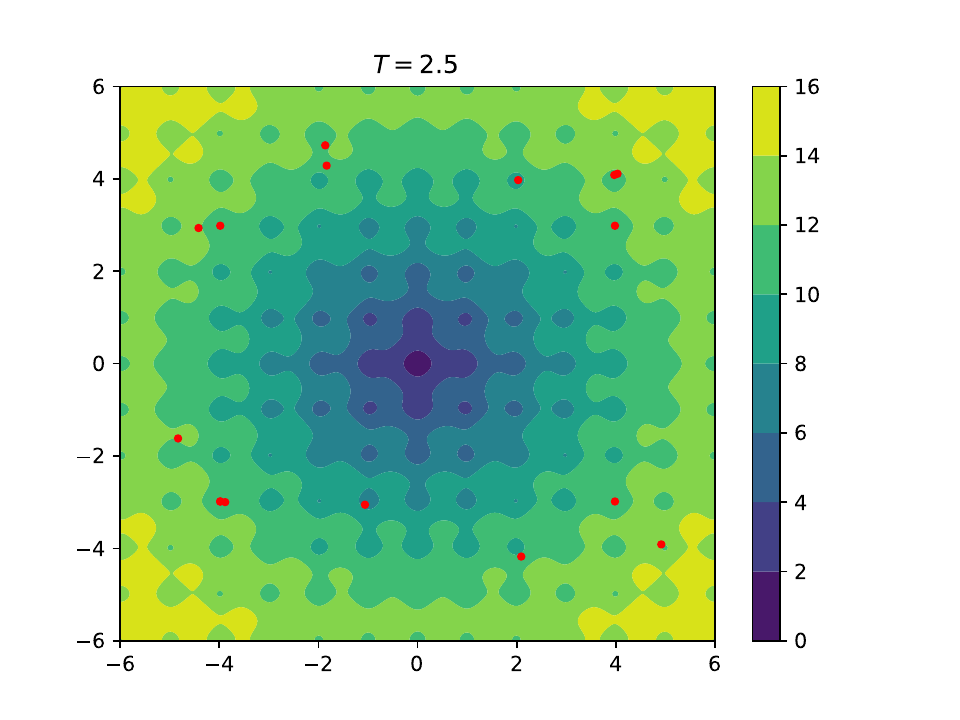}
        \includegraphics[scale = 0.33]{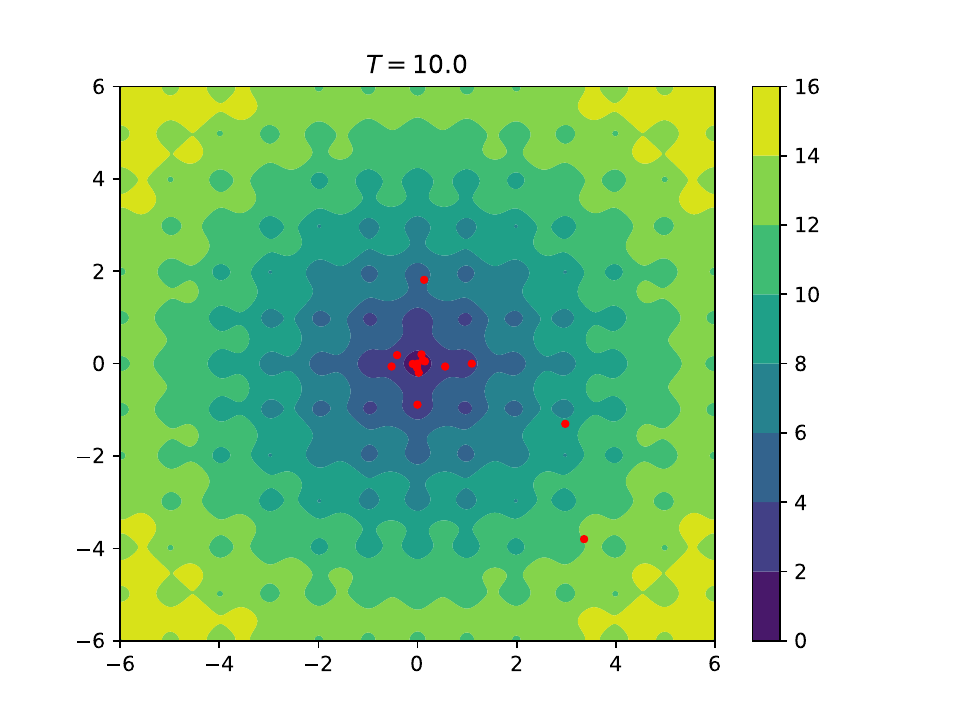}
    \caption{Swarm movement (in red) for Ackley function in $d=2$.}
    \label{fig:Ackley_2D_results_1}
\end{figure}

\begin{figure}[ht]
    \centering
    \includegraphics[scale = 0.75]{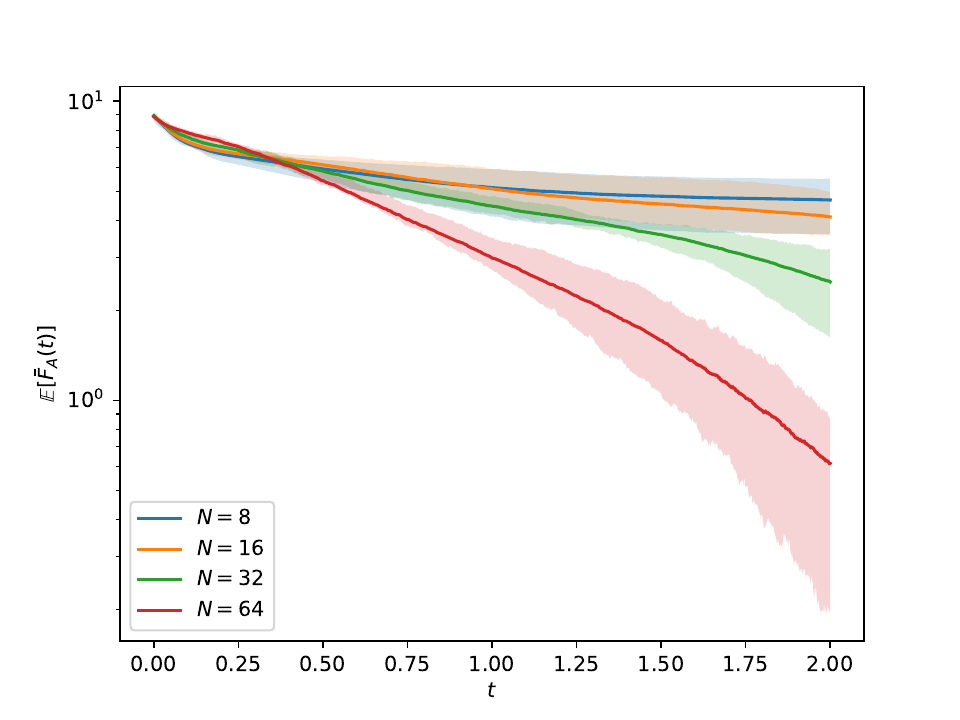}
    \caption{Expectation of $\fbarNt$ for different values of $N$ for Ackley function in $d=2$ in log-scale. The shaded area shows the difference between the lower quartile and the higher quartile for $\fbarNt$.}
    \label{fig:Ackley_2D_results_2}
\end{figure}

\subsubsection{Rastrigin function in low dimension}\label{sec:Rastrigin}
The third test is conducted for the Rastrigin function for $d=1$. The choice of parameters are $N=8$, $\betamu = 1/8$ and $\alphalam=2$. The step-size is $h=10^{-4}$ and we run it up to $T=500$.

In Figure~\ref{fig:Rastrigin_1D_results_1} we can see that the particles eventually find the global minimum after a long time. The convergence of $\fbarNt$, however, is significantly faster than finding the global minimum. We also evaluate the mean-field convergence. Running the experiment for $20$ runs, we obtain 20 trajectories of $\fbarNt$. We compute the mean, the first and third quartiles, and plot them in Figure~\ref{fig:Rastrigin_1D_results_2}. It is clear that bigger $N$ gives a higher chance of achieving a lower value for $\f$ in time.

\begin{figure}[ht]
    \centering
    \includegraphics[scale = 0.33]{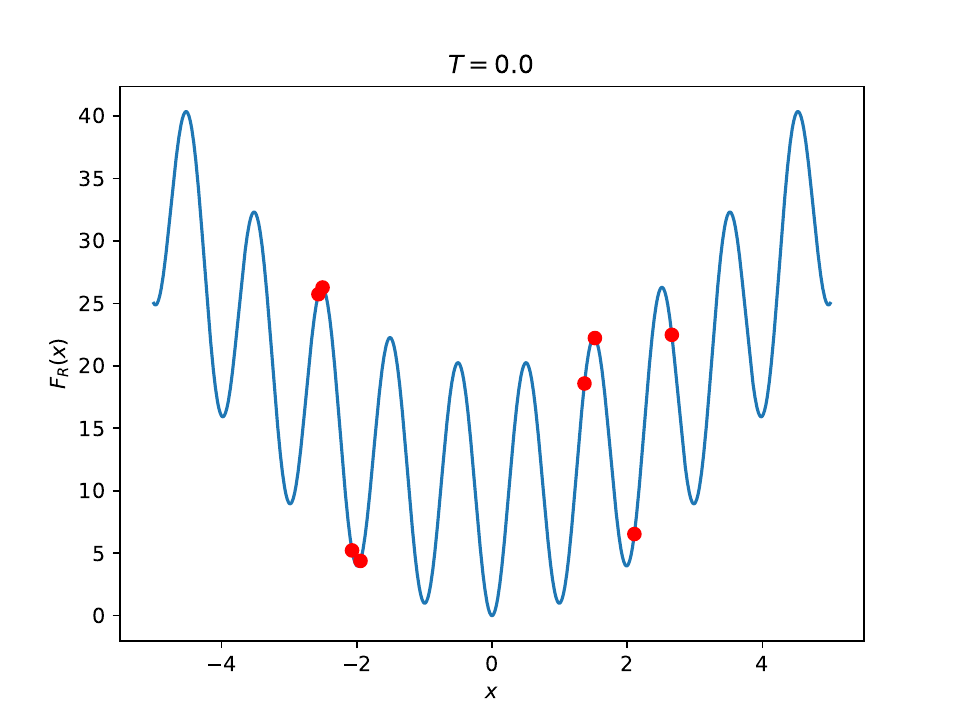}
    \includegraphics[scale = 0.33]{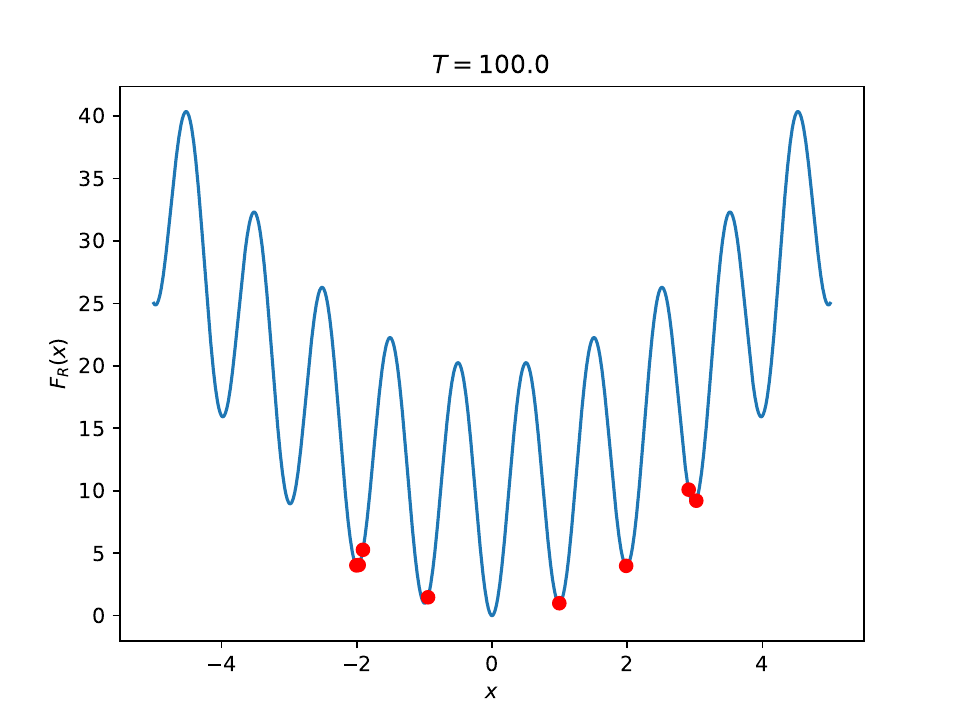}
    \includegraphics[scale = 0.33]{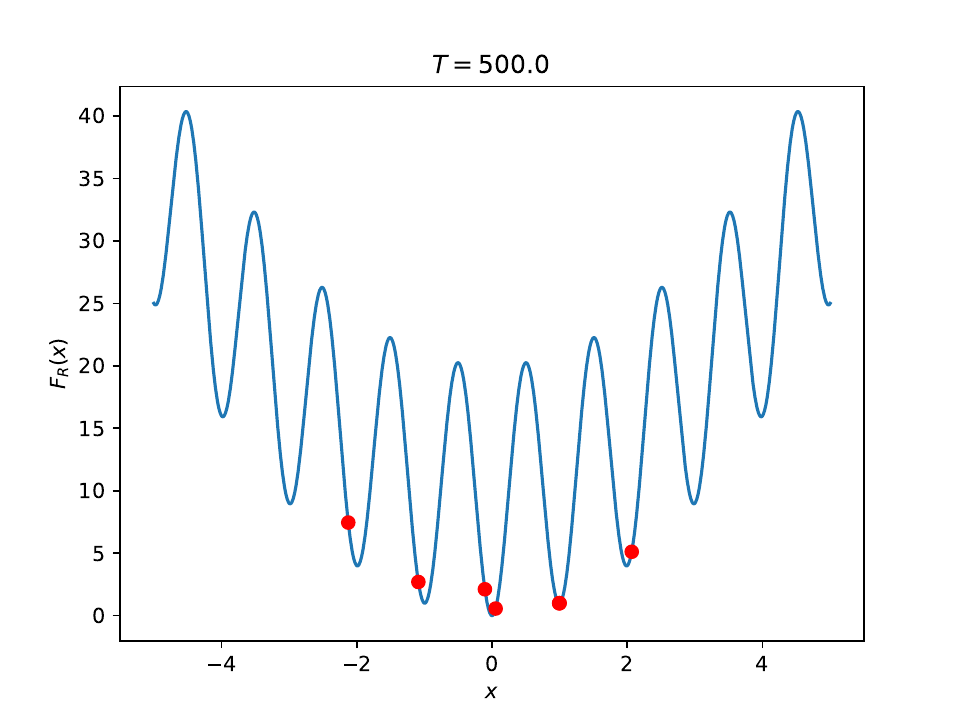}
    \caption{Swarm movement (in red) for Rastrigin function in $d=1$.}
    \label{fig:Rastrigin_1D_results_1}
\end{figure}

\begin{figure}[ht]
    \centering
    \includegraphics[scale = 0.75]{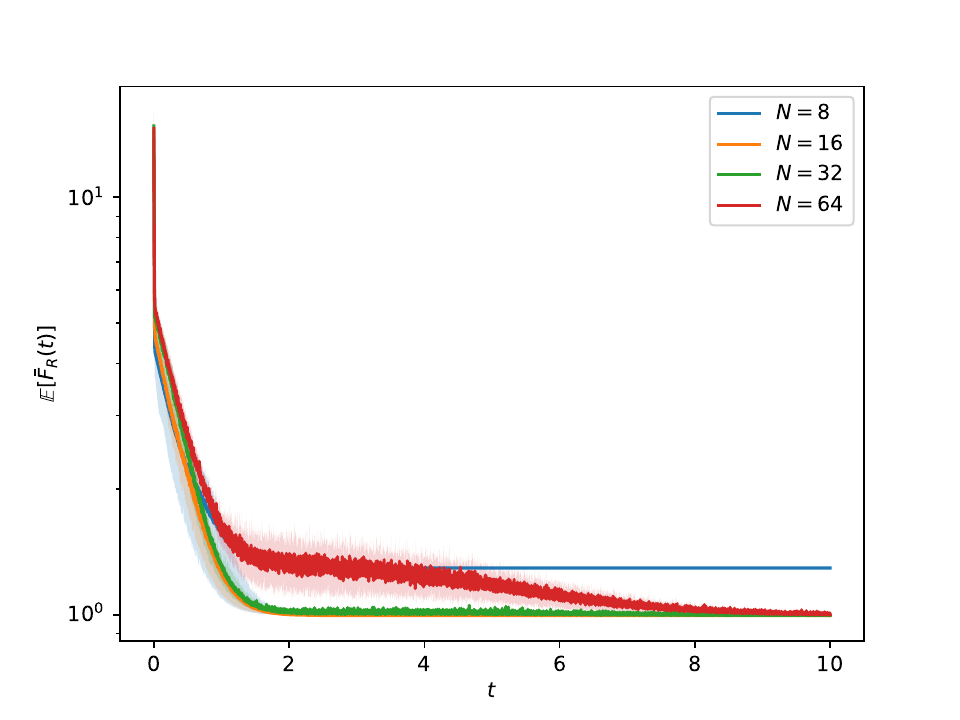}
    \caption{Expectation of $\fbarNt$ for different values of $N$ for Rastrigin function in $d=1$ in log-scale. The shaded area shows the difference between the lower quartile and the higher quartile for $\fbarNt$.}
    \label{fig:Rastrigin_1D_results_2}
\end{figure}

\subsubsection{Ackley function in higher dimensions}
The fourth experiment is to evaluate the performance of the algorithm on a high dimensional function (Ackley function for $d=10$). The hope is the Brownian motion in the dynamics can help overcoming the curse of dimensionality. We set the parameters to be $\alphalam = 1$, $\betamu = 1/N$. In Figure~\ref{fig:Ackley_10D_results_1}, we show the convergence in time of $20$ runs of experiments using Algorithm~\ref{algo:SSA}, and compare it to the solution to the deterministic system shown in~\eqref{eqn:deterministic_system}. The observation is that the deterministic system converges very quickly at the initial steps, but it saturates at local minima, and the global optimization is never achieved. On the contrary, Algorithm~\ref{algo:SSA}, using either $2000$ or $200$ samples both achieve the global minimum, albeit with slower convergence rate. We attribute this to the randomness in the method and this observation confirms Theorem~\ref{thm:small_probability}. We also emphasize that the performance of the algorithm using smaller number of samples ($200$) is compatible to that using $2000$, so a small number of samples already have demonstrated impressive numerical performance in this example.

\begin{figure}[ht]
    \centering
    \includegraphics[scale = 0.75]{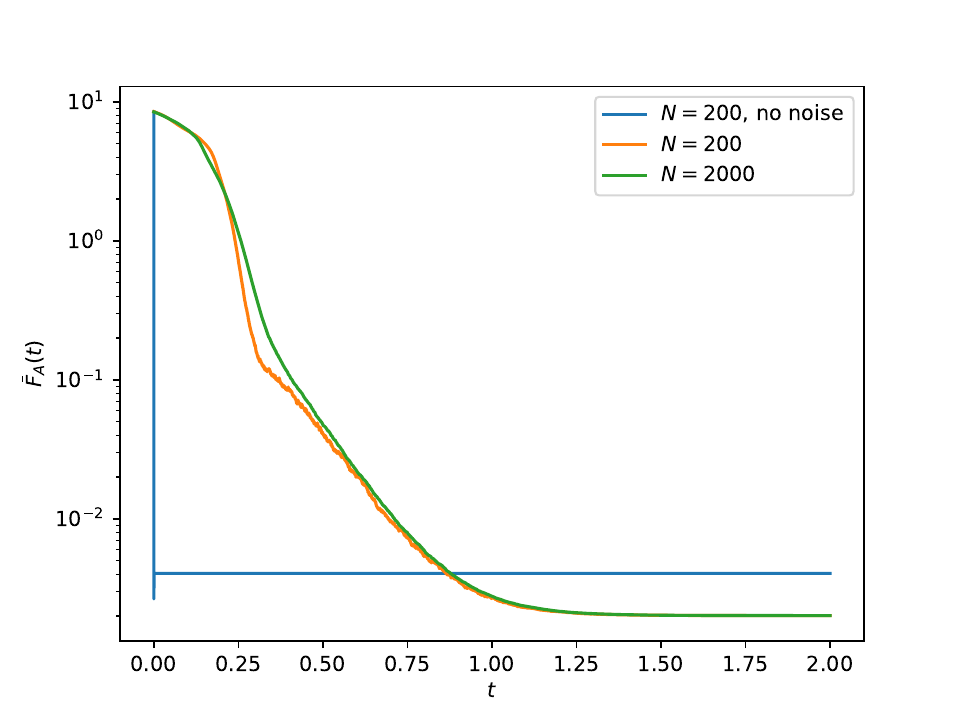}
    \caption{Long time behavior of $\fbarNt$ for different values of $N$ for Ackley function in $d=10$, in log-scale, in the absence and presence of noise.}
    \label{fig:Ackley_10D_results_1}
\end{figure}

\subsection{Numerical study}\label{sec:numerical_study}
Though not the focus of the paper, we observe, through running numerical simulations, some interesting algorithmic performance. We collect these findings in this subsection. We first present studies on parameter-tuning. As shown earlier, the performance of the algorithm depends on the choice of the $\gammab$ function that relies on the configuration of the following three parameters: $\alphalam$, $\betamu$ and $N$. While $\alphalam$ determines the strength of the Brownian motion, the cut-off point $\betamu$ and its relation to $N$ fully characterize the percentage of particles that are allowed to ``roam.'' To fully understand the relation, we run Algorithm~\ref{algo:SSA} on the Ackley functions for $d=1$ up to $T=1$ with $\gammab$ defined as in~\eqref{eqn:gamma_continuous} for $\alphalam=1$ using many choices of $(N,\betamu)$. To be more specific, we choose $N=8,\, 16,\, 32,\, 64$ and $\betamu=1/8,\, 1/16,\, 1/32,\, 1/64$ and for each $(N,\betamu)$ configuration, we run the experiment $40$ times for a collection of $40$ trajectories of $\fbarNt$. The mean and variance of these $\fbarNt$ for different $(N,\betamu)$ choices are presented, on log-scale, in Figure~\ref{fig:numerical_study}. We note that experimental-wise, it seems that $N=2/\betamu$ seems to give the most ideal convergence. Another observation is that the performance of the algorithm is consistent across the cases when $\betamu N$ are set to be constants. Namely, the convergence rate seem to agree along diagonals in Figure~\ref{fig:numerical_study}.

\begin{figure}[ht]
    \centering
    \includegraphics[scale=0.32]{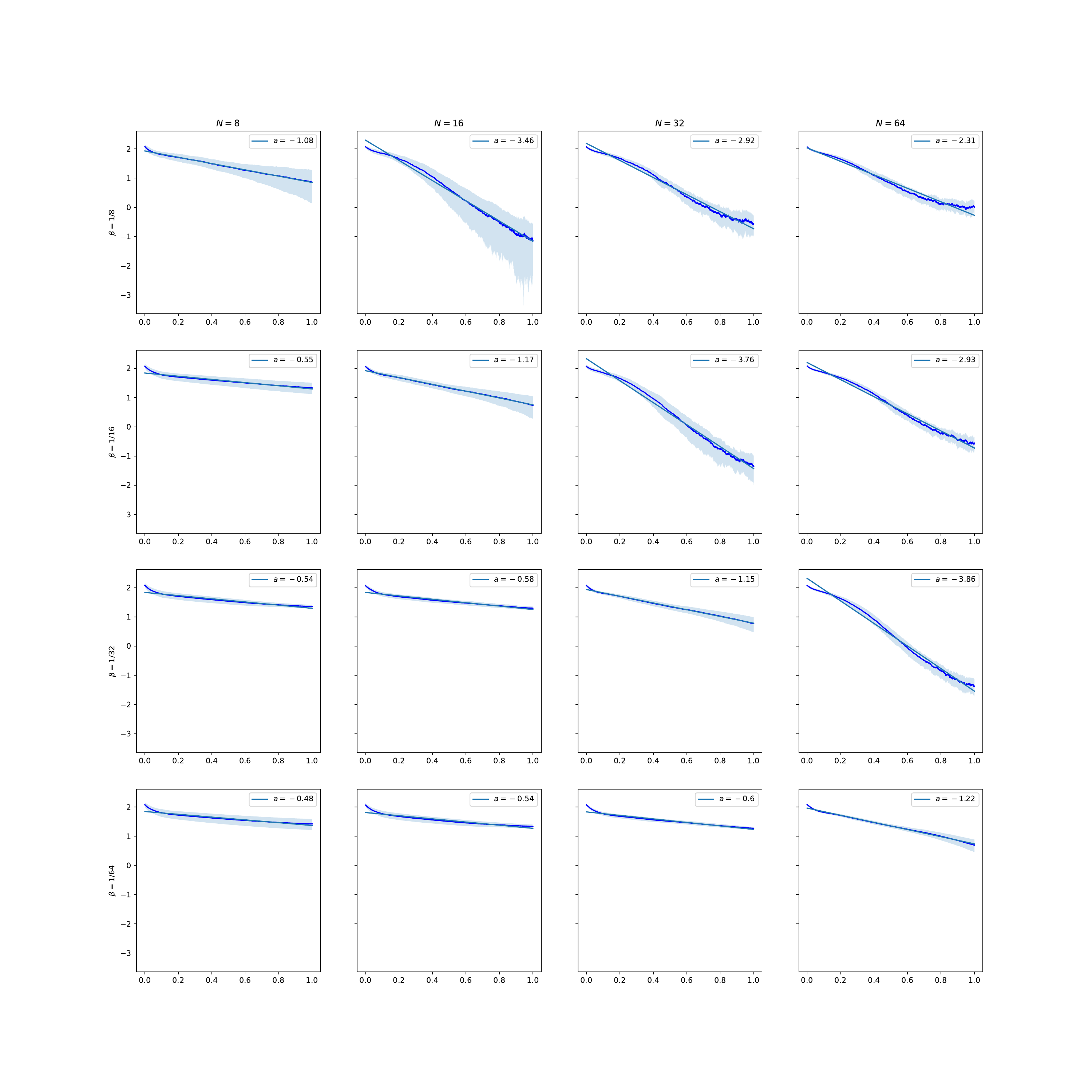}
    \caption{Convergence rates for the Ackley functions for different values of $N$ and $\beta$.}
    \label{fig:numerical_study}
\end{figure}

We then present the long time behavior of the algorithm performed on a variation of the Rastrigin function from Section~\ref{sec:Rastrigin}:
\begin{equation}\label{eqn:modified_Rastrigin}
    \f_{R}(\bx) = 10d + \sum_{i=1}^{d}[2 x_{i}^{2} - 10\cos(2\pi x_{i})] + D\,,
\end{equation}
with parameters chosen to be $d=1$ and $D=0$, and the same choices of $\{x_i\}$. In this case, we run the simulation for a very long time till $t\sim 700$. As seen in Figure \ref{fig:long-time_rastrigin}, the error seems to saturate at around $O(1)$ after $t=10$, and the second dip occurs at $t=690$, dragging the error down to $0$. The convergence takes place in a staggered pattern. The conjecture is that the value of the local minimum of the Rastrigin function is very close to that of the global one, and the samples stuck at the local minimum for a very long stretch of time, accumulating all the mass. The lighter ``explorer'' uses this time to explore and find the global minimum. Before the very small basin gets found, local minimum takes the lead and the error stays at $O(1)$, and the second dip occurs when the global basin finally was found by one explorer, and the mass gets transferred to it. We find this staggered pattern enlightening, and future research is needed for spelling out the explicit dependence of this behavior on different parameters.

\begin{figure}[ht]
    \centering
    \begin{subfigure}{.3\textwidth}  \includegraphics[width=1\linewidth]{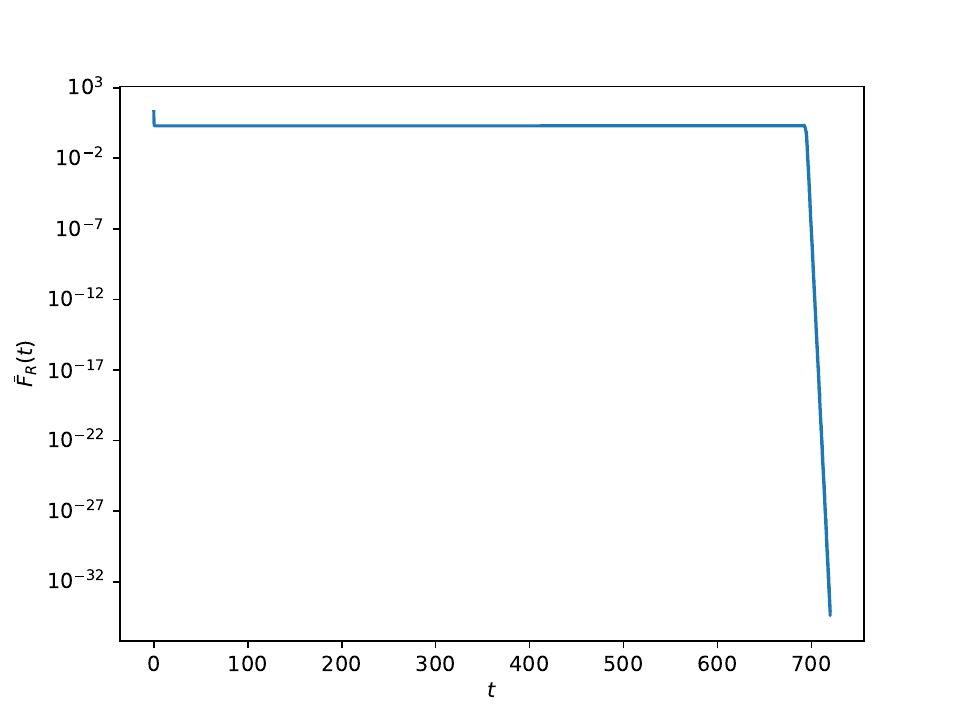}
    \caption{$t\in[0,700]$.}
    \label{fig:long-time_rastrigin700}
    \end{subfigure} 
    \begin{subfigure}{.3\textwidth}\includegraphics[width=1\linewidth]{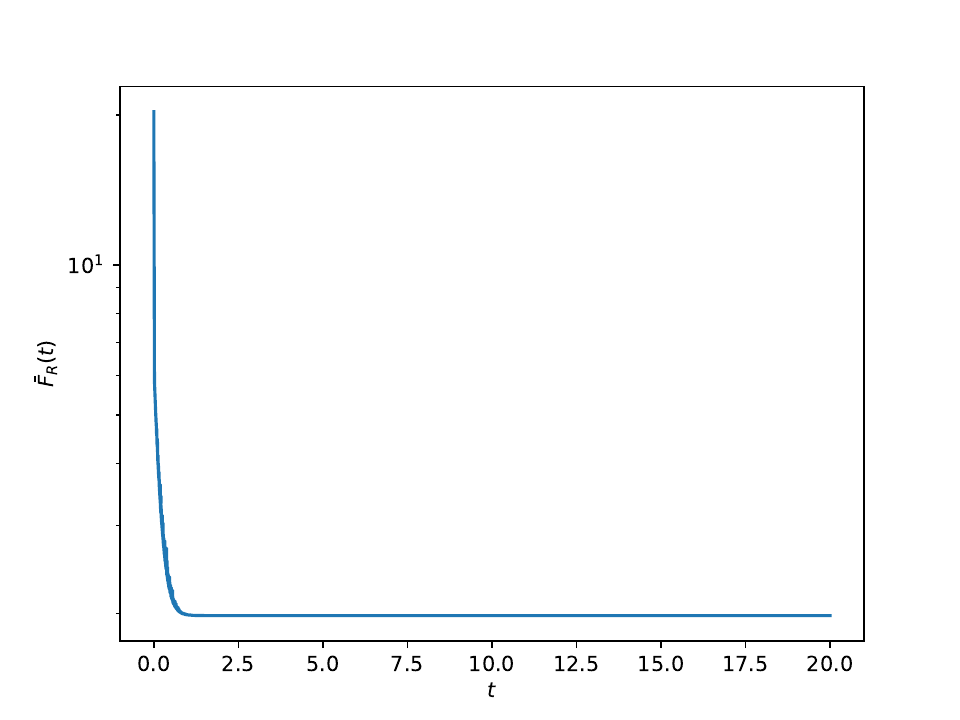}
    \caption{Zoom-in at $t\in[0,20]$.}
    \label{fig:long-time_rastrigin20}
    \end{subfigure}
    \begin{subfigure}{.3\textwidth}  \includegraphics[width=1\linewidth]{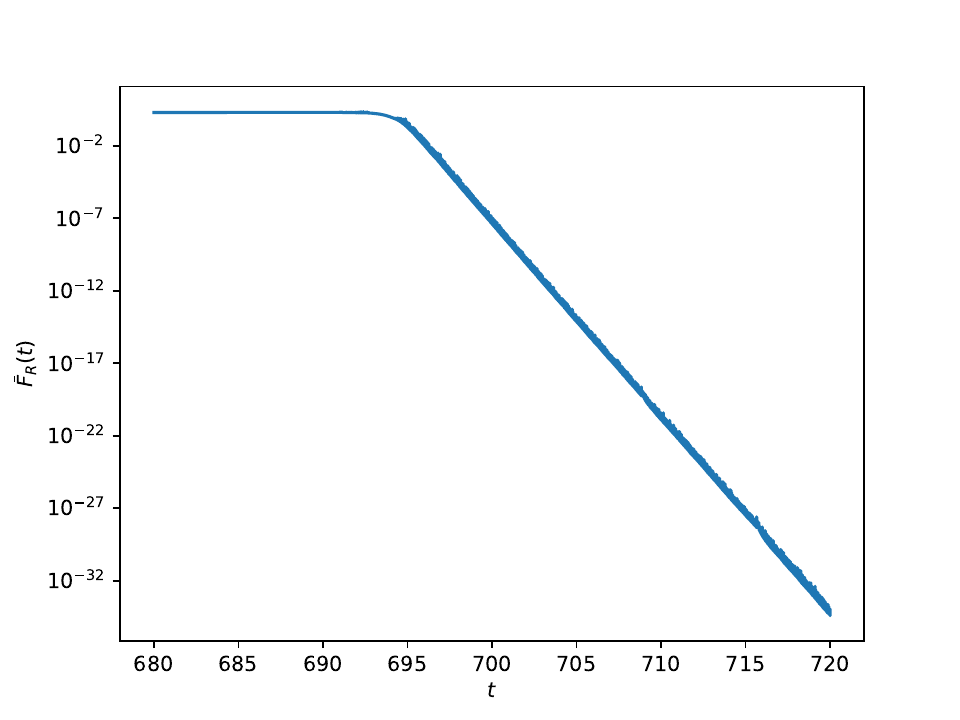}
    \caption{Zoom-in at $t\in[680,720]$.}
    \label{fig:long-time_rastrigin720}
    \end{subfigure}
      \caption{Long time behavior of $\fbarNt$ for~\eqref{eqn:modified_Rastrigin}.}\label{fig:long-time_rastrigin}
\end{figure}

\noindent
\section*{Acknowledgments} ET  thanks the Fondations Sciences Math\'{e}matiques des Paris  (FSMP) and LJLL at the Sorbonne University for the support and hospitality. All four authors thank the two anonymous reviewers for constructive and insightful suggestions.

\bibliographystyle{plain}
\bibliography{references-sbgd-sa-v2}

\newpage
\appendix
\numberwithin{equation}{section}
\section{Interacting vs. non-interacting agents}

In this appendix, we investigate the effect of communication, and study the behavior of two systems that produce pitfalls. In subsection~\ref{sec:pf_of_thhm_1} we study basic properties of the deterministic system~\eqref{eqn:deterministic_system} which governs non-interacting  agents. In subsection~\ref{sec:langevin}, we switch gears to interacting agents, governed by the stochastic system~\eqref{eqn:langevin_system} with mass communication passively adjusted according to $\bx$-values, and prove Theorem~\ref{thm:no_gamma}.

\subsection{Properties of the deterministic system (non interacting agents). }\label{sec:pf_of_thhm_1}
\begin{proposition}\label{thm:deterministic} Given the \(N-\)particle system defined in \eqref{eqn:deterministic_system}, we have the provisional minimum $\fbarN(t)$ always decay in time\footnote{Note that the time dependent quantities in the deterministic case are denotes $\square(t,\cdots)$, e.g., $\bxi(t), \fbarN(t), \nuNtx$ etc.}:
    \begin{equation}\label{eqn:bar_f_dec_determin}
    \ddt\fbarN(t)\leq 0\,,\quad\forall t\geq 0\, \qquad \fbarN(t)= \frac{1}{M(0)}\sum_{j=1}^N m^j(t) \f(\bxi(t)).
    \end{equation}
\end{proposition}
\begin{proof}
This comes from direct computation:
\[
\begin{aligned}
\ddt\fbarN(t)&=\frac{1}{M(0)}\left(\sum^N_{j=1}\left(\ddt \mi(t)\right)\f(\bxi(t))+\mi(t)\ddt \f(\bxi(t))\right)\\
&=\frac{1}{M(0)}\left(-\sum^N_{j=1}\mitt\left(\f(\bxi(t))-\fbarN(t)\right)\f(\bxi(t))-\mi(t)\left|\nabla \f(\bx^j(t))\right|^2\right)\\
&=-\frac{1}{M(0)}\sum^N_{j=1}\mitt\left(\f(\bxi(t))-\fbarN(t)\right)^2-\frac{1}{M(0)}\sum^N_{j=1}\mitt\left|\nabla \f(\bxi(t))\right|^2\leq 0\,.
\end{aligned}
\]
This concludes the proof.
\end{proof}

Similar to its stochastic counterpart, we can also derive the mean-field limiting equation in the deterministic case. This is expressed in terms of the limiting distribution $\nutx$. To be specific, we have the following proposition:
\begin{proposition}
    Let $\{\bxi(t),\mi(t)\}_{j=1}^N$ denote the crowd of agents satisfying the deterministic system \eqref{eqn:deterministic_system} subject to initial data $\{\bxi(0), \mi(0)=\frac{1}{N}\}^N_{j=1}$, independently drawn from random distribution $\rho_0$, and let $\nuNtx$ denote its empirical distribution 
\begin{equation}\label{eqn:dist_deterministic}
  \nuNtx = \sum_{j=1}^{N}\mi(t)\delta_{\bxi(t)}(\bx).
\end{equation}
Then its mean-field limit of the ensemble distribution, $\nuNtx \rightarrow \nutx$,  satisfies
    \begin{equation}\label{eqn:Fokker-Planck_det}
    \partial_{t}\nutx = \nabla_{\bx}\cdot(\nutx \nabla \f(\bx)) - (\f(\bx) - \fbarnut)\nutx,
\end{equation}
subject to $\nu(0,\cdot)=\rho_0(\cdot)$. Here $\fbarnut$ is the weighted average is:
\begin{equation}\label{eqn:weighted_ave_cont}
    \fbarnut = \frac{\ds \int \f(\bx) \nd\nutx}{\ds \int \nd\nutx}\,.
\end{equation}
\end{proposition}
\begin{proof}
To derive the mean-field limit, we test the system with a test smooth function \(\phi\):
\[\mathbb{E}_{\nut}[\phi] = \int \phi(\bx)\nd\nutx = \sum_{j=1}^{N}\mi(t)\phi(\bxi(t))\,.\]
We take the time derivative on the two sides of the equation. The left hand side provides:
\begin{equation}\label{eqn:ensemble_det_1}
    \ddt\mathbb{E}_{\nut}[\phi] = \int \partial_{t}\nutx\phi(\bx)\mathrm{d}\bx.
\end{equation}
Here  expectation is with respect to the randomly drawn initial data.  
The time derivative of the right hand side yields
\begin{equation}\label{eqn:ensemble_det_2}
\begin{split}
    \ddt\Bigg(\sum_{j=1}^{N} & \mi(t)\phi(\bxi(t))\Bigg) =\\
     &  -\sum_{j=1}^{N}\mitt\underbrace{\nabla_{\bx}\phi(\bxi(t))\cdot\nabla\f(\bxi(t))}_{\dot{\bx}=-\nabla \f(\bx(t))} - \sum_{j=1}^{N}\phi(\bxi(t))\underbrace{\mi(t)(\f(\bxi(t)) - \fbar(t))}_{\dot{m}=-m(t)(\f(\bx(t)-\fbarnut)},
     \end{split}
\end{equation}
where we integrated the ODE \eqref{eqn:deterministic_system}. Note that the first term in \eqref{eqn:ensemble_det_2} can be written as
\[-\sum_{j=1}^{N}\mi(t)\nabla_{\bx}\phi(\bxi(t))\cdot\nabla\f(\bxi(t)) = -\mathbb{E}_{\nut}[\nabla_{\bx}\phi\cdot \nabla\f] = \int \phi(\bx)\nabla_{\bx}\cdot(\nutx\nabla\f(\bx))\dx\]
using integration-by-parts, and the second term of \eqref{eqn:ensemble_det_2} can be written as
\begin{align*}
\sum_{j=1}^{N}\phi(\bxi(t))\mi(t)(\f(\bxi(t)) - \fbarnut) & = \mathbb{E}_{\nut}[\phi \f] - \mathbb{E}_{\nut}[\phi]\mathbb{E}_{\nut}[\f] \\
& = \int \nutx\phi(\bx)\f(\bx)\dx - \fbarnut\int\phi(\bx)\nu(\bx)\dx.
\end{align*}
Plugging the latter into \eqref{eqn:ensemble_det_2} and combining it with \eqref{eqn:ensemble_det_1} we recover \eqref{eqn:Fokker-Planck_det}.
\end{proof}
Some properties of this continuous limit are straightforward. In particular, obvious observations made  earlier in our discussion yield the following.
\begin{proposition}
Let $\nu$ be the mean-field limit satisfying \eqref{eqn:Fokker-Planck_det}. There holds:\newline
  $\bullet$ {Mass conservation}:
    the total mass $\ds M(t) := \int \nutx\dx$, is conserved in time,
    $\ddt M(t) = 0$; \newline
 $\bullet$ {Decrease of provisional mean}: the weighted mean \(\fbarnut\) decreases in time, $\ddt\fbarnut \leq 0$.
\end{proposition}

\subsection{Proprieties of the stochastic system (interacting agents)}\label{sec:langevin}
We turn to the proof of Theorem \ref{thm:no_gamma}.
\begin{proof}
 We may assume  without loss of generality  that $\f_{*}=0$, and that the total  mass is normalized 
\[
\iint m\dmuxmat{0}=1,
\]
so that $\ds \rho_t(\cdot)=\int m\dmutm$ can be viewed as a probability density function in $\bx$. 
Multiplying~\eqref{eqn:F-P_ex_no_gamma} by $m\f(\bx)$ and integrate, we obtain that, with a few integration by parts:
\[
\ddt\fbarmut=-\iint m|\nabla \f|^2\dmutxm-\iint m(\f-\fbarmut)^2\dmutxm+\iint m\Delta \f\dmutxm\,,
\]
and rewriting it using $\rho$,
\[
\ddt\fbarmut=-\mathbb{E}_{\rho}(|\nabla \f|^2)-\mathrm{Var}_{\rho}(\f)+\mathbb{E}_{\rho}(\Delta \f)\,.
\]
Because $\nabla_\bx \f$ is $L$-Lipschitz according to the first condition in Assumption \ref{assumption: assmptn2} and $\f_*=0$, we obtain
\[
|\nabla \f|^2\leq 2L\f(\bx)\,.
\]
This implies the first term can be lower bounded by the expectation of $\f(\bx)$:
\begin{equation}\label{eqn:first_term_lower_b0und}
-\mathbb{E}_{\rho}(|\nabla \f|^2)\geq -2L \fbarmut\,.
\end{equation}

Since $\bx_\ast$ is the unique global minimal, there exists $\epsilon'>0$ such that
\[
\f^{-1}([0,\epsilon'])\subset B_R(\bx_\ast)\,.
\]
Now, we lower bound the last two terms:
\begin{itemize}
    \item Because $\f\leq D$ according to the second condition in Assumption \ref{assumption: assmptn2} and $\f_*=0$, we obtain
    \[
    \mathrm{Var}_{\rho}(\f)= \mathbb{E}_{\rho}(\f^2)-(\fbarmut)^2\leq D\fbarmut\left(1-\frac{\fbarmut}{D}\right)\leq D\fbarmut\,,
    \]
    where we use Assumption~\ref{assumption: assmptn2}, $0\leq \f \leq D$ in the inequalities. This implies \begin{equation}\label{eqn:second_term_lower_b0und}
    -\mathrm{Var}_{\rho}(\f)\geq -D\fbarmut\,.
    \end{equation}
    \item Because $\f(x)$ is $\xi$-strongly convex in the ball of $B_R(\bx_*)$, we obtain that
    \[
    \Delta \f\geq \xi d,\quad \forall \bx\in B_R(\bx_*)
    \]
    and because $\nabla \f$ is $L$-Lipschitz, we have:
    \[
    \left|\Delta \f\right|\leq dL\,,\forall \bx\in\mathbb{R}^d\,.
    \]
    To lower bound the last term, we first deploy Markov inequality:
    \[
    \mathbb{P}(\bx\in B_R(\bx_\ast))\geq\mathbb{P}(\bx\in \f^{-1}([0,\epsilon']))=1-\mathbb{P}(\f(\bx)\geq \epsilon')\geq 1-\frac{\fbarmut}{\epsilon'}\,.
    \]
    Then, we have
\begin{equation}\label{eqn:third_term_lower_b0und}
    \mathbb{E}_{\rho}(\Delta \f)\geq d\left(\xi\underbrace{\left(1-\frac{\fbarmut}{\epsilon'}\right)}_{\text{contribution from } B_R(\bx_\ast)}-\underbrace{L\left(\frac{\fbarmut}{\epsilon'}\right)}_{\text{contribution outside } B_R(\bx_\ast)}\right)\,.
    \end{equation}
\end{itemize}
Combining \eqref{eqn:first_term_lower_b0und}, \eqref{eqn:second_term_lower_b0und}, and \eqref{eqn:third_term_lower_b0und}, we obtain
\[
\ddt\fbarmut\geq d\xi-\fbarmut\left(2L+D+\frac{d\xi}{\epsilon'}+\frac{dL}{\epsilon'}\right)\,.
\]
Set $\epsilon=d\xi/\left(2L+D+\frac{d\xi}{\epsilon'}+\frac{dL}{\epsilon'}\right)$, we obtain that $\ddt\fbarmut\geq 0$ if $\fbarmut\leq \epsilon$. This proves \eqref{eqn:lower_bound_f_bar}.
\end{proof}

\section{Technical lemmas}\label{sec:appendix_b}
We provide proofs of Lemma~\ref{lem:long_time_lemma_1}-Lemma~\ref{lem:long_time_lemma_2}, and Lemma~\ref{lem:appendix_B} that was deployed wherein.

\begin{proof}[Proof of Lemma~\ref{lem:long_time_lemma_1}]
This is obtained by contradiction. Assuming the statement is not true, meaning there is no such $\epsilon_0$ to make $\fbarmut>\f_*+\epsilon_0$, then $\inf_{t\geq 0}\fbarmut=\f_*$. Considering $\liminf_{t\rightarrow\infty}\fbarmut>\f_*$, there must exist a finite time, denoted as $t_*\geq 0$ such that $\fbarmuat{t_*}=\f_*$. According to Assumption~\ref{assumption: assmptn1}, $\f$ has a unique global minimum at $\bx_*$, so at $t_*$, $\muxmat{t_*}$ must have the form of\footnote{We note here that if there are multiple global minima $\{\bx_k\}_{k=1}^K$, the form of $\muxmat{t_*}$ becomes $\muxmat{t_*}=\sum_{k=1}^{K}\delta_{\bx_k}(\bx)\otimes \eta_k(m)$ where $\sum_k\eta_k$ has measure $1$. The rest of the proof stays the same, with $\bx_*$ replaced by the collection of $\bx_k$. The proof of this lemma depends on Lemma~\ref{lem:appendix_B}, which can also be revised to accommodate the multiple global minima case.}
\begin{equation}\label{eqn:mu_opt_form}
\muxmat{t_*}=\delta_{\bx_*}(\bx)\otimes \eta(m)\,.
\end{equation}
where $\eta$ is a probability measure over $m$-dimension supported on $\mathbb{R}_+$.

We will show below that $\muxmat{t_*}$ in the form of~\eqref{eqn:mu_opt_form} is an equilibrium. If so, $\mutxm=\muxmat{t_*}$ for all $t>t_*$, and thus $\fbarmut=\f_*$ for all time after $t_*$. This contradicts the assumption that $\liminf_{t\rightarrow\infty}\fbarmut>\f_*$, completing the proof.

To show $\muxmat{t_*}$ in the form of~\eqref{eqn:mu_opt_form} is an equilibrium amounts to plugging the form into the equation~\eqref{eqn:mean_field_pde} and show the right hand side vanishes. Since this is a probability-measured solution with a Dirac delta on $\bx$-domain, the proof has to be conducted in the weak sense. Specifically, given any smooth test function $\phi(\bx,m)$, we show the right hand side of~\eqref{eqn:mean_field_pde}, when tested by $\phi(\bx,m)$, vanishes:
\begin{equation}\label{eqn:evolu_phi_new}
\begin{aligned}
    \iint& \mathrm{R.H.S}~\eqref{eqn:mean_field_pde}\left.\right|_{\mu=\delta_{\bx_*}(\bx)\otimes \eta(m)}\times\phi(\bx,m)\nd{\bx}\nd{m}\\
    =&\underbrace{-\iint  (\f(\bx) - \fbarmut)m\partial_m \phi(\bx,m)\dmutxm}_{\text{Term I}}\\
    &-\iint \left(\underbrace{\nabla_\bx \phi(\bx,m)\cdot \nabla \f(\bx)}_{\text{Term II}}+\underbrace{\gammab(m)\Delta_\bx \phi(\bx,m)}_{\text{Term III}}\right)\dmutxm\,.
\end{aligned}
\end{equation}
We will show all three terms are zero. Indeed,
\begin{itemize}
    \item Since $\mut$ takes on the form of~\eqref{eqn:mu_opt_form},
    \begin{equation}\label{eqn:term_I}
    \begin{aligned}
            \mathrm{Term\ I}=&\iint  (\f(\bx) - \fbarmut) m\partial_m \phi(\bx,m)\dmutxm \\
            = &\iint  (F_* - F_*)m\partial_m \phi(\bx,m)\nd\eta(m)=0\,.
    \end{aligned}
    \end{equation}
    \item Similarly,
    \begin{equation}\label{eqn:term_II}
    \begin{aligned}
        \mathrm{Term\ II}=&\iint \nabla_\bx \phi(\bx,m)\cdot \nabla \f(\bx)\delta_{x_*}(\bx)\nd{\bx}\nd\eta(m)\\
        =&\nabla \f(\bx_*)\cdot\int \nabla_\bx \phi(\bx_*,m)\nd\eta(m)=0\,.    
    \end{aligned}
    \end{equation}
    \item To show $\mathrm{Term\ III}=0$, we note that
    \begin{equation}\label{eqn:term_III}
    \begin{aligned}
    \mathrm{Term\ III}=&\iint \gammab(m)\Delta_\bx \phi(\bx,m)\delta_{x_*}(\bx)\nd{\bx}\nd\eta(m)\\
    =&\int \sigma(m)\Delta_x\phi(\bx_*,m)\nd\eta(m)=0\,,
    \end{aligned}
    \end{equation}
    while in the last equation, we used Lemma~\ref{lem:appendix_B} and that $ \mathrm{Supp}(\eta(m))\subset[m_c,\infty)$.
\end{itemize}
Plugging~\eqref{eqn:term_I}-\eqref{eqn:term_III} into~\eqref{eqn:evolu_phi_new}, it is clear that $\muxmat{t_*}=\delta_{\bx_*}(\bx)\otimes \eta(m)$ is an equilibrium, and thus $\mutxm=\muxmat{t_*}$ for all $t>t_*$, making $\fbarmut=\f_*$ and contradicting the assumption that $\liminf_{t\rightarrow\infty}\fbarmut>\f_*$.
\end{proof}

\begin{proof}[Proof of Lemma~\ref{lem:long_time_lemma_2}]
According to Assumption~\ref{assumption: assmptn1}, the initial distribution satisfies $(m_c,\infty)\cap\mathrm{supp}_m(\muxmat{0})\neq\emptyset$, so using the diffusivity property in $\bx$-direction of the Fokker-Planck equation, for $\Nhood_{\epsilon/2}$ as defined in~\eqref{eqn:Nhood_eps}, there must exist $T_c>0$ such that
\[
\int_{\Nhood_{\epsilon/2}}\int^\infty_{m_c}\dmuxmat{T_c}>0\,,
\] 
which means there is no trivial probability at $T_c$ to find a sample in the set of $\Nhood_{\epsilon/2}\times(m_c,\infty)$. Denote one such sample $(\bx_s(T_c),m_s(T_c))$, we are to show that
\begin{enumerate}
    \item[(a)] $(\bx_s,m_s)(t)\in \Nhood_{\epsilon/2}\times(m_c,\infty)$ for all $t>T_c$;
    \item[(b)] $m_s(t)>\exp(\frac{\epsilon}{2}(t-T_c))m_s(T_c)>\exp(\frac{\epsilon}{2}(t-T_c))m_c$ for any $t>T_c$.
\end{enumerate}
Indeed, we recall its trajectory governed by the SDE:
\[
\left\{\begin{array}{rcl}
   \nd \bx_s(t)  & = & -\nabla \f(\bx_s(t))\dt + \sqrt{2\gammab(m_s(t))}\dW_t\\
   \ddt m_s(t) & = & -m_s(t)\left(\f(\bx_s(t)) - \fbarmut\right)\,.
\end{array}\right.
\]
Since $\bx_s(T_c)\in \Nhood_{\epsilon/2}$ and $m_s(T_c)>m_c$, then noting $\gammab(m_s(T_c))=0$, we obtain
\[
\left.\frac{\nd \bx_s(t)}{\dt}\right|_{t=T_c}=-\nabla \f(\bx_s(t)),\quad \left.\frac{\nd m_s(t)}{\dt}\right|_{t=T_c}>\frac{\epsilon}{2} m_s(t)
\]
where we used the fact that $\fbarmut>\f_*+\epsilon$ while $\f(\bx_s(T_c))<\f_*+\epsilon/2$ to obtain the equation for $m_s$. This dynamics suggests that $\f(\bx_s(t))$ keeps decreasing so $\bx_s\in\Nhood_{\epsilon/2}$, and $m_s(t)$ keeps increasing with an exponential rate:
\[
m_s(t)>\exp\left(\frac{\epsilon}{2}(t-T_c)\right)m_s(T_c)> m_c\,,
\]
confirming (a)-(b). Accordingly,
\begin{equation}\label{eqn:total_mass_lower}
\begin{aligned}
\int \limits_{\Nhood_{\epsilon/2}}&\int \limits^\infty_{m_c}m\dmutxm
   =\mathbb{E}(m(t)|\bx(t)\in \Nhood_{\epsilon/2},m(t)>m_c)\mathbb{P}\left(\left\{\bx(t)\in \Nhood_{\epsilon/2},m(t)>m_c\right\}\right)\\
 \ \geq& \mathbb{E}(m(t)|\bx(T)\in \Nhood_{\epsilon/2},m(T)>m_c)\mathbb{P}\left(\left\{\bx(T)\in \Nhood_{\epsilon/2},m(T)>m_c\right\}\right)\\
\ \geq& \mathbb{E}(\exp(\frac{\epsilon}{2}(t-T))m(T)|\bx(T)\in \Nhood_{\epsilon/2},m(T)>m_c)\mathbb{P}\left(\left\{\bx(T)\in \Nhood_{\epsilon/2},m(T)>m_c\right\}\right)\\
\ =&\exp(\frac{\epsilon}{2}(t-T))\int_{\Nhood_{\epsilon/2}}\int^\infty_{m_c}m\dmuxmat{T}\,,
\end{aligned}
\end{equation}
where we use (a) in the first inequality and (b) in the second inequality. This leads to the fact that
\[
\limsup_{t\rightarrow\infty}\int_{\mathbb{R}^d}\int^\infty_{0}m\dmutxm\geq \limsup_{t\rightarrow\infty}\int_{\Nhood_\epsilon}\int^\infty_{m_c}m\dmutxm=\infty\,,
\]
completing the proof.
\end{proof}

\begin{lemma}\label{lem:appendix_B}
    Under the condition of Theorem~\ref{thm:long_time_thm}. Let $\mutxm$ solve~\eqref{eqn:mean_field_pde}. If at a time $t_*$ so that $\muxmat{t_*}$ takes on the form of $\delta_{\bx_*}\otimes\eta(m)$ for some probability measure $\eta$, then $ \mathrm{Supp}(\eta(m))\subset[m_c,\infty)$.\footnote{We note that in the statement of the lemma, we assume $\muxmat{t_*}$ has a form of one Dirac delta located on $\bx_*$. In the situation when there are multiple global minima $\{\bx_k\}_{k=1}^K$, the proof can be adjusted to handle the form of $\muxmat{t_*}=\sum_{k=1}^K\delta_{\bx_k}\otimes\eta_k(m)$, with the conclusion being $\mathrm{Supp}(\eta_k(m))\subset[m_c,\infty)$ for all $k$. In the proof, we should employ a list of test functions with each taking the form of $\phi_k(\bx,m)=|\bx-\bx_k|^2\xi_k(\bx)$.}
\end{lemma}
\begin{proof}
This can be seen by deploying a special test function. Set the test function $\phi_*(\bx,m)=|\bx-\bx_*|^2\xi(\bx)$, where $\xi(\bx)$ is a compactly supported, non-negative, smooth function such that $\xi(\bx)=1$ in a neighborhood of $\bx_*$. Naturally by definition, $\partial_m\phi_*(\bx,m)=0$. Testing it on the mean-field PDE~\eqref{eqn:mean_field_pde}, we obtain that
\begin{equation}\label{eqn:evolu_phi}
\begin{aligned}
    & \underbrace{\ddt\iint \phi_*(\bx,m)\dmuxmat{t_*}}_{\mathrm{Term\ I}}\\
     & \qquad = -\iint \left(\underbrace{\nabla \phi_*(\bx,m)\cdot \nabla \f(\bx)}_{\mathrm{Term\ II}}+\sigma(m)\Delta \phi_*(\bx,m)\right)\dmuxmat{t_*}\,,
    \end{aligned}
\end{equation}
where the $\big(\f(\bx)-\fbarmut\big)\partial_{m}(m\mu)\left.\right|_{\muxmat{t_*}}=0$ and is dropped out upon testing, using the same argument that showed~\eqref{eqn:term_I}. We will show that both $\mathrm{Term\ I}$ and $\mathrm{Term\ II}$ are zero.

Indeed, denote $A(t) := \iint \phi_*(\bx,m)\dmutxm$. According to the definition of $\phi_*$, $A(t)\geq 0$ for all $t$. Meanwhile, $A(t_*) = \int \phi_*(\bx_*,m)\nd\eta(m)=0$. As a consequence,
\[
\mathrm{Term\ I}=\ddt\iint \phi_*(\bx,m)\dmuxmat{t_*}=\ddt A(t_*)=0\,.
\]
Similarly $\mathrm{Term\ II}=0$ is due to the fact that $\nabla \f(\bx_*)=0$.

These together suggests $\iint \sigma(m)\Delta \phi_*(\bx,m)\dmuxmat{t_*}=0$. Considering $\muxmat{t_*}=\delta_{\bx_*}\otimes\eta(m)$ and that $\Delta\phi_*(\bx_*) = 2d$, we have $\int\sigma(m)\nd\eta(m)=0$, and thus $\mathrm{Supp}(\eta(\cdot))\subset[m_c,\infty)$.
\end{proof}

\section{Numerical Experiments}\label{sec:more_num_exp}
In this section we present further numerical evidence of the performance of our Swarm-based Simulated Annealing (SSA) Algorithm \ref{algo:SSA} for different types of 2D non-convex functions. We used~\eqref{eqn:gamma_smooth} as our $\gammab$ function with different parameters and different number of particles in each case. These examples are selected from~\cite{opt_tests}, and in each plot below, we show the landscape of the objective function, the convergence of the provisional minimum in time, and four time snapshots of available samples. At initial time, the samples are purposely removed from the global basin.

\begin{figure}[ht]
    \centering
    \includegraphics[width=0.3\linewidth]{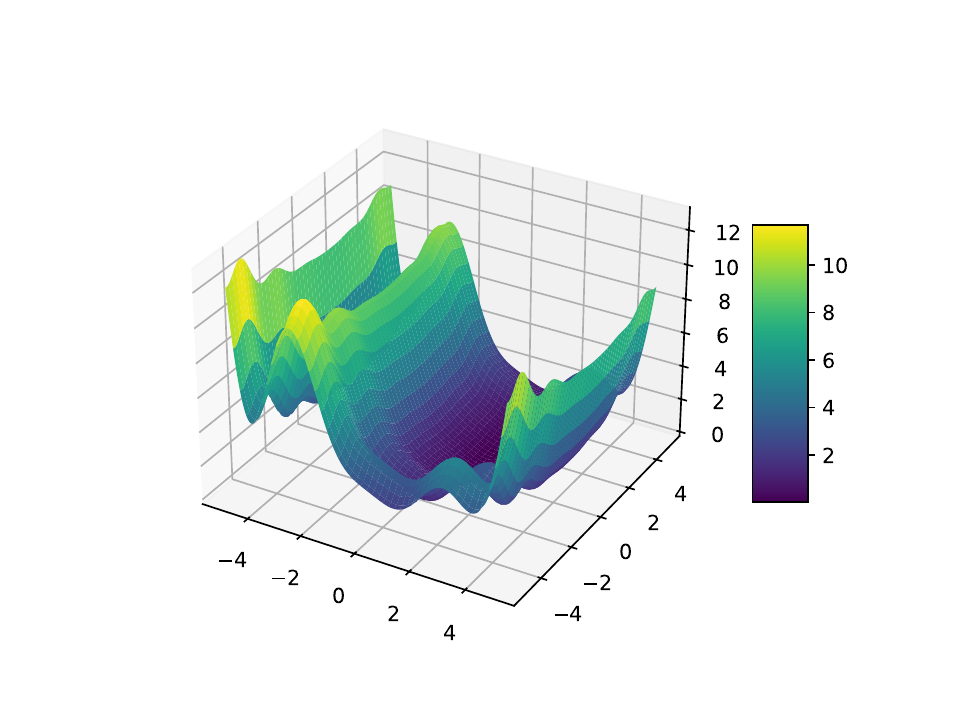}
    \includegraphics[width=0.3\linewidth]{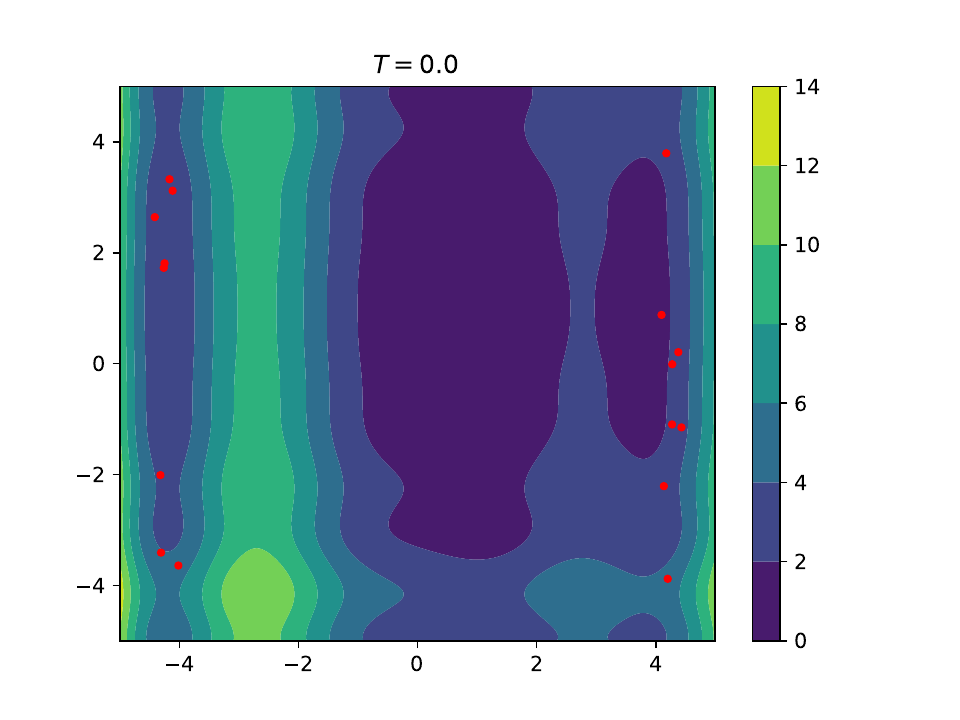}
    \includegraphics[width=0.3\linewidth]{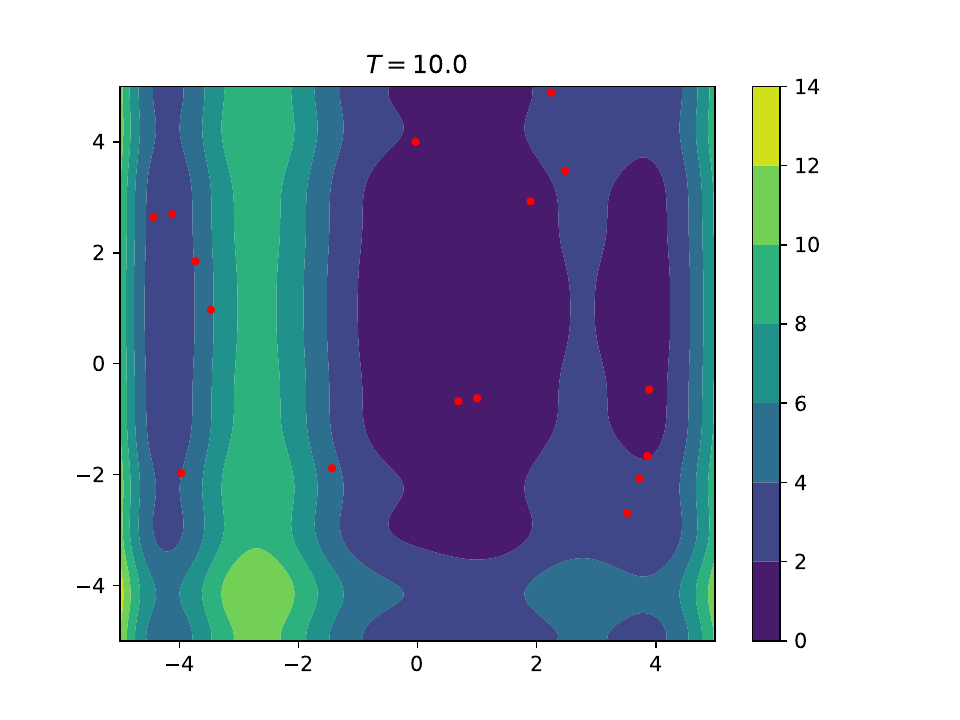}
        \includegraphics[width=0.3\linewidth]{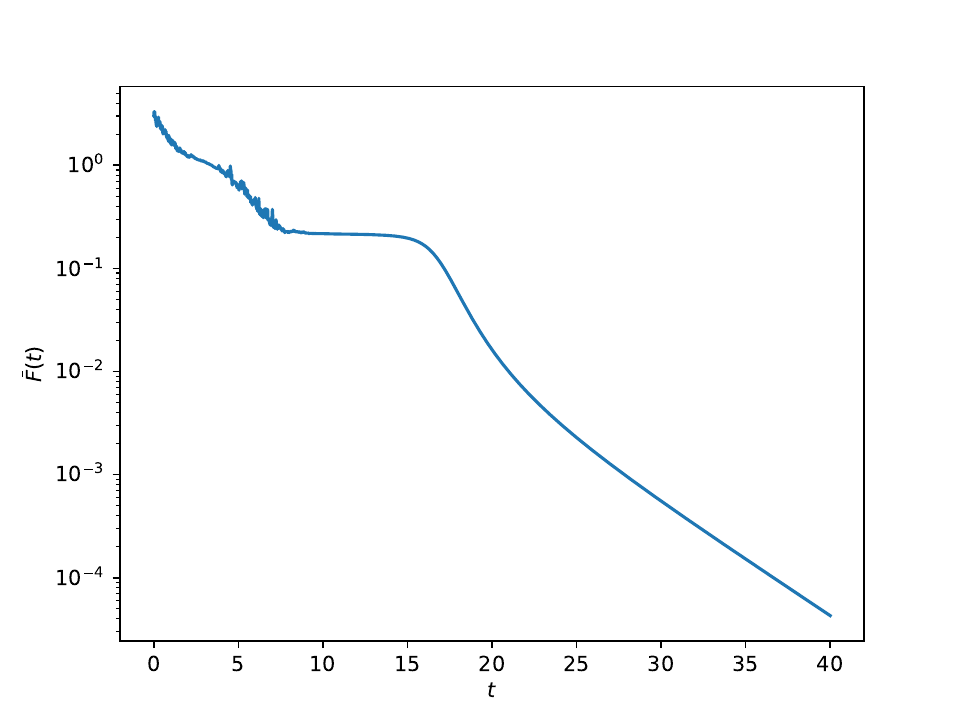}
    \includegraphics[width=0.3\linewidth]{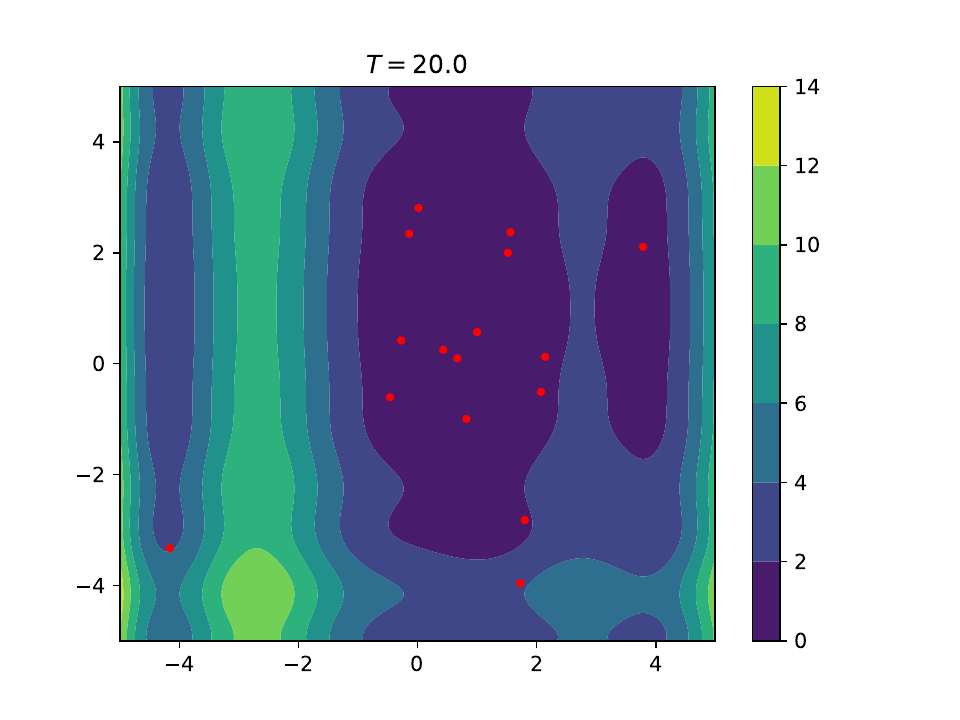}
    \includegraphics[width=0.3\linewidth]{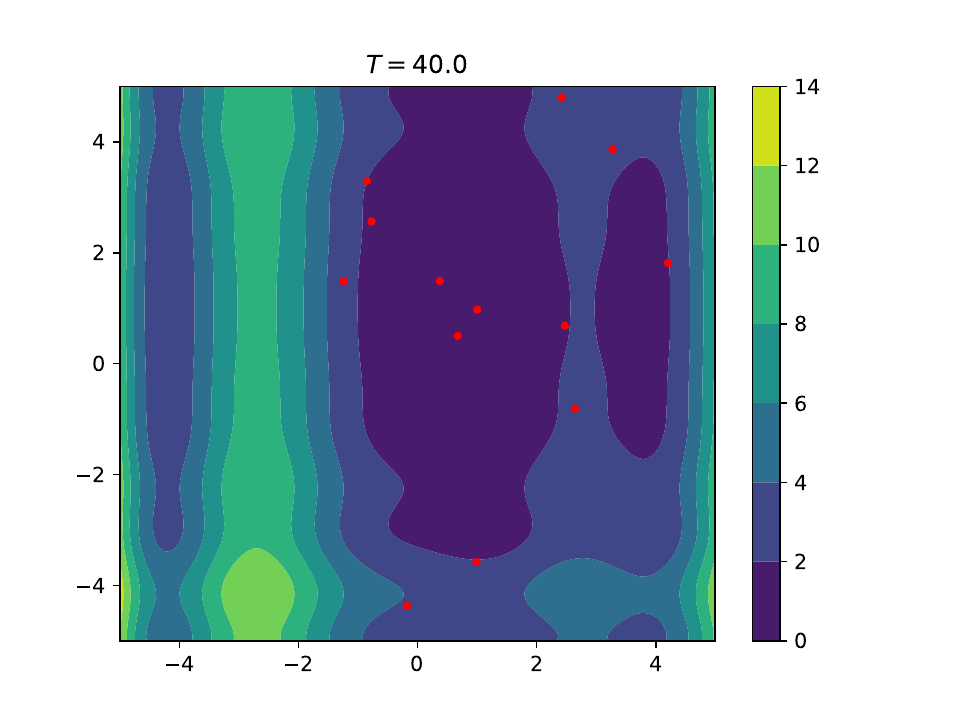}
    \caption{SSA applied to a Levy function: $\f(x,y) = \sin^{2}(\pi w(x)) + (w(x)-1)^{2}[1+10\sin^{2}(\pi w(x)+1)] + (w(y)-1)^{2}[1+\sin^{2}(2\pi w(y))]$, with  $w(x) = 1 + \frac{x-1}{4}$. The parameters used are $N=16$, $\alphalam=2$ and $\betamu=1/8$. The global minimum is $\f(x_\ast,y_\ast) = 0$ at $(x_\ast,y_\ast) = (1,1)$. Top Left: landscape of he function. Lower Left: long time behavior. Center and Right: swarm movement in time.}
    \label{fig:Levy2D}
\end{figure}

\begin{figure}[ht]
    \centering
    \includegraphics[width=0.3\linewidth]{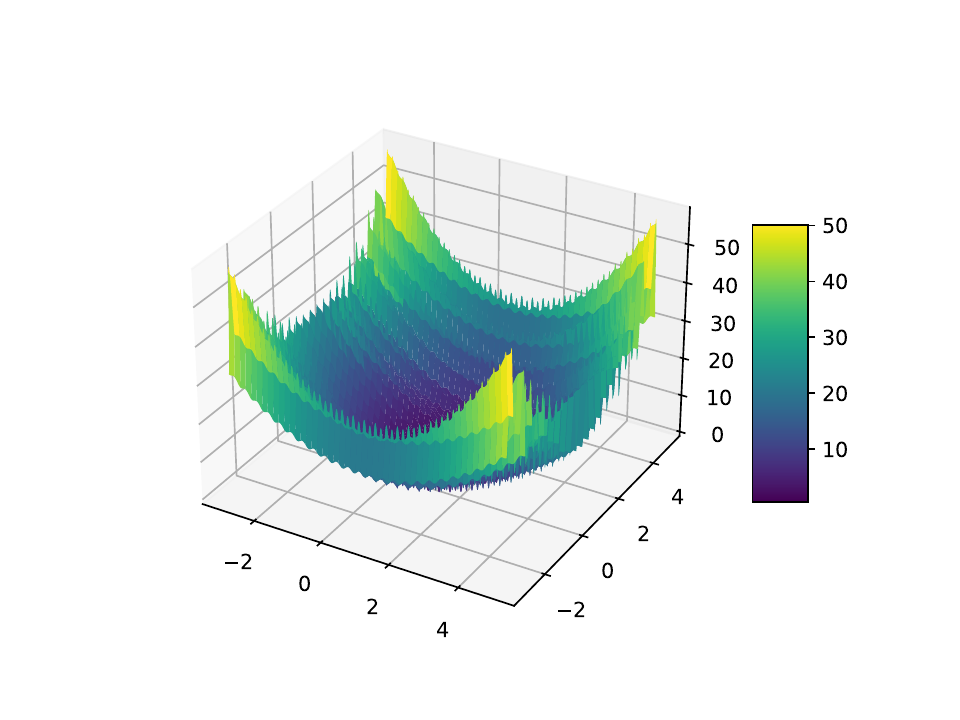}
    \includegraphics[width=0.3\linewidth]{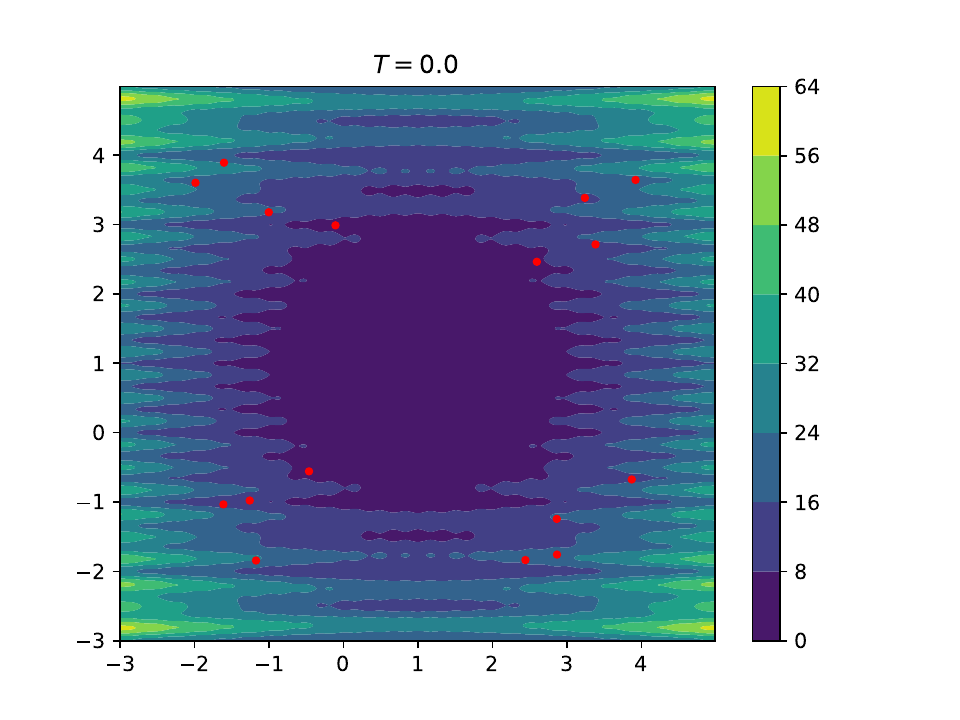}
    \includegraphics[width=0.3\linewidth]{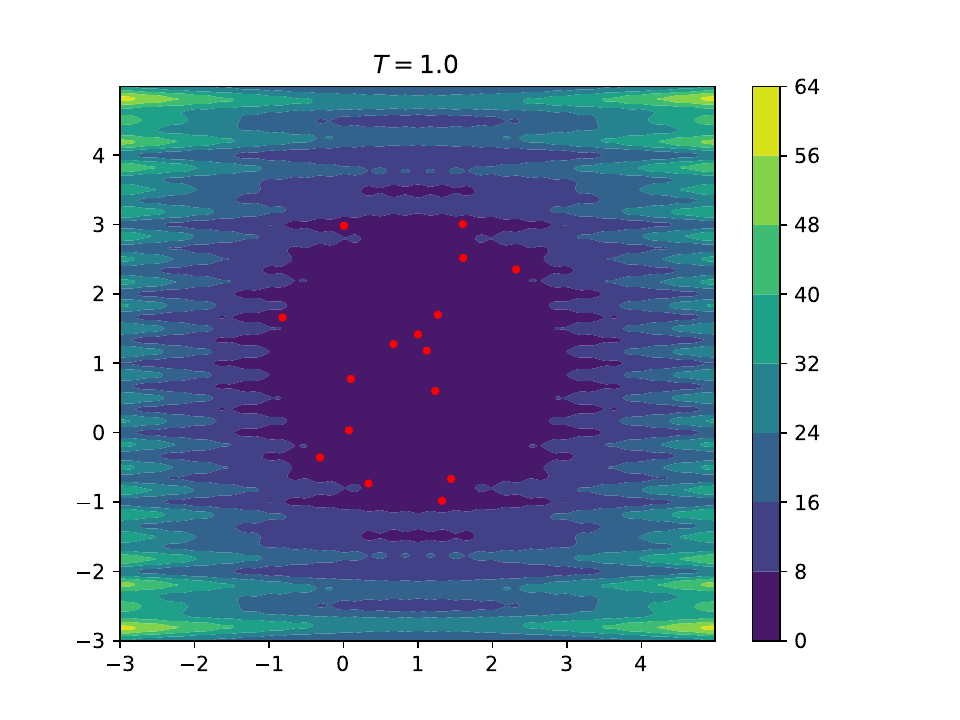}
        \includegraphics[width=0.3\linewidth]{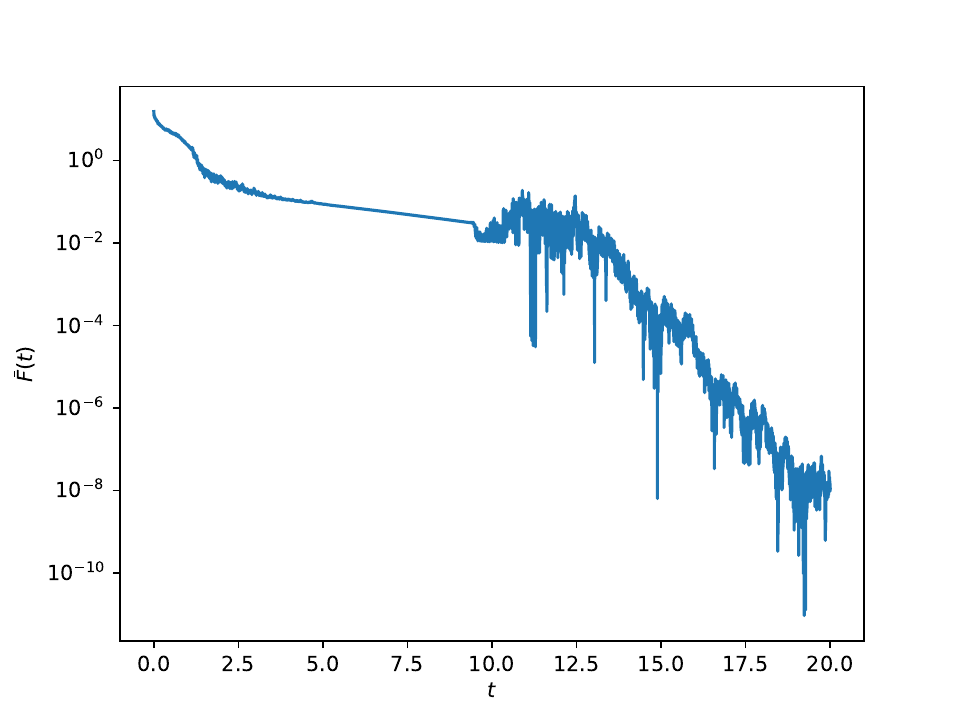}
    \includegraphics[width=0.3\linewidth]{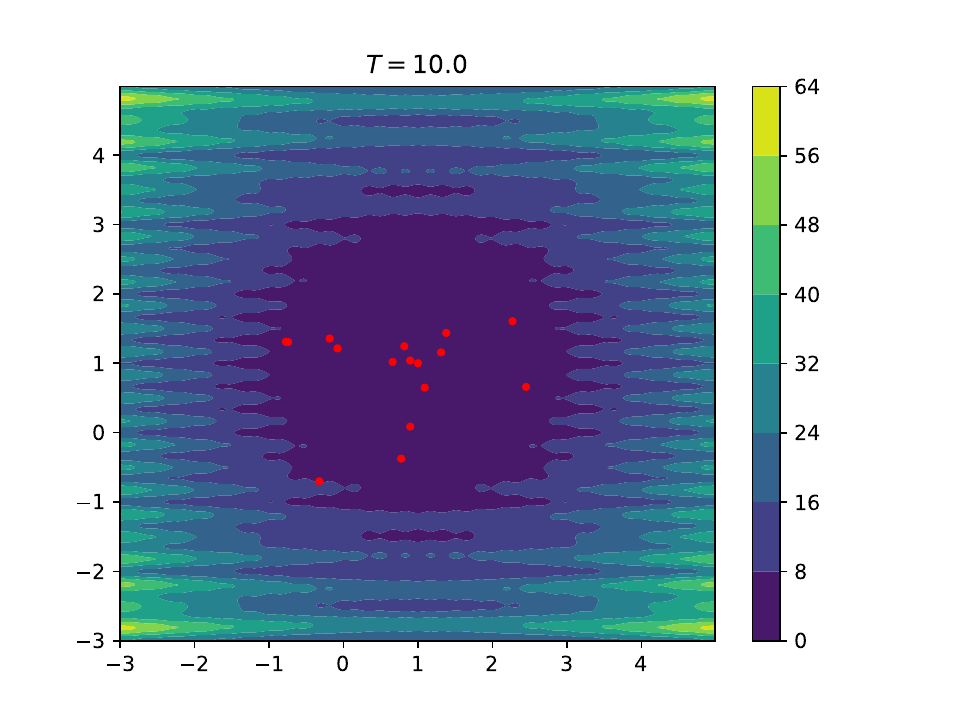}
    \includegraphics[width=0.3\linewidth]{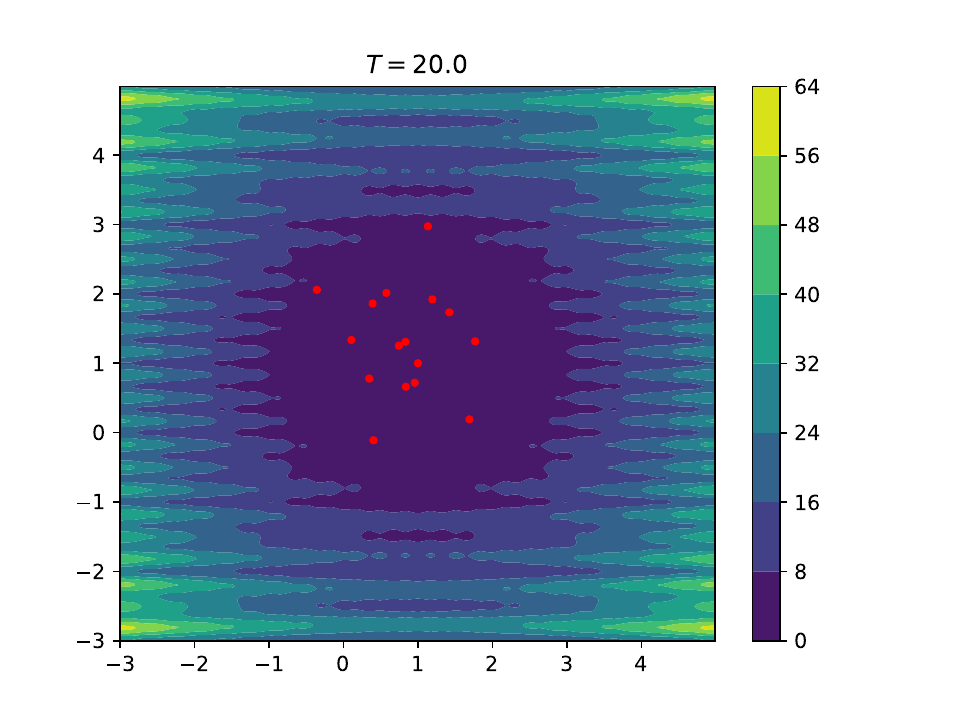}
    \caption{SSA applied to a Levy N.13 function: $\f(x,y) = \sin^{2}(3\pi x) + (x-1)^{2}[1+\sin^{2}(3\pi y)] + (y-1)^{2}[1+\sin^{2}(2\pi y)]$. The parameters used are $N=16$, $\alphalam=2$ and $\betamu=1/8$. The global minimum is $\f(x_\ast,y_\ast) = 0$ at $(x_\ast,y_\ast) = (1,1)$. Top Left: landscape of he function. Lower Left: long time behavior. Center and Right: swarm movement in time. Similar to Figure~\ref{fig:Levy2D}, we can see the provisional minimum converges to the global minimum as the particles are able to accumulate around the global minimum.}
    \label{fig:LevyN132D}
\end{figure}

\begin{figure}[ht]
    \centering
    \includegraphics[width=0.3\linewidth]{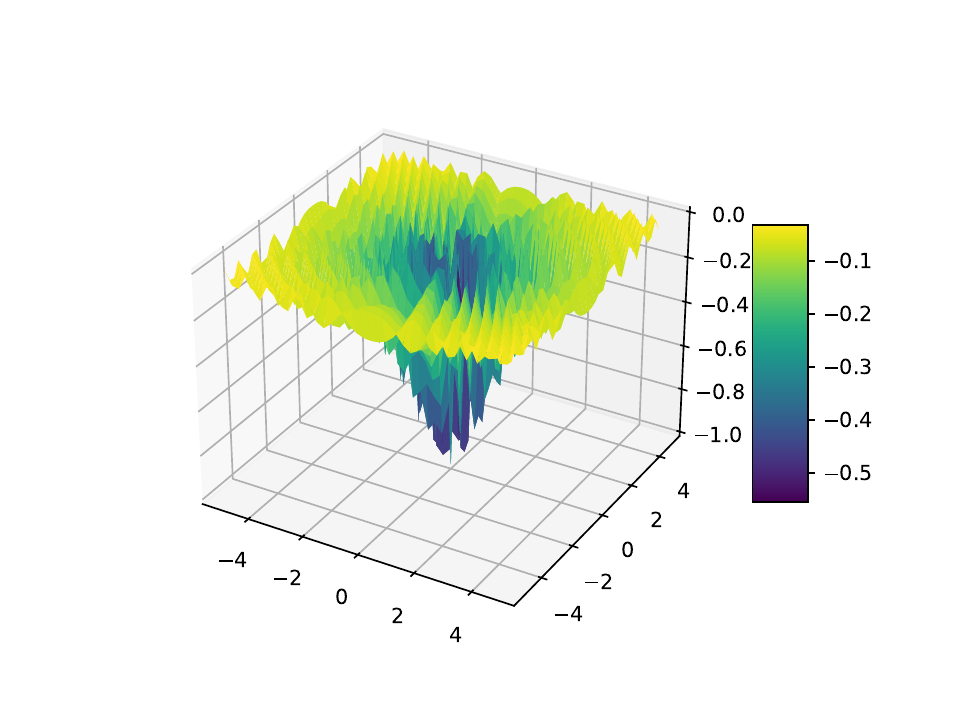}
    \includegraphics[width=0.3\linewidth]{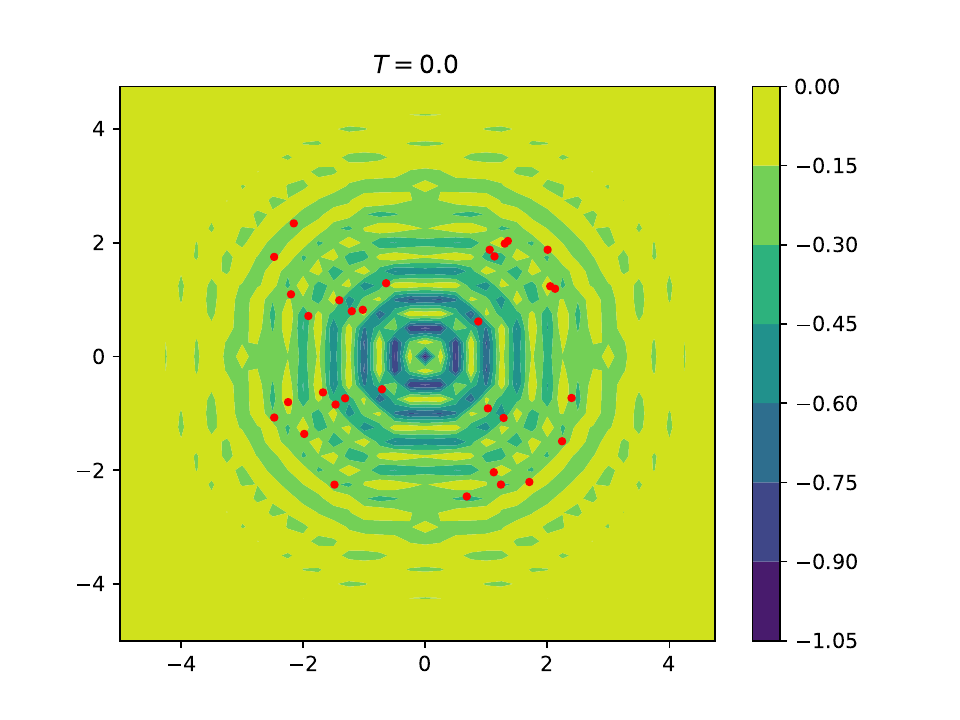}
    \includegraphics[width=0.3\linewidth]{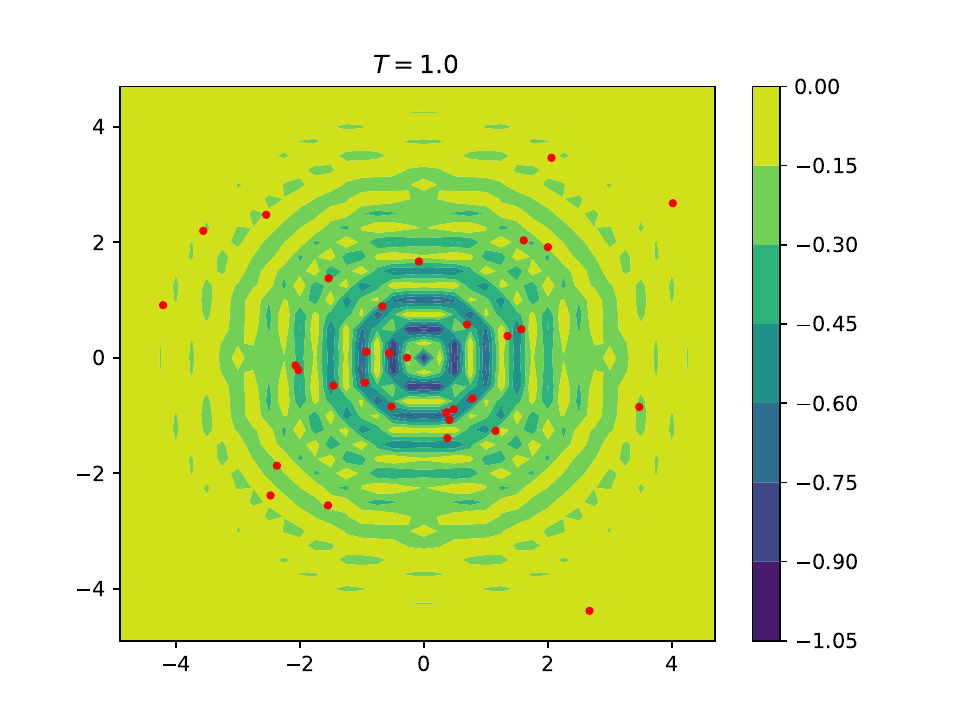}
        \includegraphics[width=0.3\linewidth]{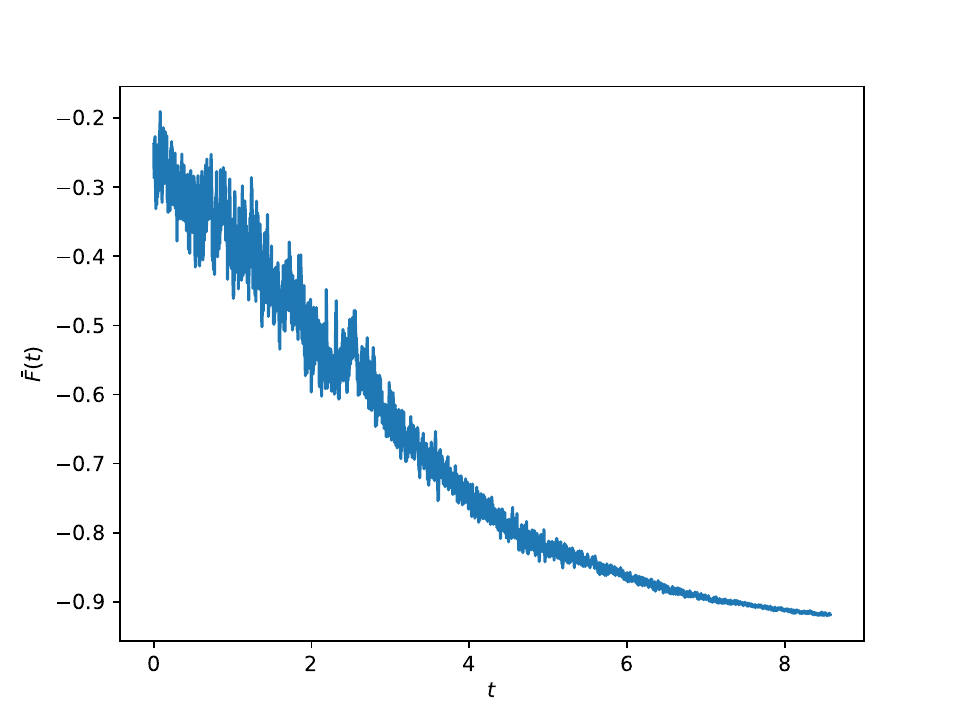}
    \includegraphics[width=0.3\linewidth]{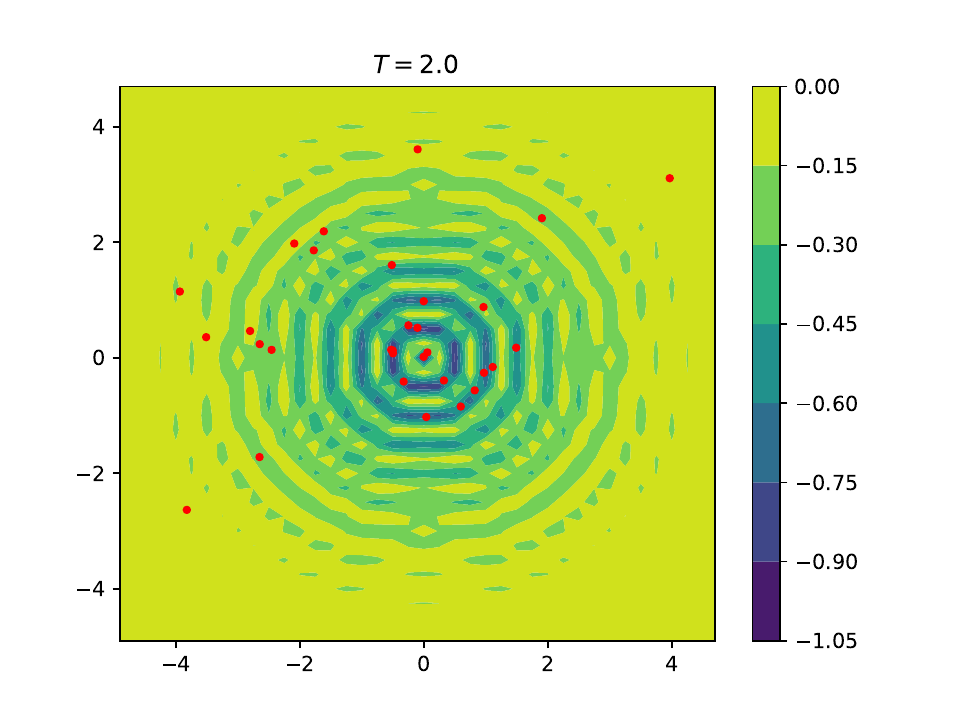}
    \includegraphics[width=0.3\linewidth]{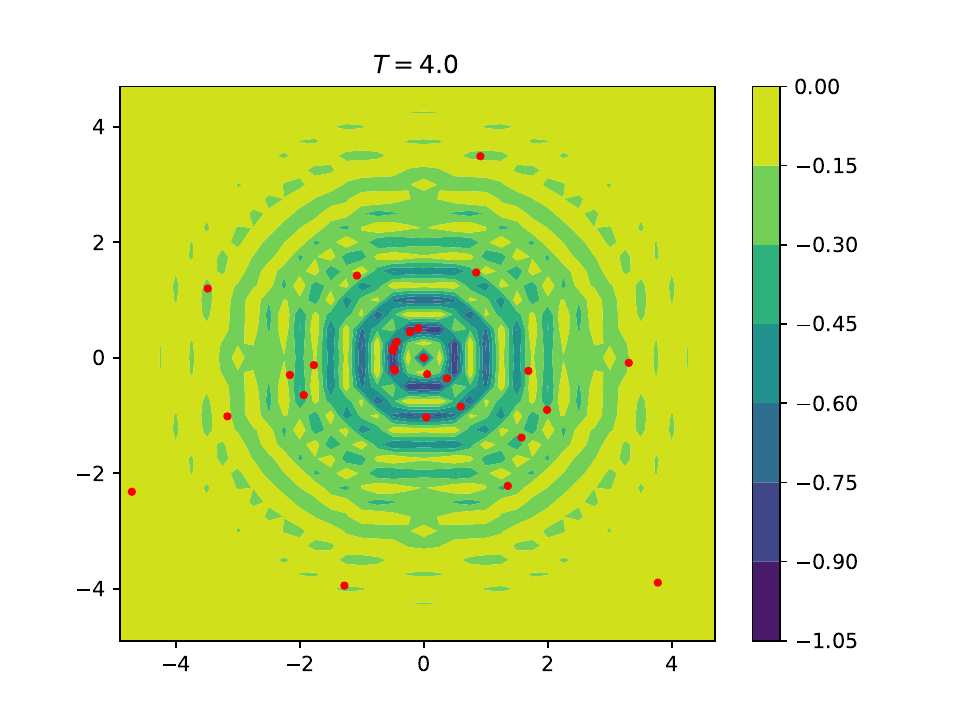}
    \caption{SSA applied to a Drop-Wave function: $\f(x,y) =  -\frac{1 + \cos(12\sqrt{x^2 + y^2})}{0.5(x^2 + y^2) + 2}$. The parameters used are $N=32$, $\alphalam=2$ and $\betamu=1/16$. The global minimum is $\f(x_\ast,y_\ast) = -1$ at $(x_\ast,y_\ast) = (0,0)$. Top Left: landscape of he function. Lower Left: long time behavior. Center and Right: swarm movement in time. Top Left: landscape of he function. Lower Left: long time behavior. Center and Right: swarm movement in time. This is a challenging example since the global basin is very small.}
    \label{fig:DropWave2D}
\end{figure}

\begin{figure}[ht]
    \centering
    \includegraphics[width=0.3\linewidth]{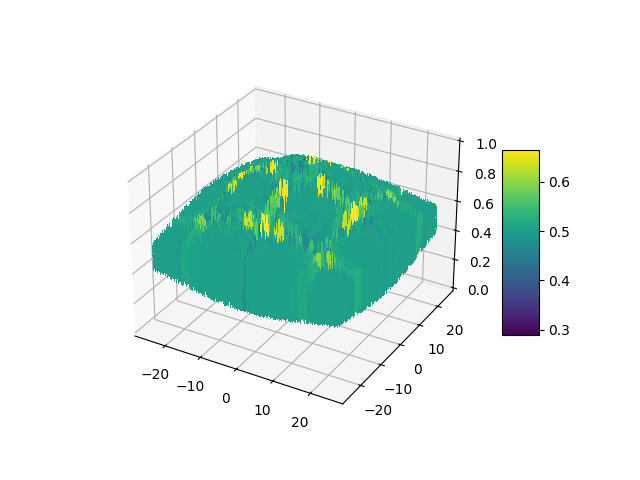}
    \includegraphics[width=0.3\linewidth]{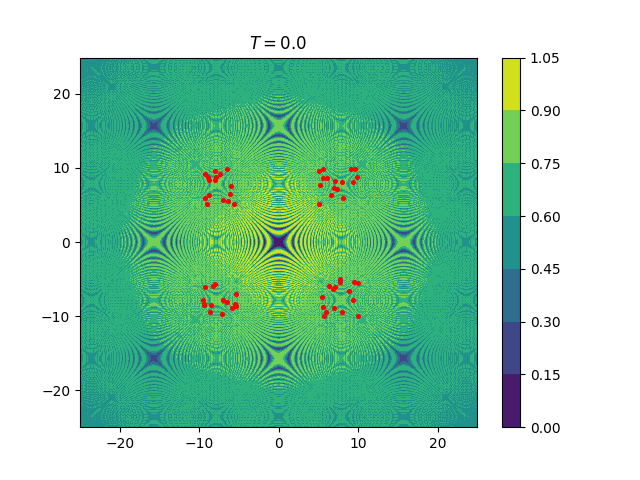}
    \includegraphics[width=0.3\linewidth]{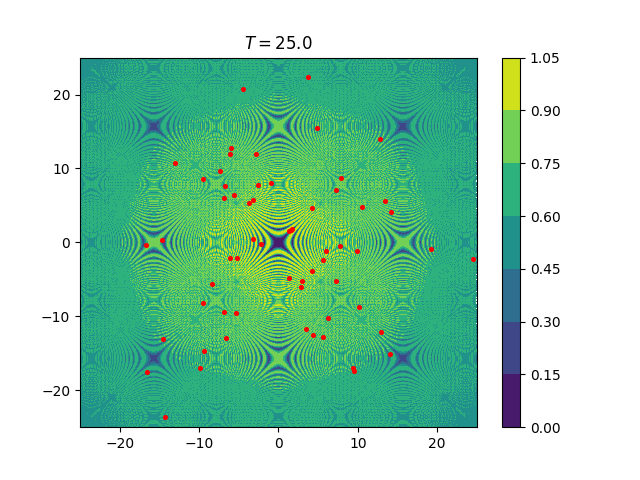}
        \includegraphics[width=0.3\linewidth]{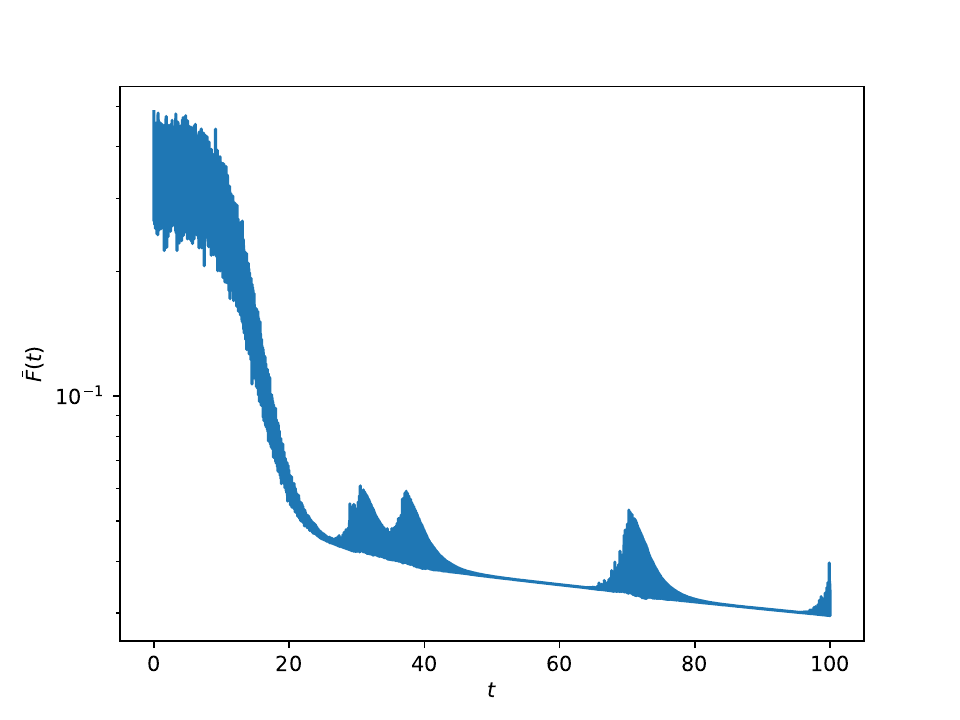}
    \includegraphics[width=0.3\linewidth]{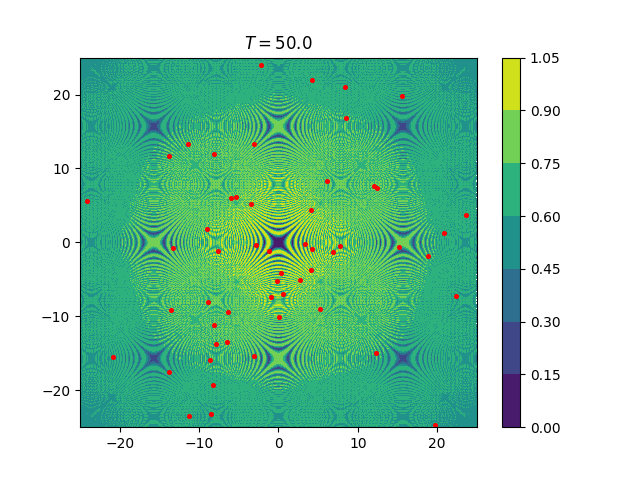}
    \includegraphics[width=0.3\linewidth]{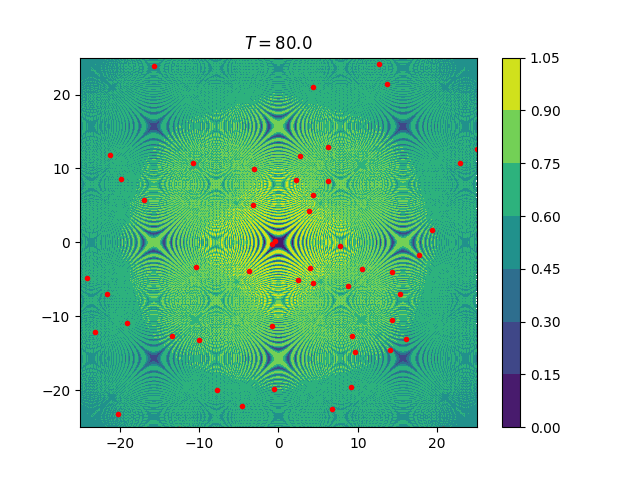}
    \caption{SSA applied to a Schaffer N.2 function: $\f(x,y) =  0.5 + \frac{\sin^{2}(x^2 - y^2) - 0.5}{[1 + 0.001(x^2 + y^2)]^2}$. The parameters used are $N=64$, $\alphalam=2$ and $\betamu=1/32$. The global minimum is $\f(x_\ast,y_\ast) = 0$ at $(x_\ast,y_\ast) = (0,0)$. Top Left: landscape of he function. Lower Left: long time behavior. Center and Right: swarm movement in time. This is also a challenging example since the global basin is very small.}
    \label{fig:Schaffer2D}
\end{figure}

\begin{figure}[ht]
    \centering
    \includegraphics[width=0.3\linewidth]{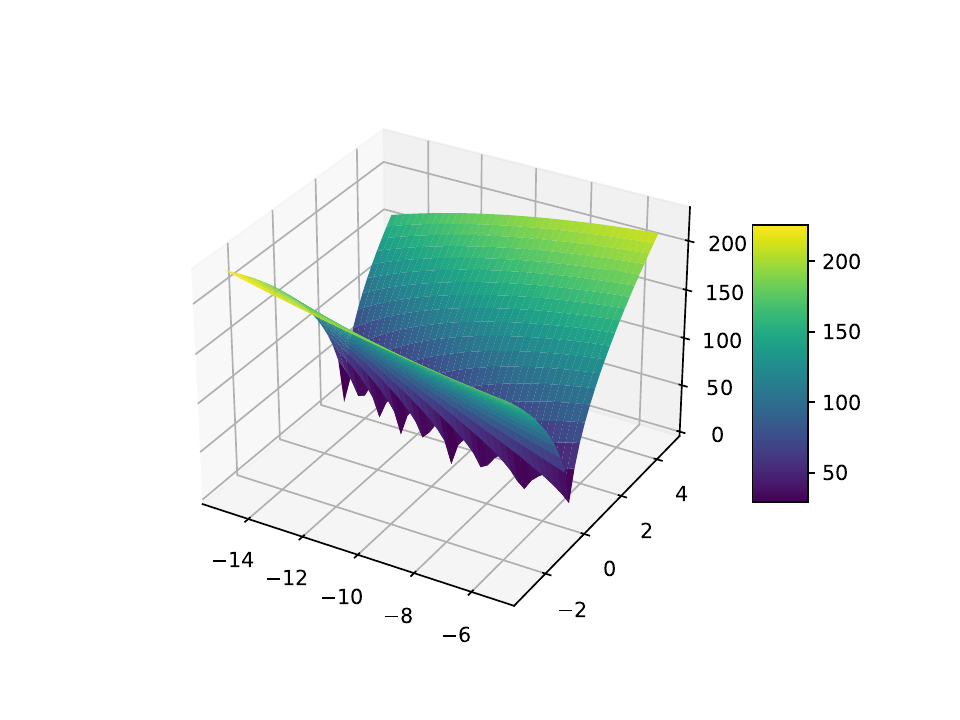}
    \includegraphics[width=0.3\linewidth]{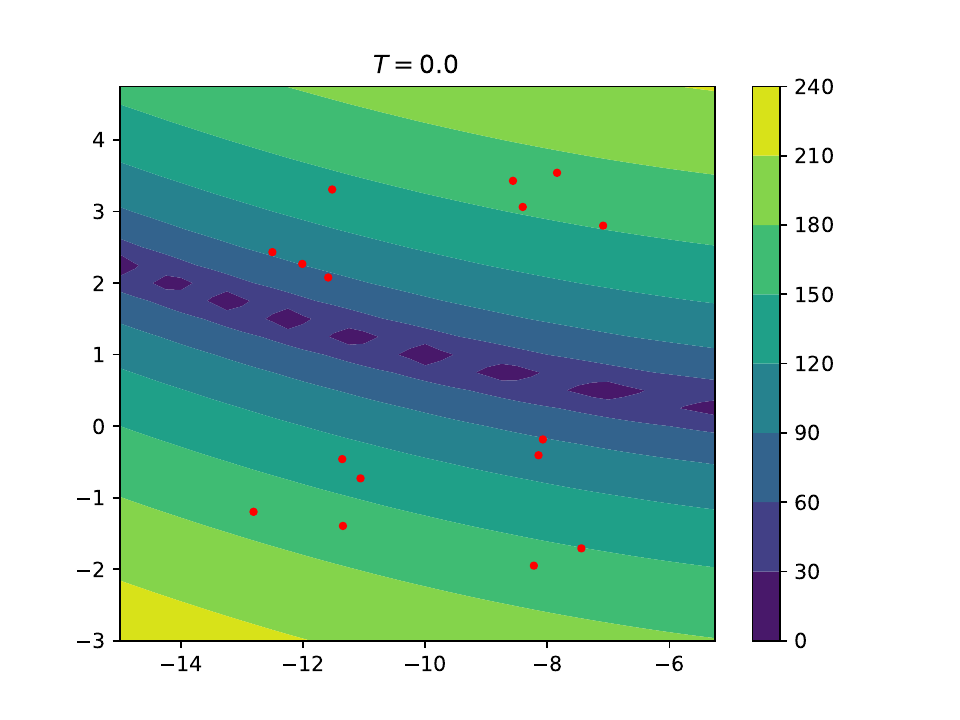}
    \includegraphics[width=0.3\linewidth]{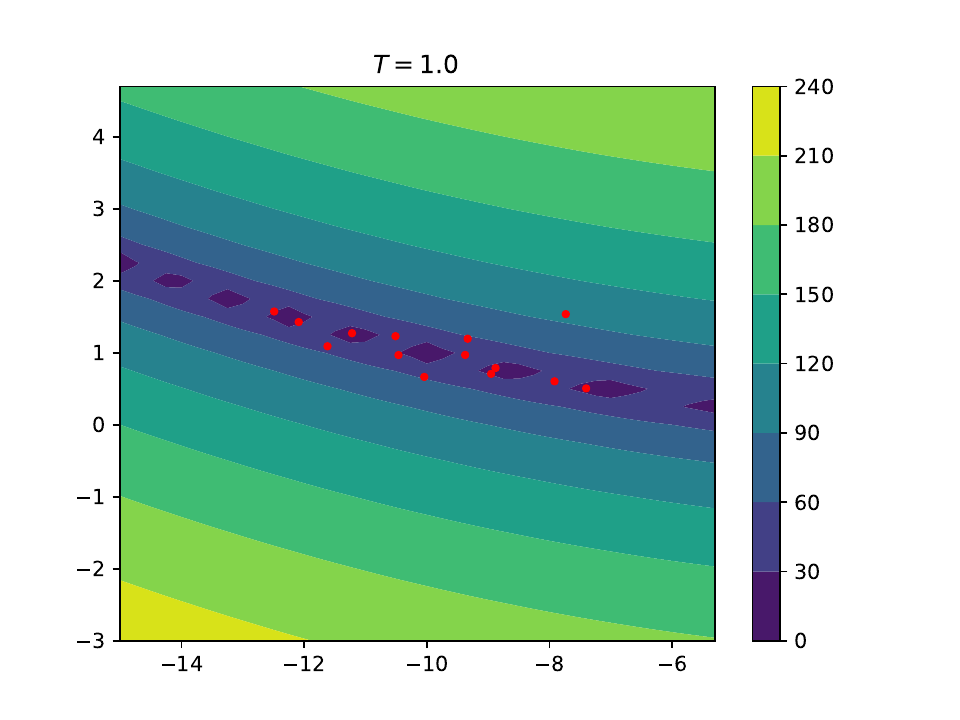}
        \includegraphics[width=0.3\linewidth]{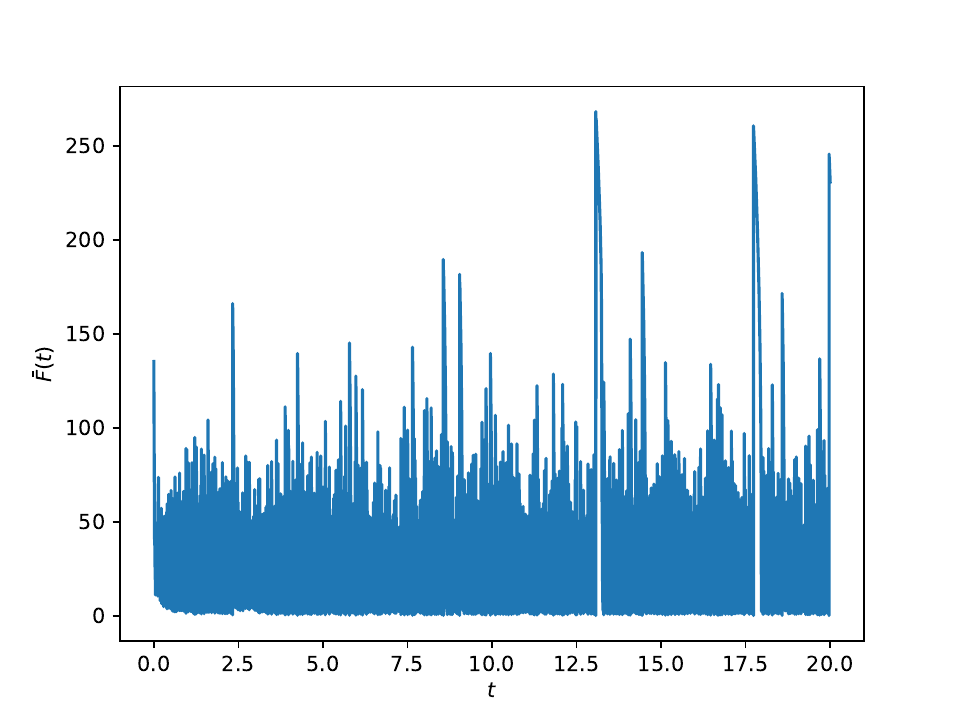}
    \includegraphics[width=0.3\linewidth]{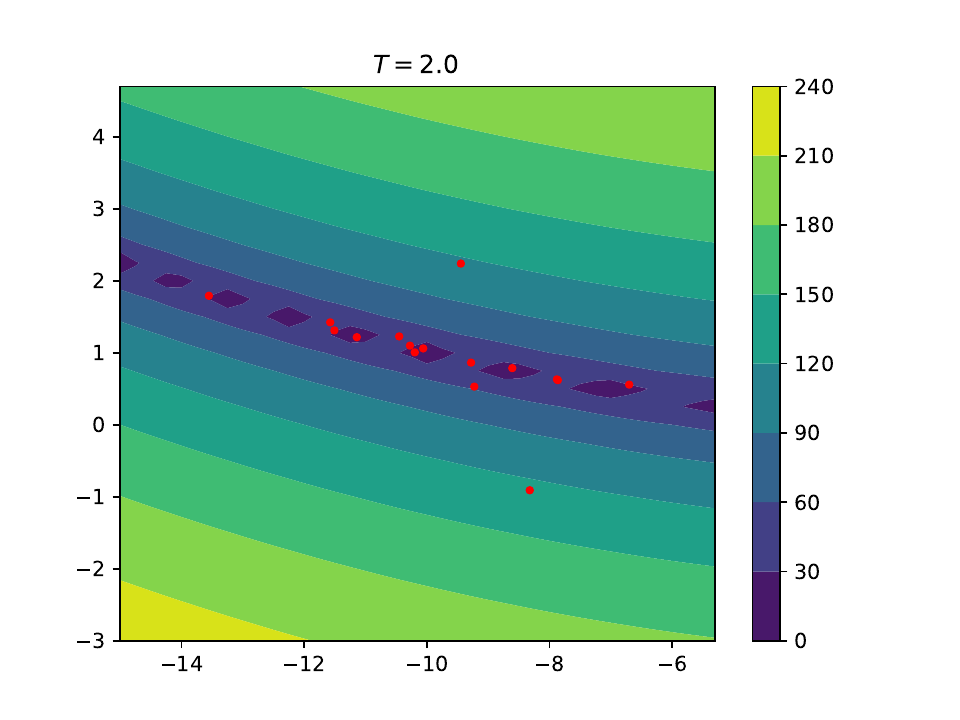}
    \includegraphics[width=0.3\linewidth]{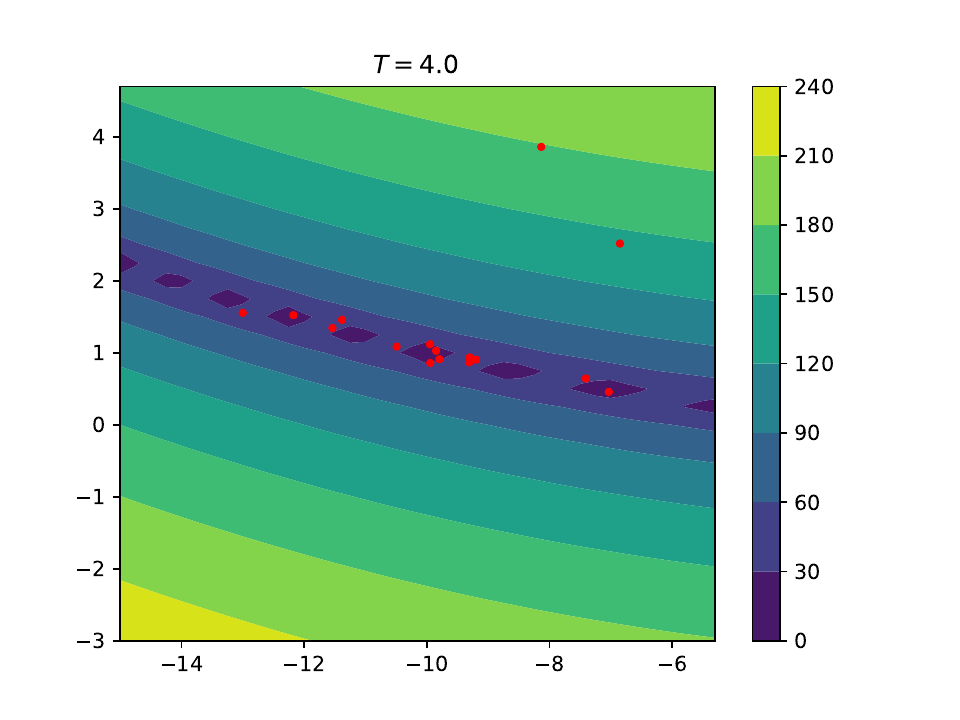}
    \caption{SSA applied to a Bukin N.6 function: $\f(x,y) =  100\sqrt{|y-0.01x^2|} + 0.01|x+10|$. The parameters used are $N=16$, $\alphalam=0.25$ and $\betamu=1/16$. The global minimum is $\f(x_\ast,y_\ast) = 0$ at $(x_\ast,y_\ast) = (-10,1)$. Top Left: landscape of he function. Lower Left: long time behavior. Center and Right: swarm movement in time. This is a very special example: it is observed that the particles can reach the global basin with ease but since all local and global basins are very shallow, samples with Brownian motion easily jump out. As a consequence, the provisional minimum oscillates very wildly.}
    \label{fig:BukinN62D}
\end{figure}

\begin{figure}[ht]
    \centering
    \includegraphics[width=0.3\linewidth]{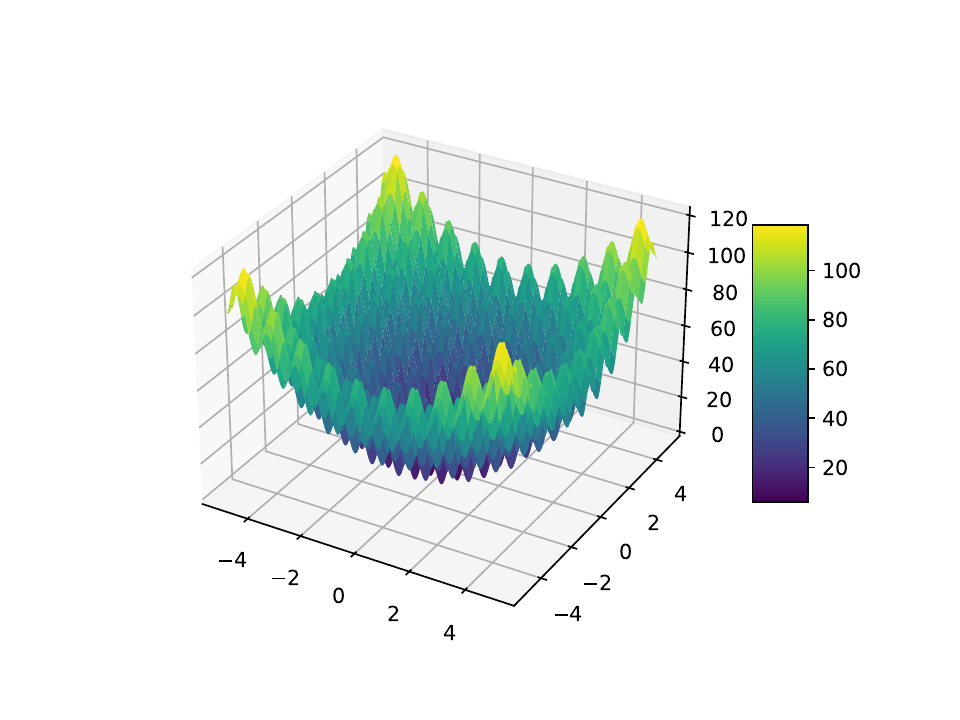}
    \includegraphics[width=0.3\linewidth]{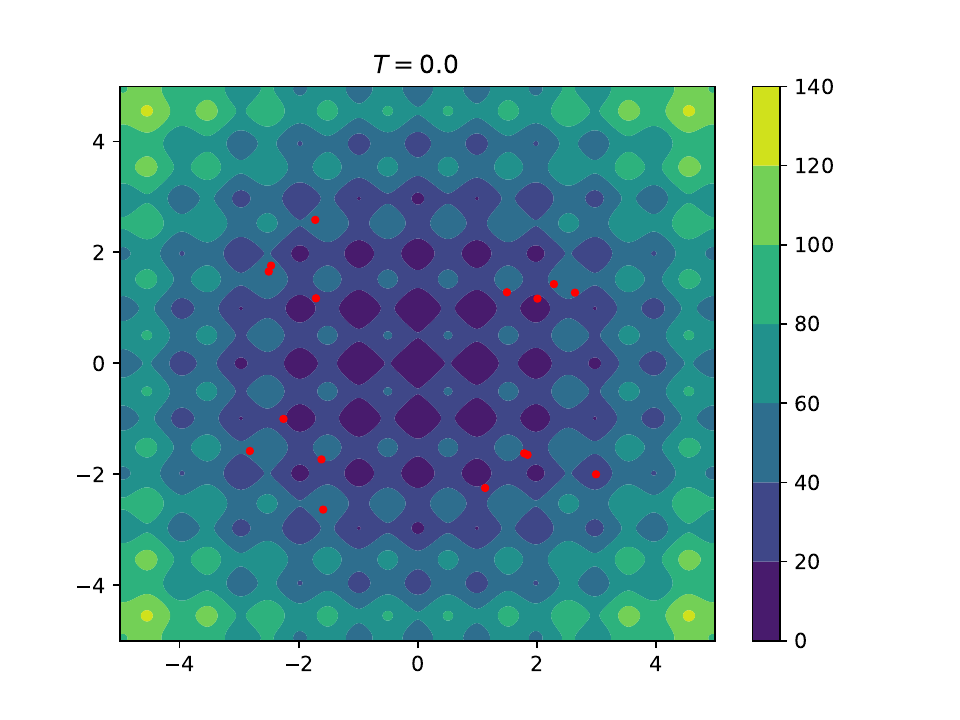}
    \includegraphics[width=0.3\linewidth]{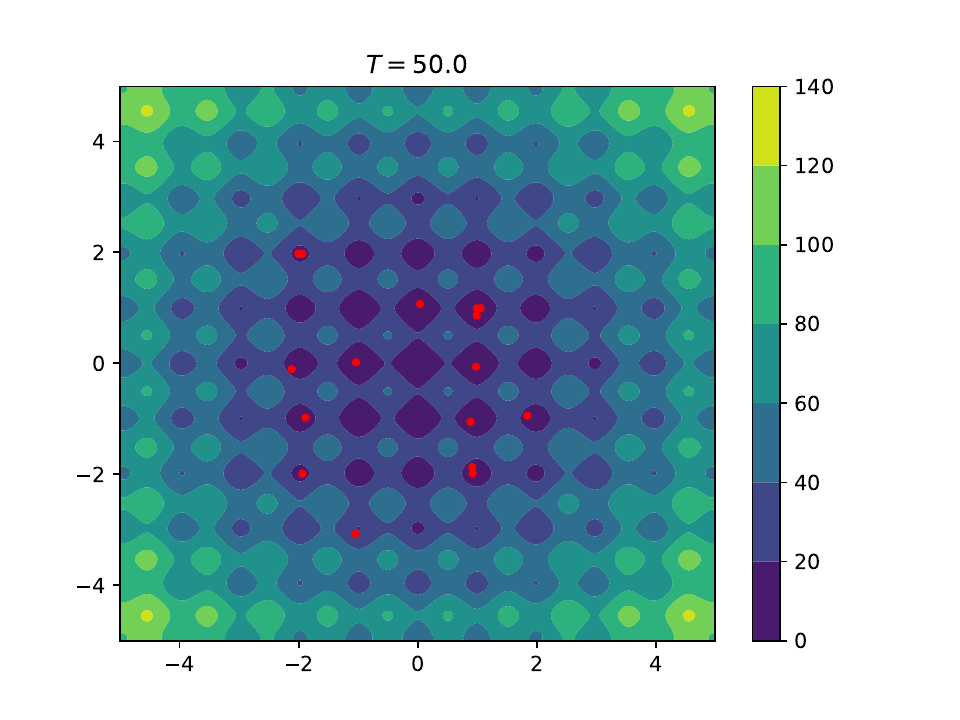}
        \includegraphics[width=0.3\linewidth]{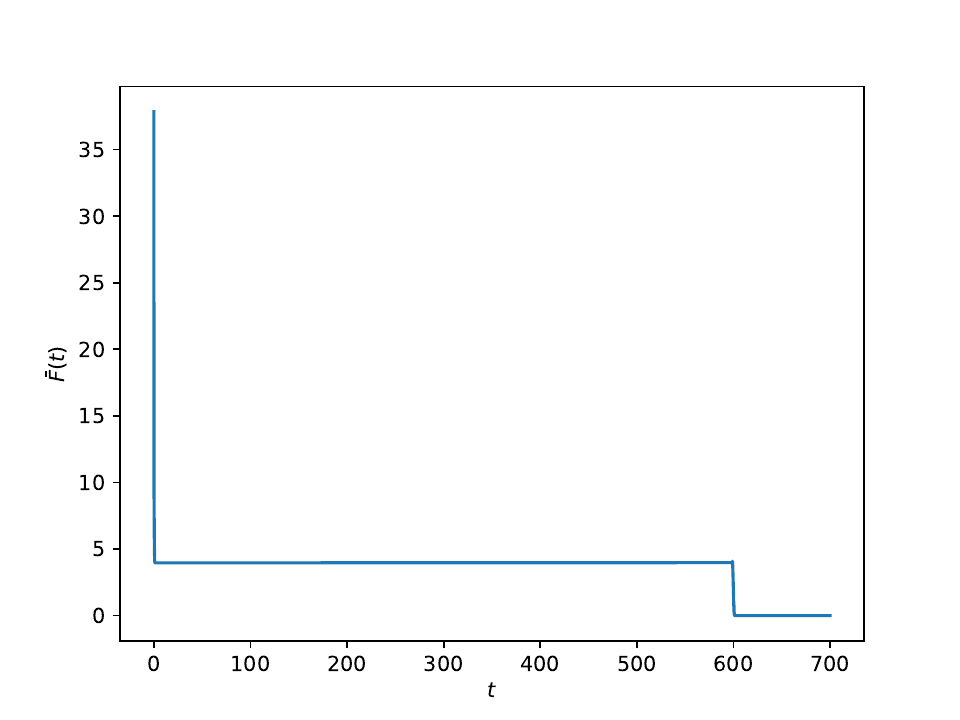}
    \includegraphics[width=0.3\linewidth]{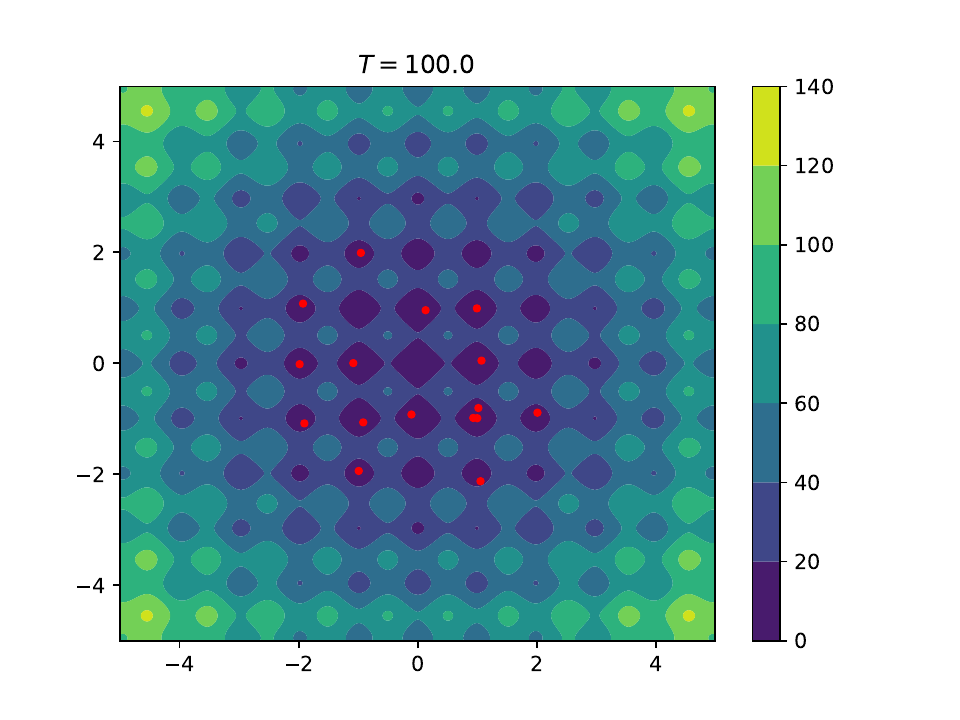}
    \includegraphics[width=0.3\linewidth]{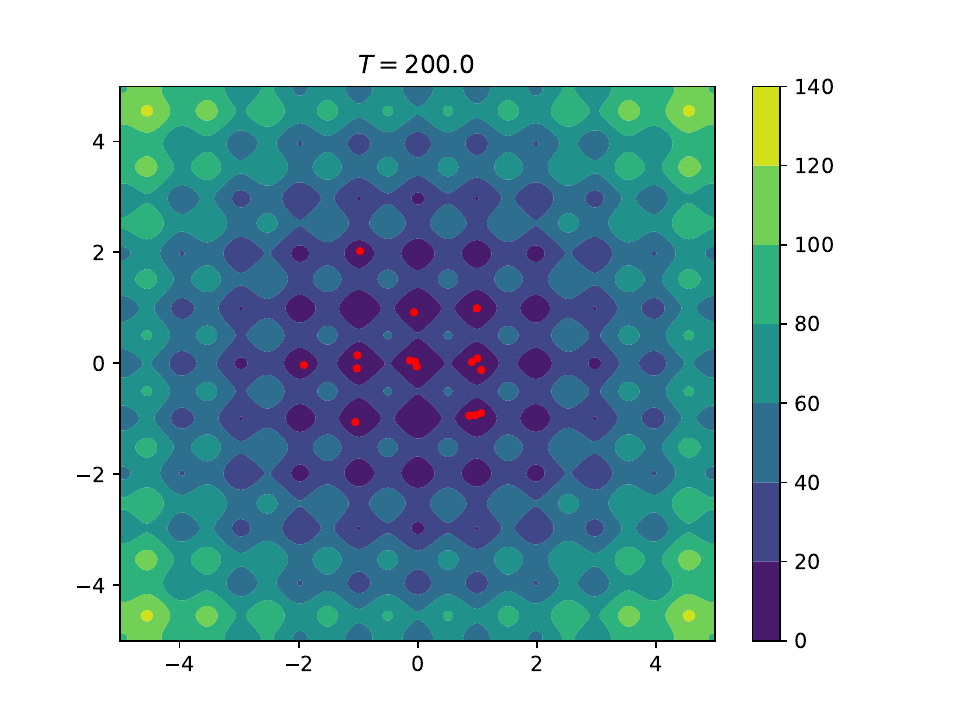}
    \caption{SSA applied to a Rastrigin function: $\f(x,y) =  20 + 2x^2 + 2y^2 - \cos(2\pi x) - \cos(2\pi y)$. The parameters used are $N=16$, $\alphalam=2$ and $\betamu=1/8$. The global minimum is $\f(x_\ast,y_\ast) = 0$ at $(x_\ast,y_\ast) = (0,0)$. Top Left: landscape of he function. Lower Left: long time behavior. Center and Right: swarm movement in time. Similar to the 1D Rastrigin case (see Figure~\ref{fig:long-time_rastrigin}), the provisional minimum quickly decreases to a plateau, and stays there for a long time frame before descending again. The major challenge in this example is that the local minimum has very similar value to the global minimum, and often, when found, trick the swarm to saturates.}
    \label{fig:Rastrigin2D}
\end{figure}

\end{document}